\documentclass[a4paper,10pt]{article}
\usepackage{setspace}
\usepackage[utf8]{inputenc}
\usepackage{hyperref}
\usepackage{amssymb,amsmath,amsfonts,amsthm,textcomp,enumerate,mathrsfs,eufrak,version,mathtools}
\usepackage{array}
\usepackage{lipsum}
\usepackage{vmargin}
\usepackage{titlesec}
\usepackage[titletoc,toc,title]{appendix}
\titleformat{\subsubsection}{\Large\scshape\raggedright}{}{0em}{}[\titlerule]
\usepackage[final]{pdfpages}
\usepackage[T1]{fontenc}
\usepackage{color}
\usepackage[all,cmtip]{xy}
\newtheorem{theo}{Theorem}[section]
\newtheorem{prop}[theo]{Proposition}
\newtheorem{lem}[theo]{Lemma}
\newtheorem{cor}[theo]{Corollary}
\newtheorem{deftn}[theo]{Definition}

\title{Characteristic cycles, micro local packets and packets with cohomology}
\author{Nicol\'as Arancibia Robert}
\date{}
\begin{document}
\maketitle
\begin{abstract}
Relying on work of Kashiwara-Schapira 
and Schmid-Vilonen, we describe 
the behaviour of characteristic cycles with respect to the operation of geometric induction, the 
geometric counterpart of taking parabolic or cohomological induction in representation theory. 
By doing this, we are able to describe to some extent 
the characteristic cycle associated to an induced representation, 
in terms of the characteristic cycle of the representation being induced. More precisely, 
under the hypothesis that the infinitesimal character is regular (and dominant),  
we show that the 
characteristic cycle of an induced representation 
 splits in two terms. We describe the first term 
 precisely,  but we are not able to do the same for the second one. What we are able to say,  
is that this second term is supported on the boundary 
of the space generated by the inclusion 
in the flag variety of $G$,
of the flag variety of the Levi subgroup.
As a consequence, we prove that the cohomology packets defined by Adams and Johnson in \cite{Adams-Johnson} are micro-packets, that is to say that 
the cohomological constructions of \cite{Adams-Johnson} are  particular cases of the 
sheaf-theoretic ones in \cite{ABV}. 
\end{abstract}
\tableofcontents
\section{Introduction}
Let $G$ be a connected reductive algebraic group defined over a number field $F$.
In \cite{Arthur84} and \cite{Arthur89}, Arthur 
gives a conjectural description of the  discrete  spectrum of $G$  by introducing at each place $v$ 
of $F$ a set of parameters $\Psi_v(G)$, that should parameterize all the unitary representations of $G(F_v)$ that are of interest for global applications. 
More precisely, Arthur conjectured that attached to every parameter $\psi_v\in \Psi_v(G)$
we should have a finite set $\Pi_{\psi_v}(G(F_v))$, called an $A$-packet, 
of irreducible representations of $G(F_v)$, 
uniquely
characterized 
 by the following properties:
\begin{itemize}
\item $\Pi_{\psi_v}(G(F_v))$ consists of unitary representations.
\item The parameter $\psi_v$ corresponds to a unique $L$-parameter $\varphi_{\psi_v}$ and $\Pi_{\psi_v}(G(F_v))$ contains the $L$-packet associated to 
$\varphi_{\psi_v}$.
\item $\Pi_{\psi_v}(G(F_v))$ is the support of a stable virtual character distribution on $G(F_v)$.
\item $\Pi_{\psi_v}(G(F_v))$ verifies the ordinary and twisted spectral transfer identities predicted by the theory of endoscopy.
\end{itemize}
Furthermore, any representation occurring in the discrete spectrum of square integrable automorphic representations of $G$, should be a restricted product over all places of representations in the corresponding $A$-packets.


In the case when $G$ is a real reductive algebraic group, Adams, Barbasch and Vogan proposed in \cite{ABV} a candidate for an $A$-packet, proving in the process 
all of the predicted properties with the exception of the twisted endoscopic identity and unitarity. 
The packets in \cite{ABV}, to which we refer from now on as micro-packets or ABV-packets,
are defined by means of sophisticated geometrical methods.
As explained in the introduction of \cite{ABV}, the 
inspiration behind their construction 
comes from the combination of ideas of Langlands and Shelstad (concerning dual groups and endoscopy) with those of Kazhdan and
Lusztig (concerning the fine structure of irreducible representations), to describe the representations of $G(\mathbb{R})$ in terms of an appropriate geometry on an $L$-group. 
The geometric methods are 
remarkable, but they have the constraint
of being extremely difficult to calculate in practice. 
Without considering some exceptions, like the case of ABV-packets attached to tempered Arthur parameters (see Section 4.1 below) or to principal unipotent Arthur parameters (See Chapter 27 \cite{ABV} and Section 4.2 below),  
we cannot identify the members of an ABV-packet in any known classification 
(in the Langlands classification for example). 
The difficulty 
comes from the central role played by characteristic cycles in their construction. 
These cycles are geometric invariants 
that can be understood as a way to measure how far a constructible
sheaf is from being a local system.

In the present article, relying on work of Kashiwara-Schapira 
and Schmid-Vilonen, we describe 
the behaviour of characteristic cycles with respect to the operation of geometric induction, the 
geometric counterpart of taking parabolic or cohomological induction in representation theory. 
By doing this, we are able to describe to some extent
the characteristic cycle associated to an induced representation, 
in terms of the 
characteristic cycle of the representation being induced (see Theorem \ref{prop:ccLG} below). More precisely, 
under the hypothesis that the infinitesimal character is regular (and dominant),  
we show that the 
characteristic cycle of an induced representation 
 splits in two terms. We describe the first term 
 precisely,  but we are not able to do the same for the second one. What we are able to say,  
is that this second term is supported on the boundary 
of the space generated by the inclusion
in the flag variety of $G$,
of the flag variety of the Levi subgroup.
This description is enough for our aplications.
Before continuing with a more detailed description 
on the behaviour of characteristic cycles under induction,
 let us mention some consequences of it.
 
As a first application we have
the proof that the cohomology packets defined by Adams and Johnson in \cite{Adams-Johnson} are micro-packets.
In more detail, Adams and Johnson proposed in \cite{Adams-Johnson} a candidate for an $A$-packet
by attaching to any member in a particular family of Arthur parameters (see points (AJ1), (AJ2) and (AJ3) 
Section 4.3),
a packet consisting of 
representations cohomologically induced from unitary characters. 
Now, 
from the behaviour of characteristic cycles under induction, 
the description of the ABV-packets corresponding to any Arthur parameter 
in the family studied in \cite{Adams-Johnson},  
reduces to the description of ABV-packets corresponding to essentially principal unipotent Arthur parameters, which are well understood (see Section 4.2).
From this reduction
we prove in Theorem \ref{theo:ABV-AJ}, that the cohomological constructions of \cite{Adams-Johnson} are  particular cases of the ones in \cite{ABV}. It is important to mention
 that the equality between Adams-Johnson and ABV-packets is known to experts, but to my knowledge no proof of it can be found in the literature. Let us also say that from this equality  
and the proof in \cite{AMR} that for classical groups
the packets defined in \cite{Adams-Johnson} are $A$-packets (\cite{Arthur}),
we conclude that in the
framework of \cite{Adams-Johnson} and for classical groups, the three constructions of $A$-packets coincide.

A second application 
is related 
to the proof that for classical groups the $A$-packets introduced in \cite{Arthur} are ABV-packets (work in progress with Jeffrey Adams and Paul Mezo).
An important step in the proof, is the description of the ABV-packets for the general linear group.
The understanding of the behaviour of characteristic cycles under induction 
is crucial in the proof that for the general linear group,
ABV-packets are Langlands packets, that is, they consist of a single representation.


Let us give now a quick overview on  
how geometric induction affects characteristic cycles.
We begin by introducing the geometric induction functor. 
Suppose $G$ is a connected reductive complex algebraic group 
with Lie algebra $\mathfrak{g}$, and let $\sigma$ be a real form of $G$  
 (see Equation (2.1)(a-c) \cite{ABV}).
Let $\theta$ be the Cartan involution attached
to $\sigma$ as in 
Equation (5d)-(5g) \cite{AVParameters}, and write 
$K=G^{\theta}$ for the group of its fixed points. 
Suppose $Q$ is a parabolic subgroup of $G$ 
whose Lie algebra $\mathfrak{q}$ is \textbf{germane} with respect to the pair $(G,K)$ (i.e. $\mathfrak{q}$ satisfies the equivalent conditions of Proposition 4.74 \cite{Knapp-Vogan}). 
 Let $Q=LU$ be the 
Levi decomposition of $Q$. By Proposition 
4.74 \cite{Knapp-Vogan} and Proposition 4.78 \cite{Knapp-Vogan}
$L$ is a reductive subgroup stable under $\sigma$, and the restriction of $\theta$ to $L$ defines a Cartan involution for $L$. We are going to be interested in two extreme cases of the notion of germane. One is that the parabolic subgroup $Q$ is real, that is stable under $\sigma$. The other case is that 
$Q$ is stable under the Cartan involution $\theta$. 

Denote by $X_G$ the flag variety of $G$,  
 and consider the fibration $X_G\rightarrow G/Q$. Its fiber over $Q$ 
can be identified with the flag variety $X_L$ of $L$. We denote the inclusion of that fiber in 
$X_G$ by 
\begin{align}\label{eq:mapvarieties}
\iota:X_L&\longrightarrow X_G.
\end{align}
Since $K$ is the set of fixed point of the Cartan involution $\theta$, by Propostion 6.16 \cite{ABV}, $K$
acts over the flag variety 
$X_G$ of $G$ with finitely many orbits. Similarly, $K\cap L=L^{\theta}$, hence $K\cap L$ acts with finitely many orbits over $X_L$.

Let $D_{c}^{b}(X_G,K)$ be the 
$K$-equivariant bounded derived category of sheaves of complex vector spaces on $X_G$ having cohomology sheaves constructible with respect to an algebraic stratification of $X_G$. Living inside this category we have the subcategory 
$\mathcal{P}(X_G,K)$ of $K$-equivariant perverse sheaves on $X_G$.
Set $\mathcal{D}_{X_G}$ to be the sheaf of algebraic differential operators on $X_G$. The Riemann-Hilbert correspondence (see Theorem 7.2.1\cite{Hotta} and Theorem 7.2.5 \cite{Hotta}) defines an equivalence of categories between $\mathcal{P}(X_G,K)$ and the category  
$\mathcal{D}(X_G,K)$ 
of $K$-equivariant 
$\mathcal{D}_{X_{G}}$-modules on $X_G$. 
Now, write $\mathcal{M}(\mathfrak{g},K)$ for the category of $(\mathfrak{g}, K)$-modules of $G$, and $\mathcal{M}(\mathfrak{g},K, I_{X_G})$ for the subcategory of $(\mathfrak{g}, K)$-modules of $G$ annihilated by the kernel $I_{X_G}$ of the operator representation (see equations (\ref{eq:operatorrepresentation}) and (\ref{eq:keroperatorrepresentation}) below).
The categories $\mathcal{M}(\mathfrak{g},K, I_{X_G})$ and $\mathcal{D}(X_G,K)$ are identified through the Beilinson-Bernstein correspondence (\cite{BB}),
 and 
composing this functor with the Riemann-Hilbert correspondence we obtain the equivalence of categories 
\begin{align*}
\Phi_{X_G}:\mathcal{M}(\mathfrak{g},K,I_{X_G}) \xrightarrow{\sim} \mathcal{P}(X_G,K).
\end{align*}
Consequently there is a bijection between the corresponding Grothendieck groups 
$K\mathcal{M}(\mathfrak{g},K,I_{X_G})$ and $K\mathcal{P}(X_G,K)$.
Similarly, by denoting 
$\mathfrak{l}$ the Lie algebra of $L$ and $K_L=K\cap L$, we define 
$\mathcal{P}(X_L,K_L)$, 
$\mathcal{M}(\mathfrak{l},K_L,I_{X_L})$ and $\Phi_{X_L}:\mathcal{M}(\mathfrak{l},K_L,I_{X_L}) \xrightarrow{\sim} \mathcal{P}(X_L,K_L)$.
Now, let 
\begin{align}\label{eq:cohomologicalinductionintro}
\left({}^{u}\mathscr{R}_{(\mathfrak{q},K_{L})}^{(\mathfrak{g},K)}\right)^{j}:\mathcal{M}(\mathfrak{l},K_L)\rightarrow \mathcal{M}(\mathfrak{g},K) 
\end{align}
be the induction functor introduced in Equations
(11.71c)-(11.71d) \cite{Knapp-Vogan}, and define 
\begin{align*}
\mathscr{R}_{(\mathfrak{q},K_{L})}^{(\mathfrak{g},K)}:&\mathcal{M}(\mathfrak{l},K_L)\rightarrow \mathcal{M}(\mathfrak{g},K)\\
&V\mapsto 
\left\{
\begin{array}{ll}
\left({}^{u}\mathscr{R}_{(\mathfrak{q},K_{L})}^{(\mathfrak{g},K)}\right)^{\text{dim}(\mathfrak{u}\cap \mathfrak{k})}(V\otimes(\bigwedge^{\text{top}}\mathfrak{u}))& \text{ if }Q \text{ is }\theta\text{-stable}\\
\left({}^{u}\mathscr{R}_{(\mathfrak{q},K_{L})}^{(\mathfrak{g},K)}\right)^{0}(V\otimes\mathbb{C}_{\rho})& \text{if }Q \text{ is real}.
\end{array} 
\right.
\end{align*}
When $Q$ is $\theta$-stable $\mathscr{R}_{(\mathfrak{q},K_{L})}^{(\mathfrak{g},K)}(V)$ corresponds to the Vogan-Zuckerman cohomological induction functor. When $Q$ is real, by Proposition 11.57 \cite{Knapp-Vogan},  
 $\mathscr{R}_{(\mathfrak{q},K_{L})}^{(\mathfrak{g},K)}(V)$ is the   
underlying $(\mathfrak{g},K)$-module of the parabolically induced representation $\text{Ind}_{Q(\sigma,\mathbb{R})}^{G(\sigma,\mathbb{R})}(\pi)$, with $\pi$ the admissible representation of $L(\sigma,\mathbb{R})$ with underlying $(\mathfrak{l},K_{L})$-module $V$. 

The induction functor in representation theory has 
 a geometric analogue 
 $$I_{L}^{G}:D_{c}^{b}(X_L,K_L)\rightarrow D_{c}^{b}(X_G,K)$$
defined through the Bernstein induction functor (see \cite{ViMir}).
The Bernstein induction functor 
$$\Gamma_{L}^{G}:D_{c}^{b}(X_G,K_L)\rightarrow D_{c}^{b}(X_G,K)$$
is 
the right adjoint of the forgetful functor $\text{Forget}_{K}^{K_L}:D_{c}^{b}(X_G,K)\rightarrow D_{c}^{b}(X_G,K_L)$. The geometric induction functor 
 is defined then  
as the composition
$$I_{L}^{G}=\Gamma_{L}^{G}\circ R\iota_{\ast}.$$ 
It satisfies the identity
$$I_{L}^{G}\circ \Phi_{X_L}=\Phi_{X_G}\circ \mathscr{R}_{(\mathfrak{q},K_{L})}^{(\mathfrak{g},K)},$$ 
which induces the following commutative diagram between the corresponding Grothendieck groups 
\begin{align}\label{eq:diagramintroduction}
  \xymatrix{
    K\mathcal{M}(\mathfrak{g},K,I_{X_{G}}) \ar[r]^{\Phi_{X_G}} & K\mathcal{P}(X_G,K)\\
    K\mathcal{M}(\mathfrak{l},K_{L},I_{X_{L}}) \ar[u]^{\mathscr{R}_{(\mathfrak{q},K_{L})}^{(\mathfrak{g},K)}}\ar[r]^{\Phi_{X_L}} & K\mathcal{P}(X_L,K_{L})\ar[u]^{I_{L}^{G}}. }
\end{align}

Now that the geometric induction functor has been introduced, we  
turn to the description of 
how geometric induction affects characteristic cycles. 
The characteristic cycle of a perverse sheaf can be constructed  
through Morse theory, or via the Riemann-Hilbert correspondence as the characteristic cycle of the 
associated $\mathcal{D}$-module. In any case, the characteristic cycle 
can be seen as a map $CC:K\mathcal{P}(X_G,K)\rightarrow \mathscr{L}(X_G,K)$
from $K\mathcal{P}(X_G,K)$ to the set of formal sums
$$\mathscr{L}(X_G,K)=\left\{\sum_{{K-\mathrm{orbits }~S~\mathrm{in}~X_G}}m_S[\overline{T^{\ast}_{S}X_G}]:m_S\in\mathbb{Z}\right\}.$$

Schmid and Vilonen, combining the work of Kashiwara-Schapira on proper direct images, and  
their own work on direct open embeddings, 
describe in \cite{SV} the effect on characteristic cycles 
of taking direct images by an arbitrary algebraic morphism. This, applied to the Bernstein induction functor, 
leads to a map $\left(I_{L}^{G}\right)_{\ast}:\mathscr{L}(X_L,K_L)\rightarrow\mathscr{L}(X_G,K)$ that extends (\ref{eq:diagramintroduction}) into the commutative diagram:  
\begin{align*}
  \xymatrix{
   K\mathcal{M}(\mathfrak{g},K,I_{X_{G}}) \ar[r]^{\Phi_{X_G}} & K\mathcal{P}(X_{G},K)\ar[r]^{CC} & \mathscr{L}(X_G,K)\\
    K\mathcal{M}(\mathfrak{l},K_{L},I_{X_{L}}) \ar[u]^{\mathscr{R}_{(\mathfrak{q},K_{L})}^{(\mathfrak{g},K)}}\ar[r]^{\Phi_{X_L}} & K\mathcal{P}(X_{L},K_{L})\ar[u]^{I_{L}^{G}}\ar[r]^{CC} & \mathscr{L}(X_{L},K_L).\ar[u]^{\left(I_{L}^{G}\right)_{\ast}} 
     }
\end{align*}
Consequently, for every $K_L$-equivariant perverse sheaf $\mathcal{F}_L$ on $X_{L}$ we have
$$CC(I_{L}^{G}\mathcal{F}_L)=\left(I_{L}^{G}\right)_{\ast}CC(\mathcal{F}_L).$$
The description of the image of $\left(I_{L}^{G}\right)_{\ast}$ 
is given in Theorem \ref{prop:ccLG} through a formula for 
$CC(I_{L}^{G}\mathcal{F}_L)$ 
in terms of the characteristic cycle of $\mathcal{F}_L$. 
More explicitly, writing 
$$CC(\mathcal{F}_{L})=\sum_{K_{L}-\mathrm{orbits }~S_L~\mathrm{in}~X_{L}} m_{S_L}[\overline{T_{S_L}^{\ast}X_{L}}]$$
for the characteristic cycle of $\mathcal{F}_{L}$, 
we prove 
that the image of $\mathcal{F}_L$ under 
$\left(I_{L}^{G}\right)_{\ast}$ is equal to
\begin{align}\label{eq:dernierintro}
CC(\Gamma_{L}^{G}\mathcal{F}_L)&=\left(I_{L}^{G}\right)_{\ast}CC(\mathcal{F}_L)\\
&=\sum_{K_L-\mathrm{orbits }~S_L~\mathrm{in}~X_L} m_{S_{L}}[\overline{T_{K\cdot\iota(S_L)}^{\ast}X_G}]+
\sum_{\substack{K-orbits~S\text{~in~}\\
(K\cdot \iota(X_{L}))^c~\cap~ \partial(K\cdot \iota(X_{L})) }} m_{S}[\overline{T_{S}^{\ast}X_G}].
\end{align}
%
From this equality we are able to deduce that 
the cohomology packets 
introduced by Adams and Johnson in 
\cite{Adams-Johnson}, are  examples of micro-packets.
The comparison between both types of packets is done in Section 4 where we also describe 
the family of Arthur parameters considered in the work of Adams and Johnson and
give a description of the Adams-Johnson packets. 

We continue by outlining the contents of the paper.
 In Section 2 we introduce the geometric induction functor as a geometric counterpart of the 
induction functor in representation theory. We review the work of Kashiwara-Schapira \cite{Kashiwara-Schapira} on proper
direct images and the work of Schmid-Vilonen \cite{SV} on open direct images. 
We end the section by explaining how characteristic cycles behave under geometric induction. 

Section 3 is devoted to presenting the objects of \cite{ABV} required 
to define the micro-packets.
We recall the notion of real forms and introduce the concepts of extended group and representation of a strong real form. We define the dual objects that are going to parameterize the representations 
and review the Local Langlands Correspondence as stated in \cite{ABV}. 

We start Section 4 by introducing micro-packets. Then we describe the micro-packets corresponding to 
tempered parameters (Section 4.1) and to essentially principal unipotent Arthur parameters (Section 4.2).
We end the section by studying the case of cohomologically induced packets or Adams-Johnson packets. Finally, relying on the work done in Section 2 we prove that Adams-Johnson packets are micro-packets (Section 4.3). 

It is in sections 4.2 and 4.3, 
and in the study of Section 2 of the geometric induction functor,  
that most of the original work of the present article is done.\\

To end with this introduction, let us mention that in \cite{Trappa}
(see Proposition 2.4 \cite{Trappa} and Corollary 2.5 \cite{Trappa}) the authors
obtain a similar result to the one of Theorem \ref{prop:ccLG} (see Equation (\ref{eq:dernierintro}) above) by using the definition of characteristic cycles in terms of normal slices (see Chapter II.6.A \cite{GM}).\\ 

\textbf{Acknowledgements:} The author wishes to thank Paul Mezo for useful discussions and 
enlightening remarks during the preparation of this document. The author would
also like to thank the referee for a thorough reading of the manuscript, 
and for their many helpful comments.

\section{Geometric induction}

In this section we recall Bernstein's induction functor
and use it to define a geometric analogue of the induction functor in representation theory. 
Once the geometric induction functor has been introduced, 
we turn to explain how it affects characteristic cycles. By doing this, we are able to describe the characteristic cycle associated to an induced representation (through the Beilinson-Bernstein Correspondence), in terms of the characteristic cycle of the representation being induced. As we will see in Section 4.3, this will have as a consequence the possibility to reduce in some particular
cases the computation of the micro-packets associated to a group, to the computation of the micro-packets associated to a Levi subgroup.\\

Let us begin by introducing the geometric objects required to the definition of characteristic cycles and
the geometric induction functor. These 
objects are also going to be required in Section 4 
for the definition of the micro-packets.\\ 

Suppose $X$ is a smooth complex algebraic variety 
on which an algebraic group $H$ acts with finitely many orbits. Define (see Appendix B \cite{ViMir} and Definition 7.7 \cite{ABV})\\

$\bullet$ $D_{c}^{b}(X)$ to be the bounded derived category of sheaves of complex vector spaces on $X$ having cohomology sheaves constructible with respect to an algebraic stratification 
 on $X$.\\

$\bullet$ $D_{c}^{b}(X,H)$ to be the subcategory of $D_{c}^{b}(X)$ consisting of 
$H$-equivariant sheaves of complex vector spaces on $X$ having cohomology sheaves constructible with respect to the algebraic stratification 
defined by the 
$H$-orbits
on $X$.\\
~\\
Living inside this 
last category we have the category of $H$-equivariant perverse sheaves on $X$ (see Definition 2.1.2 \cite{BBD}). We write\\

$\bullet$ $\mathcal{P}(X,H)$ for be the category of $H$-equivariant perverse sheaves on $X$.\\
~\\
Next, set $\mathcal{D}_{X}$ to be the sheaf of algebraic differential operators on $X$ and define:\\

$\bullet$ $\mathcal{D}(X,H)$ to be the category of $H$-equivariant coherent sheaves of 
$\mathcal{D}_{X}$-modules on $X$.\\

$\bullet$ ${D}^{b}(\mathcal{D}_X,H)$ to be the $H$-equivariant bounded derived category of sheaves of $\mathcal{D}_{X}$-modules on $X$ having coherent cohomology sheaves.\\
~\\
The categories $\mathcal{P}(X,H)$ and $\mathcal{D}(X,H)$
are abelian, and every object has finite length. To 
each of them corresponds a Grothendieck group that we denote respectively by 
\begin{equation}\label{eq:ggroups}
K\mathcal{P}(X,H)\quad \text{and}\quad K\mathcal{D}(X,H).
\end{equation}
The four previous categories 
 are related through the Riemann-Hilbert correspondence.
\begin{theo}[Riemann-Hilbert Correspondence, see Theorem 7.2.1 \cite{Hotta}, Theorem 7.2.5 \cite{Hotta}
and Theorem 7.9 \cite{ABV}]\label{theo:rhcorrespondence}
The de Rham functor induces an equivalence of categories
\begin{align*}
DR:{D}^{b}(\mathcal{D}_X,H)\rightarrow D_{c}^{b}(X,H)
\end{align*}
such that if we restrict $DR$ to the full subcategory $\mathcal{D}(X,H)$ of 
${D}^{b}(\mathcal{D}_X,H)$ we obtain an equivalence of categories
\begin{align*}
DR:\mathcal{D}(X,H)\rightarrow\mathcal{P}(X,H).
\end{align*}
This induces an isomorphism of Grothendieck groups
\begin{align*}
DR:K\mathcal{D}(X,H)\rightarrow K\mathcal{P}(X,H).
\end{align*}
\end{theo}
We use the previous isomorphism to identify the Grothendieck groups of (\ref{eq:ggroups})
writing simply $K(X,H)$ instead of $K\mathcal{P}(X,H)$ and $K\mathcal{D}(X,H).$\\

In this paper we are mostly interested in the case of $X$ being a flag variety. 
More precisely, let $G$ be a connected reductive complex algebraic group
with Lie algebra $\mathfrak{g}$. Fix a real form $\sigma$ of $G$ (see Equation (2.1)(a-c) \cite{ABV}) and 
write $\theta$ for the Cartan involution attached
to $\sigma$ as in 
Equation (5d)-(5g) \cite{AVParameters}. 
Let $K=G^{\theta}$ be the group of fixed points of $\theta$. 
 The flag variety $X_G$ of $G$ is defined as the set of all Borel subgroups of $G$ (or equivalently as the set of all Borel subalgebras of $\mathfrak{g}$). The group $G$ acts on $X_G$ by conjugation, 
this action is transitive and if we restrict it to 
$K$, the number of $K$-orbits is finite (See the proof of Propostion 6.16 \cite{ABV}). Most of the work will be done
in the framework of the categories $D_{c}^{b}(X_{G},K)$, $D^{b}(\mathcal{D}_{X_G},K)$ and the subcategories of $K$-equivariant perverse sheaves and $K$-equivariant $\mathcal{D}_{X_G}$-modules on $X_{G}$.\\

To compare the geometric induction functor with induction in representation theory, 
we need to relate the categories just introduced with the category of $(\mathfrak{g},K)$-modules of $G$.
This is done through the Beilinson-Bernstein correspondence, that we describe below 
 following Chapter 8 \cite{ABV}.   

 Define
\begin{align*}
\mathcal{M}(\mathfrak{g},K) \text{ to be the category of }(\mathfrak{g}, K)\text{-modules of }G. 
\end{align*}
We have an equivalence of categories between $\mathcal{M}(\mathfrak{g},K)$ 
and the category of (infinitesimal equivalence classes of) admissible representations of $G(\mathbb{R},\sigma)$. 
Using this equivalence, we shall blur the distinction between these two categories by referring to their objects indiscriminately as representations of $G(\mathbb{R},\sigma)$.

As before, let $X_{G}$ be the flag variety of $G$ and write
$\mathcal{D}_{X_{G}}$ for the sheaf of algebraic differential operators on $X_{G}$. We define
$$D_{X_{G}}=\Gamma \mathcal{D}_{X_{G}}$$
to be the algebra of global sections of $\mathcal{D}_{X_{G}}$. We recall that every
element of $\mathfrak{g}$ defines a global vector field on $X_{G}$
and that this identification extends to an algebra homomorphism
\begin{align}\label{eq:operatorrepresentation}
\psi_{X_{G}}:U(\mathfrak{g})\longrightarrow D_{X_{G}}
\end{align}
called the \textbf{operator representation} of $U(\mathfrak{g})$.
The kernel of $\psi_{X_{G}}$ is a two-sided ideal denoted by
\begin{align}\label{eq:keroperatorrepresentation}
I_{X_{G}}=\text{ker}\psi_{X_{G}}.
\end{align}
Now, if $\mathcal{M}$ is any sheaf of $\mathcal{D}_{X_{G}}$-modules, 
then the vector space $M=\Gamma \mathcal{M}$ obtained by taking global sections
is in a natural way a $D_{X_{G}}$-module and therefore, via $\psi_{X_{G}}$, a module for $U(\mathfrak{g})/I_{X_{G}}$.
The functor sending the $\mathcal{D}_{X_{G}}$-module $\mathcal{M}$ to the  
$U(\mathfrak{g})/I_{X_{G}}$-module $M$ is called the \textbf{global sections functor}.
In the other direction, if $M$ is any module for $U(\mathfrak{g})/I_{X_{G}}$ 
then we may form the tensor product
$$\mathcal{M}=\mathcal{D}_{X_{G}}\otimes_{\psi_{X_{G}}(U(\mathfrak{g})/I_{X_{G}})} M.$$
This is a sheaf of $\mathcal{D}_{X_{G}}$-modules on $X_{G}$. 
The functor sending $M$ to $\mathcal{M}$ is called \textbf{localization}. 
\begin{theo}[Beilinson-Bernstein localization theorem, 
see \cite{BB}, Theorem 3.8 \cite{BoBrI}, Theorem 1.9 \cite{BoBrIII}
and Theorem 8.3 \cite{ABV}]\label{theo:bbcorespondance}
We have:
\begin{enumerate}[i.]
\item The operator representation $\psi_{X_{G}}:U(\mathfrak{g})\longrightarrow D_{X_{G}}$ is surjective.
\item The global sections and localization functors provide an equivalence of categories between
quasicoherent sheaves of $\mathcal{D}_{X_{G}}$-modules on $X_G$ and modules for $U(\mathfrak{g})/I_{X_{G}}$.
\item Let
\begin{align*}
\mathcal{M}(\mathfrak{g},K,I_{X_{G}})\text{ be the category of }(\mathfrak{g}, K)\text{-modules of }G\text{ annihilated by }I_{X_G}. 
\end{align*}
Then the global sections functor and localization functor provide an equivalence of categories between:
$$\mathcal{D}(X_G,K)\quad\text{ and }\quad \mathcal{M}(\mathfrak{g},K,I_{X_{G}}).$$
This induces an isomorphism of Grothendieck groups
$K(X_G,K)\rightarrow K\mathcal{M}(\mathfrak{g},K,I_{X_{G}}).$
\end{enumerate}
\end{theo}
 Suppose $Q$ is a parabolic subgroup of $G$ 
whose Lie algebra $\mathfrak{q}$ is \textbf{germane} with respect to the pair $(G,K)$ (i.e. $\mathfrak{q}$ satisfies the equivalent conditions of Proposition 4.74 \cite{Knapp-Vogan}). 
 Let $Q=LU$ be the 
Levi decomposition of $Q$. By Proposition 
4.74 \cite{Knapp-Vogan} and Proposition 4.78 \cite{Knapp-Vogan}
$L$ is a reductive subgroup stable under $\sigma$, and the restriction of $\theta$ to $L$ defines a Cartan involution for $L$. 
We are going to be interested in two extreme cases for the notion of germane. One is that the parabolic subgroup $Q$ is real, that is stable under $\sigma$. The other case is that 
$Q$ is stable under the Cartan involution $\theta$. From this point forward we suppose that $Q$ is in any of this two cases.

 Consider the fibration $X_G\rightarrow G/Q$. Its fiber over $Q$ 
can be identified with the flag variety $X_L$ of $L$. We denote the inclusion of that fiber in 
$X_G$ by 
\begin{align}\label{eq:mapvarieties}
\iota:X_L&\longrightarrow X_G.
\end{align}
Define $K_{L}=K\cap L$, since $K_{L}=L^{\theta}$, it acts with finitely many orbits over $X_L$. 
Finally write $\mathfrak{l}$ for the Lie algebra
of $L$, $\mathfrak{k}$ for the Lie algebra of $K$, and $\mathfrak{u}$ for the Lie algebra of $U$. 
Let 
\begin{align*}
\left({}^{u}\mathscr{R}_{(\mathfrak{q},K_{L})}^{(\mathfrak{g},K)}\right)^{j}:\mathcal{M}(\mathfrak{l},K_L)\rightarrow \mathcal{M}(\mathfrak{g},K) 
\end{align*}
be the induction functor introduced in Equations
(11.71c)-(11.71d) \cite{Knapp-Vogan}, and define 
\begin{align}\label{eq:cohomologicalinduction}
\mathscr{R}_{(\mathfrak{q},K_{L})}^{(\mathfrak{g},K)}:&~\mathcal{M}(\mathfrak{l},K_L)\rightarrow \mathcal{M}(\mathfrak{g},K)\\
&~V\mapsto 
\left\{
\begin{array}{ll}
\left({}^{u}\mathscr{R}_{(\mathfrak{q},K_{L})}^{(\mathfrak{g},K)}\right)^{\text{dim}(\mathfrak{u}\cap \mathfrak{k})}(V\otimes(\bigwedge^{\text{top}}\mathfrak{u}))& \text{if }Q \text{ is }\theta\text{-stable}\\
\left({}^{u}\mathscr{R}_{(\mathfrak{q},K_{L})}^{(\mathfrak{g},K)}\right)^{0}(V\otimes\mathbb{C}_{\rho})& \text{if }Q \text{ is real}.
\end{array} 
\right.\nonumber
\end{align}
When $Q$ is $\theta$-stable $\mathscr{R}_{(\mathfrak{q},K_{L})}^{(\mathfrak{g},K)}(V)$ corresponds to the Vogan-Zuckerman cohomological induction functor. In the case of $Q$ real, we have from Proposition 11.57 \cite{Knapp-Vogan},  
that $\mathscr{R}_{(\mathfrak{q},K_{L})}^{(\mathfrak{g},K)}(V)$ is the   
underlying $(\mathfrak{g},K)$-module of the parabolically induced representation $\text{Ind}_{Q(\sigma,\mathbb{R})}^{G(\sigma,\mathbb{R})}(\pi)$, with $\pi$ the admissible representation of $L(\sigma,\mathbb{R})$ with underlying $(\mathfrak{l},K_{L})$-module $V$. 
Since we 
are identifying representations with their underlying 
$(\mathfrak{g},K)$-modules, 
we are not going to make any distinction between parabolic induction and 
the functor in (\ref{eq:cohomologicalinduction})
for real $Q$. 

We now begin with the description of the geometric induction functor. The objective is to define a functor
$$I_{L}^{G}:D_{c}^{b}(X_L,K_{L})\rightarrow D_{c}^{b}(X_G,K),$$
that makes the following diagram commutative
\begin{align}\label{eq:cdiagramme1}
  \xymatrix{
    K\mathcal{M}(\mathfrak{g},K,I_{X_{G}}) \ar[r] & K(X_G,K)\\
    K\mathcal{M}(\mathfrak{l},K_{L},I_{X_{L}}) \ar[u]^{\mathscr{R}_{(\mathfrak{l},K_{L})}^{(\mathfrak{g},K)}}\ar[r] & K(X_L,K_{L})\ar[u]^{I_{L}^{G}}, }.
\end{align}
Here the horizontal arrows are given by Theorem \ref{theo:bbcorespondance}. The construction of $I_{L}^{G}$ is based on Bernstein's geometric functor. 
\begin{deftn}[Section 1.1 \cite{ViMir}]\label{deftn:geominduction0}
Suppose $Y$ is a smooth complex algebraic variety 
on which the algebraic group $G$ acts with finitely many orbits.
For any subgroup $H$ of $G$ we define the \textbf{Bernstein induction functor}
$$\Gamma_{H}^{G}:D_{c}^{b}(Y,H)\rightarrow D_{c}^{b}(Y,G)$$
as the right adjoint of the forgetful functor from $D_{c}^{b}(Y,G)$
to $D_{c}^{b}(Y,H)$ (see \cite{Bi} for the proof of its existence). More precisely, 
consider the diagram
\begin{align}\label{eq:diagraminduction1}
\xymatrix{
    G\times Y \ar[r]^\mu \ar[d]_p  & G\times_{H} Y \ar[d]^a \\
    Y  & Y
  }
\end{align}  
given by
$$  \xymatrix{
    (g,y) \ar[r] \ar[d]  & \overline{(g,y)} \ar[d] \\
    y  & g\cdot y
  }$$
where $G\times_{H} Y$ is the quotient of $G\times Y$ by the $H$-action $h\cdot(g,y)=(gh^{-1},hy)$.
From Theorem (A.2) (iii)  \cite{ViMir}, for $\mathcal{F}\in D_{c}^{b}(Y,H)$ there is a unique $\widetilde{\mathcal{F}}\in D_{c}^{b}(G\times_{H}Y,G)$ such that
$p^{\ast}\mathcal{F}=\mu^{\ast}\widetilde{\mathcal{F}}$. Bernstein's induction functor is defined as 
\begin{align*}
\Gamma_{H}^{G}\mathcal{F}=Ra_{\ast}\widetilde{\mathcal{F}},\quad \mathcal{F}\in D_{c}^{b}(Y,B),
\end{align*} 
where $R{a_{\ast}}:D_{c}^{b}(G\times_{H}Y,G)\rightarrow D_{c}^{b}(Y,G)$
is the right derived functor of the direct image functor defined by $a$.
Equivalently, one can also define $\Gamma_{H}^{G}:D_{c}^{b}(Y,H)\rightarrow D_{c}^{b}(Y,G)$
via the diagram
\begin{align}\label{eq:diagraminduction2}
 \xymatrix{
    G\times Y \ar[r]^\nu \ar[d]_b  & G/H\times Y \ar[d]^p \\
    Y  & Y
  }
\end{align}
given by
$$  \xymatrix{
    (g,y) \ar[r] \ar[d]  & {(gH,y)} \ar[d] \\
    g^{-1}\cdot y  &  y
  }$$
Then for $\mathcal{F}\in D_{c}^{b}(Y,H)$ we have
\begin{align*}
\Gamma_{H}^{G}\mathcal{F}=Rp_{\ast}{\mathcal{F}'},
\end{align*}
where $\mathcal{F}'$ is the unique element in $D_{c}^{b}(G/H\times Y,G)$ such that $b^{\ast}\mathcal{F}=\nu^{\ast}{\mathcal{F}'}$ (Theorem (A.2) (iii)  \cite{ViMir}) and $R{p_{\ast}}:D_{c}^{b}(G/{H}\times Y,G)\rightarrow D_{c}^{b}(Y,G)$ is the right derived functor of the direct image functor defined by $p$.
\end{deftn}


We can now give the definition of the geometric induction functor.
\begin{deftn}\label{deftn:geominduction}
Let $$\iota:X_L\longrightarrow X_G$$ be the inclusion defined in Equation (\ref{eq:mapvarieties}) and write
$R{\iota_{\ast}}:D_{c}^{b}(X_L,K_L)\rightarrow D_{c}^{b}(X_G,K_L)$
for the right derived functor of the direct image functor defined by $\iota$.
\textbf{ The geometric induction functor} $I_{L}^{G}:D_{c}^{b}(X_L,K_L)\rightarrow D_{c}^{b}(X_G,K)$ is defined as 
\begin{align}\label{eq:induction}
I_{L}^{G}(\mathcal{F})=\Gamma_{K_L}^{K}(R\iota_{\ast}\mathcal{F})=Ra_{\ast}(\widetilde{R\iota_{\ast}\mathcal{F}}),\quad \mathcal{F}\in D_{c}^{b}(X_L,K_L),
\end{align}
where $\widetilde{R\iota_{\ast}\mathcal{F}}$ is the unique element in 
$D_{c}^{b}(K\times_{K_L}X_G,K)$ satisfying
$p^{\ast}(R{\iota}_{\ast}\mathcal{F})=\mu^{\ast}(\widetilde{R\iota_{\ast}\mathcal{F}})$.
Equivalently, the geometric induction functor can also be defined as 
\begin{align*}
I_{L}^{G}(\mathcal{F})=\Gamma_{K_L}^{K}(R\iota_{\ast}\mathcal{F})=Rp_{\ast}(({R\iota_{\ast}\mathcal{F}})'),\quad \mathcal{F}\in D_{c}^{b}(X_L,K_L),
\end{align*}
where $({R\iota_{\ast}\mathcal{F}})'$ is the unique element in 
$D_{c}^{b}(K/{K_L}\times X_G,K)$ satisfying
$b^{\ast}(R{\iota}_{\ast}\mathcal{F})=\nu^{\ast}(({R\iota_{\ast}\mathcal{F}})')$.
\end{deftn}


Now that the geometric induction functor has been introduced, we explain  
how it affects characteristic cycles.
The characteristic cycle of a perverse sheaf 
 can be defined through Morse theory (see II.6.A \cite{GM}), 
 via the use of vanishing cycles (see Proposition 4.3.20 \cite{Dimca}), or 
through the Riemann-Hilbert correspondence. This latter option is the one that we follow in this 
article. 
 Let us sketch their definition. 
We do this in the same framework as that of the beginning of this section (i.e. $X$ is a smooth complex variety on which an algebraic group $H$ acts with finitely many orbits). 
  
To every $\mathcal{D}_X$-module $M$ 
 it is fairly easy to define 
a variety $\mathrm{Ch}(M)$ (see for example definition (2.1.2) \cite{Hotta})
which is a closed, involutive and conic analytic subvariety
in the complex cotangent bundle $T^{\ast}X$. Let $P\in \mathcal{P}(X,H)$ and 
write $M$ for the corresponding $\mathcal{D}_X$-module under the the Riemann-Hilbert correspondence.
We define 
$\mathrm{Ch}(P):=\mathrm{Ch}(M).$
The variety $\mathrm{Ch}(P)$ is called the \textbf{characteristic
variety} of $P$. 
More generally, let $\mathcal{F}\in D_{c}^{b}(X,H)$ and write $\mathcal{M}$
for the corresponding element in $D_{c}^{b}(\mathcal{D}_X,H)$ under the Riemann-Hilbert correspondence.  Then we define the characteristic variety of $\mathcal{F}$ as
$$\mathrm{Ch}(\mathcal{F}):=\bigcup_i \mathrm{Ch}(H^{i}\mathcal{M}).$$
Now, suppose $\mathcal{F}$ is constructible with respect to the stratification $\mathcal{S}=\{X_{j}\}_{j}$.
Then since $\mathrm{Ch}(\mathcal{F})$ is a closed, conic, analytic Lagrangian subset in the cotangent
bundle $T^{\ast}(X)$, each irreducible component $C$ of 
$\mathrm{Ch}(\mathcal{F})$ is of the form $C=\overline{T_{X_{j}}^{\ast}X}$ for some (unique) $j$ 
(see Remark 4.3.16 (ii) \cite{Dimca}).
When $\mathcal{F}\in D_{c}^{b}(X,H)$, 
the characteristic variety
$\mathrm{Ch}(\mathcal{F})$ is contained in 
$T_{H}^{\ast}(X)$, 
the conormal bundle to the $H$-action.
The
$H$\textbf{-components} (that is, the smallest $H$-invariant union of irreducible components) of 
$T_{H}^{\ast}(X)$ are the closures
$\overline{T_{S}^{\ast}X}$ of conormal bundles of $H$-orbits $S$ in $X$ (see Lemma 19.2 (b)\cite{ABV}), and $\mathrm{Ch}(\mathcal{F})$ is an union of these ${H}${-components}
(see Proposition 19.12 (c) \cite{ABV}).

The characteristic cycle is defined when 
we take 
the multiplicities of the components into account. Let $P$ be a perverse sheaf,
then we define the \textbf{characteristic cycle} $CC(P)$ of $P$ 
as the associated cycle of the characteristic variety $\mathrm{Ch}(P)$ (see Definition 2.2.2 \cite{Hotta}), that is
$$CC(P)=\sum_{C\in I(\mathrm{Ch}(P))}m_{C}(P)[C],$$
where $I(\mathrm{Ch}(P))$ denotes the set of the irreducible components of $\mathrm{Ch}(P)$ and
$m_{C}(P)$ corresponds to the length of the local ring $\mathcal{O}_{C,\mathrm{Ch}(P)}$ (see Section 1.5 \cite{fulton}). When $P\in \mathcal{P}(X,H)$ 
all irreducible component of $\mathrm{Ch}(P)$ in the $H$-component $\overline{T_{S}^{\ast}X}$ have the same length. We denote this length by $m_{S}(P)$. By Proposition 19.12 (c) \cite{ABV} the characteristic cycle of $P$ decomposes as 
$$CC(P)=\sum_{{H-\mathrm{orbits }~S~\mathrm{in}~X}}m_S(P)[\overline{T^{\ast}_{S}X}]$$
More generally, if $\mathcal{F}\in D_{c}^{b}(X)$, then
the characteristic cycle of $\mathcal{F}$ is defined as the formal sum 
$CC(\mathcal{F}):=\sum_{C\in I(\mathrm{Ch}(\mathcal{F}))}m_{C}(\mathcal{F})[C],$
where 
 for each irreducible component $C$ of $\mathrm{Ch}(\mathcal{F})$, $m_{C}(\mathcal{F})$ is an integer defined as in Equation (2.5) \cite{SV}.
 If we write $\mathscr{L}(X)$ to be the set of formal $\mathbb{Z}$-linear combinations 
of irreducible \textbf{analytic Lagrangian conic subvarieties} in the 
cotangent bundle $T^{\ast}X$, 
then 
the characteristic cycle can be seen as a map $CC:D_{c}^{b}(X)\rightarrow \mathscr{L}(X)$. In particular, if we define $\mathscr{L}(X,H)$  
to be the set of formal sums 
\begin{align}\label{eq:cyclespace}
\mathscr{L}(X,H)=\left\{\sum_{{H-\mathrm{orbits }~S~\mathrm{in}~X}}m_S(\mathcal{F})[\overline{T^{\ast}_{S}X}]:m_S\in\mathbb{Z}\right\},
\end{align}
then taking characteristic cycles defines a 
map
$CC:D_{c}^{b}(X,H)\rightarrow \mathscr{L}(X,H),$
and from Theorem 2.2.3 \cite{Hotta} (see also Proposition 19.12(e) \cite{ABV}) 
we obtain a $\mathbb{Z}$-linear map
$$CC:K(X,H)\rightarrow \mathscr{L}(X,H).$$
Let $\mathcal{F}\in D_{c}^{b}(X,H)$. Following the notation of \cite{ABV} (Definition 1.30  \cite{ABV} and Proposition 19.12  \cite{ABV}) for each $H$-orbit $S$ in $X$ we denote  
\begin{align}\label{eq:microlocalmultiplicity}
\chi_{S}^{mic}(\mathcal{F}):=m_{S}(\mathcal{F})
\end{align}
and call $\chi_{S}^{mic}(\mathcal{F})$ the \textbf{microlocal multiplicity along} $S$. Finally, for each formal sum 
$C=\sum_{i}m_{C_{i}}[C_{i}]\in \mathscr{L}(X)$
we define the support $|C|$ of $C$, as 
\begin{align}\label{eq:supportCycle}
|C|=\overline{\bigcup_{m_{C_i}\neq 0}C_{i}}.
\end{align}
Notice that by definition, for every $\mathcal{F}\in \mathcal{P}(X,H)$, 
$|CC(\mathcal{F})|=\mathrm{Ch}(\mathcal{F})$.\\ 

As a first result, let us describe how the functor $\Gamma_{H}^{G}$ affects the characteristic variety. 
\begin{lem}\label{lem:contention}
In the setting of Definition \ref{deftn:geominduction0}, let $\mathcal{F}\in D_{c}^{b}(Y,H)$. Then
\begin{align}\label{eq:contention}
G\cdot \mathrm{Ch}(\mathcal{F}) \subset \mathrm{Ch}(\Gamma_{H}^{G}\mathcal{F})
\subset \overline{G\cdot \mathrm{Ch}(\mathcal{F})} 
\end{align}
\end{lem}
\begin{proof}
The right inclusion in (\ref{eq:contention})
is Lemma (1.2) \cite{ViMir}. 
For the left inclusion we consider the definition of $\Gamma_{H}^{G}$ via Diagram (\ref{eq:diagraminduction2}). Let $\mathcal{F}'\in D_{c}^{b}(G/H\times Y,G)$ be such that 
$b^{\ast}\mathcal{F}=\nu^{\ast}{\mathcal{F}'}$.
We use the description of Ch$(Rp_{\ast}\mathcal{F}')$ 
given in Proposition B2 \cite{ViMir}. 
Let $\overline{G/H}$ be a smooth compactification of $G/H$.
Then $p:G/H\times Y\rightarrow Y$ factors as 
$p:G/H \times Y\xrightarrow{i} \overline{G/H}\times Y\xrightarrow{q} Y$,  
and as explained in the proof of Lemma B2 \cite{ViMir} we have
$$\mathrm{Ch}(Rp_{\ast}\mathcal{F}')=pr(\mathrm{Ch}(R{i}_{\ast}\mathcal{F}')),$$ where 
$pr$ denote the projection $pr:T^{\ast}(G/H\times Y)\rightarrow T^{\ast}Y$.
Since $i:G/H \times Y\rightarrow \overline{G/H}\times Y$ is an open embedding
$$\mathrm{Ch}(\mathcal{F}')\subset \mathrm{Ch}(R{i}_{\ast}\mathcal{F}').$$
Therefore
$$pr( \mathrm{Ch}(\mathcal{F}'))\subset \mathrm{Ch}(Rp_{\ast}\mathcal{F}')=\mathrm{Ch}(\Gamma_{H}^{G}\mathcal{F}).$$
Finally, by Proposition B1 \cite{ViMir}, $pr( \mathrm{Ch}(\mathcal{F}'))=pr( \mathrm{Ch}(b^{\ast}\mathcal{F}))$ and 
from the proof of Lemma (1.2) \cite{ViMir} we obtain $pr( \mathrm{Ch}(b^{\ast}\mathcal{F}))=
G\cdot \mathrm{Ch}(\mathcal{F})$. Equation (\ref{eq:contention}) follows. 
\end{proof}
Let us start now with the study of the relation between characteristic cycles 
and the geometric induction functor.
Our main result is a formula for 
$CC(I_{L}^{G}(\mathcal{F}))$ 
in terms of the characteristic cycle 
of $\mathcal{F}$. 
From Definition \ref{deftn:geominduction} 
it is clear that in order to compute  $CC(I_{L}^{G}(\mathcal{F}))$ 
it will be necessary first to describe the behaviour of 
characteristic cycles under 
taking the direct images 
$R\iota_{\ast}\text{ and }Ra_{\ast},$ and second 
to reduce the characterization of $CC(\widetilde{R\iota_{\ast}\mathcal{F}})$
to the one of $CC({\mathcal{F}})$. The second step will be done in Proposition 
\ref{prop:reduction2} and Corollary \ref{cor:reductionstep2}
below. Another option to compute $CC(I_{L}^{G}(\mathcal{F}))$ is 
to make use of $Rp_{\ast}$ instead of $Ra_{\ast}$, 
and to reduce the characterization of $CC(({R\iota_{\ast}\mathcal{F}})')$
to the one of $CC({\mathcal{F}})$, but since 
Proposition 7.14 \cite{ABV} gives us a description of $CC(\widetilde{R\iota_{\ast}\mathcal{F}})$, in this article we work principally with the definition of $I_{L}^{G}$ via Diagram (\ref{eq:diagraminduction1}).

To deal with the direct images 
$R\iota_{\ast}$
and $Ra_{\ast}$, we describe the pushforward of cycles 
$$\iota_{\ast}:\mathscr{L}(X_{L},K_L)\rightarrow \mathscr{L}(X_{G},K_L)\quad\text{ and }\quad a_{\ast}:\mathscr{L}(K\times_{K_L}X_{G},K)\rightarrow \mathscr{L}(X_{G},K)$$
that make the diagrams
\begin{align}\label{eq:diagramcycle1}
  \xymatrix{
    {D}_{c}^{b}(X_{G},K_L)\ar[r]^{CC} & \mathscr{L}(X_{G},K_L)\\
    {D}_{c}^{b}(X_{L},K_L)\ar[u]^{R\iota_{\ast}}\ar[r]^{CC} & 
    \mathscr{L}(X_{L},K_L)\ar[u]^{\iota_{\ast}} 
     }
    ~\quad\text{ and }
    \xymatrix{
    {D}_{c}^{b}(X_{G},K)\ar[r]^{CC} & \mathscr{L}(X_{G},K)\\
    {D}_{c}^{b}(K\times_{K_L}X_{G},K)\ar[u]^{Ra_{\ast}}\ar[r]^{CC} & \mathscr{L}(K\times_{K_L}X_{G},K)\ar[u]^{a_{\ast}} 
     }
\end{align}
commutative. 
This will be done in a more general context than the one of the functions $\iota:X_L\rightarrow X_G$ 
and $a:K\times_{K_L}X_{G}\rightarrow X_{G}$.
Consider a morphism $F:X\rightarrow Y$ between two 
smooth algebraic varieties. We work initially in the derived categories
and restrict our attention to the  equivariant subcategories when
working with the equivariant maps $a$ and $\iota$. 
 The definition of $F_{\ast}:\mathscr{L}(X)\rightarrow \mathscr{L}(Y)$ and proof 
of the commutativity of the diagram
 \begin{align*}
  \xymatrix{
    D_{c}^{b}(Y)\ar[r]^{CC} & \mathscr{L}(Y)\\
    D_{c}^{b}(X)\ar[u]^{RF_{\ast}}\ar[r]^{CC} & \mathscr{L}(X)\ar[u]^{F_{\ast}} 
     }
\end{align*}
 is due principally to the work of Kashiwara-Schapira \cite{Kashiwara-Schapira} and Schmid-Vilonen \cite{SV}.
We give a short review of their work. 
We begin by noticing that Schmid-Vilonen work in the derived category of sheaves having cohomology sheaves constructible with respect to a \textbf{\textit{semi-algebraic}} stratification. 
We consider then the derived categories introduced at the beginning of this section as subcategories of 
this larger category, and restrict their result to our framework when working with an algebraic map
$F:X\rightarrow Y$. By abuse of notation 
we write, as in the paragraph previous to Equation (\ref{eq:cyclespace}), $\mathscr{L}(X)$ (respectively $\mathscr{L}(Y)$) for the set of \textbf{\textit{semi-algebraic Lagrangian cycles}} in $T^{\ast}X$ (respectively $T^{\ast}Y$).

The definition of $F_{\ast}$ for an arbitrary algebraic map reduces to the case of proper maps and open embeddings. We begin by describing $F_{\ast}$ in the 
case that $F$ is a proper map.
Consider the diagram
\begin{align}\label{eq:dF}
T^{\ast}X\xleftarrow{dF} X\times_{Y}T^{\ast} Y\xrightarrow{\tau} T^{\ast}Y,  
\end{align}
where $\tau:X\times_{Y}T^{\ast} Y\rightarrow T^{\ast}Y$ is the projection on the second coordinate and for all $(x,(F(x),\lambda))\in X\times_Y T^{\ast}Y$ we have $dF(x,(F(x),\lambda))=(x,\lambda\circ dF_{x})$.
The assumption of $F$ being proper implies that $\tau$ is proper. Hence
we can, as in Section 1.4 \cite{fulton}, define a pushforward 
of cycles $\tau_{\ast}:\mathscr{L}(X\times_{Y}T^{\ast}Y)\rightarrow\mathscr{L}(T^{\ast}Y)$.
Moreover, intersection theory (see Section 6.1, 6.2 and 8.1 \cite{fulton}) 
allows us to construct a 
pullback of cycles $dF^{\ast}:\mathscr{L}(T^{\ast}X)\rightarrow\mathscr{L}(X\times_{Y}T^{\ast}Y)$. 
The map $F_{\ast}:\mathscr{L}(X)\rightarrow \mathscr{L}(Y)$ 
is then defined as the composition of these two functions (see Equation 2.16 \cite{SV})
\begin{align}\label{eq:Gysin}
F_{\ast}:=\tau_{\ast}\circ dF^{\ast}.
\end{align}
The following result due to Kashiwara-Schapira  
relates $F_{\ast}$ to the right derived functor $RF_{\ast}:D_{c}^{b}(X)\rightarrow D_{c}^{b}(Y)$
and proves the commutativity of (\ref{eq:diagramcycle1}) in the case of proper maps.
\begin{prop}[Proposition 9.4.2 \cite{Kashiwara-Schapira}]\label{prop:KS}
Let $F:X\rightarrow Y$ be a proper map. Then for all $\mathcal{F}\in D_{c}^{b}(X)$
$$CC(RF_{\ast}\mathcal{F})=F_{\ast}CC(\mathcal{F}).$$
\end{prop}
As explained at the end of Section 3 \cite{SV},   
we can give a more explicit description of $F_{\ast}CC(\mathcal{F})$  
by choosing a transverse family of cycles with limit equal to
$CC(\mathcal{F})$. 
More precisely, suppose $C\in \mathscr{L}(X)$ is transverse to the map 
$dF: X\times_{Y} T^{\ast}Y \rightarrow T^{\ast}X$. Then the geometric inverse image $dF^{-1}(C)$ of $C$ is well-defined as a cycle in $X\times_{Y} T^{\ast}Y$ and we have
$$dF^{\ast}(C)=dF^{-1}(C).$$ 
Consequently, $\tau_{\ast}dF^{-1}(C)$ is a well-defined cycle in $\mathscr{L}(Y)$.
Now, by Lemma 3.26 \cite{SV} we can choose 
for every cycle $C_{0}\in \mathscr{L}(X)$
a family $\{C_{s}\}_{s\in (0,b)}\subset \mathscr{L}(X)$ 
such that the map $dF: X\times_{Y} T^{\ast}Y \rightarrow T^{\ast}X$ is transverse to supp$(C_{s})\subset T^{\ast}X$; for every $s\in (0,b)$ and 
$$C_0=\lim_{s\rightarrow 0}C_{s}.$$
For more details about the construction of this family of cycles, 
see Equation (3.10) and (3.11) \cite{SV}. Equations (3.12-3.16) \cite{SV} provide 
 the notion of limit of a family of cycles. 
Schmid and Vilonen prove
\begin{prop}[Proposition 3.27 of \cite{SV}]\label{prop:familylimit}
Suppose $F:X\rightarrow Y$ is proper. Let $C_{0}\in \mathscr{L}(X)$. Choose a family $\{C_{s}\}_{s\in (0,b)}\subset \mathscr{L}(X)$ 
with limit $C_{0}$ and such that 
$dF: X\times_{Y} T^{\ast}Y \rightarrow T^{\ast}X$ is transverse to the support $|C_{s}|\subset T^{\ast}X$, for every $s\in (0,b)$.
Then
\begin{align}\label{eq:familylimit}
F_{\ast}(C_{0})=\lim_{s\rightarrow 0}\tau_{\ast}dF^{-1}(C_{s}).
\end{align} 
\end{prop}
Having described $F_{\ast}$ when $F$ is proper, we explain now how to define 
$F_{\ast}:\mathscr{L}(X)\rightarrow \mathscr{L}(Y)$
in the case when $F:X\rightarrow Y$ is an open embedding. 
We follow
Chapter 4 \cite{SV}.  
We start by choosing a real valued, semialgebraic $C^{1}$-function $f:X\rightarrow \mathbb{R}$, such that
\begin{enumerate}
\item the boundary $\partial X$ is the zero set of $f$,
\item $f$ is positive on $X$.
\end{enumerate}
For more details on the existence of this map, see Equation (4.1) \cite{SV} and 
Proposition I.4.5 \cite{Shiota}.
Suppose $C\in \mathscr{L}(X)$, and for each $s>0$ define 
$C + sd\log f$ as the cycle of $X$ equal to the image of $C$ under the automorphism of 
$T^{\ast}X$ defined by
\begin{align}\label{eq:open1}
(x,\xi)\mapsto \left(x,\xi+s\frac{df_{x}}{f(x)}\right).
\end{align}
Theorem 4.2 \cite{SV} relates the limit of the family of cycles $\{C+s d\log f\}_{s>0}$
(see page 468 \cite{SV} for the proof of why this family defines a family of cycles)
to the direct image $RF_{\ast}: D_c^{b}(X)\rightarrow D_c^{b}(Y)$.
We state it here as
\begin{prop}\label{prop:openembeding}
Suppose $F:X\rightarrow Y$ is an open embedding. Let $\mathcal{F}\in D_{c}^{b}(X)$, then
$$CC(RF_{\ast}\mathcal{F})=\lim_{s\rightarrow 0} CC(\mathcal{F})+s d\log f.$$
\end{prop}
We notice that, while the family of cycles $\{CC(\mathcal{F})+s d\log f\}_{s>0}$
does not necessarily live in the set of characteristic cycles for the derived category of 
sheaves whose cohomology is constructible with respect to an algebraic stratification, the limit does.
Following Proposition \ref{prop:openembeding} 
we define for each $\mathcal{F}\in D_{c}^{b}(X)$  
\begin{align}\label{eq:openpush}
F_{\ast}CC(\mathcal{F}):=\lim_{s\rightarrow 0} CC(\mathcal{F})+s d\log f.
\end{align}

Finally, to treat the case of an arbitrary algebraic map $F:X\rightarrow Y$  
we follow Chapter 6  \cite{SV}. We embed $X$ as an open subset of a compact algebraic manifold $\overline{X}$, and we factor $F$ into a product of three mappings: the closed embedding 
\begin{align*}
i:X &\rightarrow X\times Y\\
     x&\mapsto(x,F(x))
\end{align*}
which is a simple case of a proper direct image, the open inclusion
$$j : X\times Y \rightarrow \overline{X} \times Y ,$$
and the projection 
$$\overline{p} : \overline{X}\times Y \rightarrow Y$$
which is also a proper map.
Then we can factor the derived functor $RF_{\ast}:D_{c}^{b}(X)\rightarrow D_{c}^{b}(Y)$ into the product 
$$RF_{\ast}=R\overline{p}_{\ast}\circ Rj_{\ast}\circ Ri_{\ast}.$$
From Theorem \ref{prop:KS} and Theorem \ref{prop:openembeding}
for each $\mathcal{F}\in D_{c}^{b}(X)$ we have
\begin{align}\label{eq:generalpushforward}
CC(RF_{\ast}\mathcal{F})&=CC(R\overline{p}_{\ast}\circ Rj_{\ast}\circ Ri_{\ast}(\mathcal{F}))\\
&=\overline{p}_{\ast}CC(Rj_{\ast}\circ Ri_{\ast}(\mathcal{F}))\nonumber\\
&=(\overline{p}_{\ast}\circ j_{\ast})CC(Ri_{\ast}\mathcal{F})\nonumber\\
&=(\overline{p}_{\ast}\circ j_{\ast}\circ i_{\ast})CC(\mathcal{F}).\nonumber
\end{align}
Consequently, we define
\begin{align}\label{eq:svfunctor}
F_{\ast}&:\mathscr{L}(X)\rightarrow \mathscr{L}(Y)\\
F_{\ast}&:=\overline{p}_{\ast}\circ j_{\ast}\circ i_{\ast}.\nonumber
\end{align}
Let us return to our map $a:K\times_{K_L}X_G\rightarrow X_G$ of Definition \ref{deftn:geominduction0}. 
Suppose $\mathcal{F}\in D_{c}^{b}(X_{L},K_L)$
and consider the sheaf 
$I_{L}^{G}\mathcal{F}=Ra_{\ast}(\widetilde{R\iota_{\ast}\mathcal{F}})$
of Definition \ref{deftn:geominduction}.
Our objective is to give a more explicit description of 
$CC(I_{L}^{G}\mathcal{F})$
by reducing its computation to the one for  
the cycle $CC(\mathcal{F})$.
From (\ref{eq:generalpushforward}) 
we can write
\begin{align}\label{eq:reduction}
CC(I_{L}^{G}\mathcal{F})=CC(Ra_{\ast}(\widetilde{R\iota_{\ast}\mathcal{F}}))=a_{\ast}CC
(\widetilde{R\iota_{\ast}\mathcal{F}}).
\end{align}
Thus, to be able to compute
$CC(I_{L}^{G}\mathcal{F})$ 
the first step is to relate the characteristic cycle of $\widetilde{R\iota_{\ast}\mathcal{F}}$
to the characteristic cycle of $\mathcal{F}$. 
This is done in the two following results. The first is a reformulation of Proposition 7.14 \cite{ABV}, Proposition 20.2 \cite{ABV} and Lemma 1.4 \cite{ViMir}.
\begin{prop}\label{prop:reduction2}
Suppose $X$ is a smooth complex algebraic variety 
on which an algebraic group $H$ acts with finitely many orbits.
Suppose $G$ is an algebraic group containing $H$. 
Consider the bundle 
\begin{align*}
Y=G\times_{H}X,
\end{align*}
on which the group G acts by
$$g\cdot (g',x)=(gg',x).$$ 
Then: 
\begin{enumerate}[1.]
\item The inclusion 
\begin{align*}
i:X&\rightarrow Y\\
x& \mapsto \text{ equivalence class of }(e,x)
\end{align*}
induces a bijection 
from $H$-orbits on $X$ 
to $G$-orbits on $Y$. Furthermore, this bijection preserves the inclusion relations 
 of closures.

\item There are natural equivalences of categories,
\begin{align*}
D_{c}^{b}(X,H)\cong D_{c}^{b}(G\times_{H}X,G),\quad 
\mathcal{P}(X,H)\cong \mathcal{P}(G\times_{H}X,G),\quad
\mathcal{D}(X,H)\cong \mathcal{D}(G\times_{H}X,G).
\end{align*}

\item Write $j:\mathfrak{h}\times X\rightarrow \mathfrak{g}\times X$
for the inclusion, and consider the bundle map
$$j\times \mathcal{A}:\mathfrak{h}\times X\rightarrow \mathfrak{g}\times TX,$$
with $\mathcal{A}:\mathfrak{h}\times X\rightarrow TX$ defined by Equation (19.1)(c) \cite{ABV}. Write $Q$
for the quotient bundle: the fiber at $x$ is
$$Q_x=(\mathfrak{g}\times T_x X)/\{(X,\mathcal{A}_{x}(X)):X\in\mathfrak{h}\}.$$
Then the tangent bundle of $Y$ is naturally isomorphic to the bundle on $Y$ induced by $Q$
$$TY\cong G\times_{H}Q.$$
\item The action mapping $\mathcal{A}_{y}:\mathfrak{h}\times Y\rightarrow TY$ (see Equation (19.1) \cite{ABV}) may be computed as
follows. Fix a representative $(g,x)$ for the point $y$ of $G\times_{H}X$, and
an element $Z\in\mathfrak{g}$. Then
$$\mathcal{A}_{y}(Z)=\mathrm{~class~of~}(g,(\mathrm{Ad}(g^{-1})Z,(x,0))).$$
Here $(x,0)$ is the zero element of $T_x X$, so the term paired with $g$ on
the right side represents a class in $Q_x$.
\item The conormal bundle to the $G$-action on the induced bundle  $G\times_{H}X$ is naturally induced by 
the conormal bundle to the $H$-action on $X$:
\begin{align*}
T_{G}^{\ast}(G\times_{H}X)\cong G\times_{H}T_{H}^{\ast}X.
\end{align*}
\item 
Suppose $\mathcal{F}$ and $\widetilde{\mathcal{F}}$ correspond
through any of the equivalences of categories of (2), above.
Then
$$CC(\widetilde{\mathcal{F}})=G\times_{H} CC({\mathcal{F}}).$$
In particular, the microlocal multiplicities (see Equation (\ref{eq:microlocalmultiplicity})) are given by
$$\chi_{G\times_{H}S}^{mic}(\widetilde{\mathcal{F}})=\chi_{S}^{mic}(\mathcal{F}),$$
and
$$CC(\widetilde{\mathcal{F}})=\sum_{H-\mathrm{orbits }~S~\mathrm{in}~X}\chi_{S}^{mic}(\mathcal{F})[\overline{T_{G\times_{H}S}^{\ast}(G\times_{H}X)}].$$ 
\end{enumerate}
\end{prop}
\begin{cor}\label{cor:reductionstep2}
In the setting of Definition \ref{deftn:geominduction},
 let $\mathcal{F}\in D_{c}^{b}(X_{L},K_{L})$. Then
\begin{align}\label{eq:reductionstep2}
CC(\widetilde{R\iota_{\ast}\mathcal{F}})=\sum_{K_L-\mathrm{orbits }~S_L~\mathrm{in}~X_{L}} \chi_{S_L}^{mic}(\mathcal{F})[\overline{T_{K\times_{K_L} S_L}^{\ast}(K\times_{K_L}X_{G}})].
\end{align}
where $\widetilde{R\iota_{\ast}\mathcal{F}}$ is the unique element in 
$D_{c}^{b}(K\times_{K_L}X_G,K)$ satisfying
$p^{\ast}(R{\iota}_{\ast}\mathcal{F})=\mu^{\ast}(\widetilde{R\iota_{\ast}\mathcal{F}})$ (see Definition \ref{deftn:geominduction0}).
\end{cor}
Before giving the proof of the corollary, notice that from the definition of  $\widetilde{R\iota_{\ast}\mathcal{F}}$
and the proof of Proposition \ref{prop:reduction2}(2) (see pages 93-94 \cite{ABV}), 
$R\iota_{\ast}\mathcal{F}$ and $\widetilde{R\iota_{\ast}\mathcal{F}}$ correspond through the 
equivalence of categories of Proposition \ref{prop:reduction2}(2).
\begin{proof}
Suppose $\mathcal{F}\in D_{c}^{b}(X_{L},K_{L})$ and write
$$CC(\mathcal{F})=\sum_{K_{L}-\mathrm{orbits }~S_L~\mathrm{in}~X_{L}} \chi_{S_L}^{mic}(\mathcal{F})[\overline{T_{S}^{\ast}X_{L}}]$$
for the corresponding characteristic cycle.
From Proposition \ref{prop:reduction2}(6), 
the characteristic cycle of
$\widetilde{R\iota_{\ast}\mathcal{F}}$
may be identified as a cycle in $T^{\ast}_{K}(K\times_{K_L}X)$ as 
\begin{align*}
CC(\widetilde{R\iota_{\ast}\mathcal{F}})=K\times_{K_L} CC(R\iota_{\ast}\mathcal{F}).
\end{align*}
From Proposition 6.16 \cite{ABV} and Corollary 6.21 \cite{ABV} the inclusion 
$\iota:X_L \longrightarrow X_G$
is a closed immersion and in consequence proper.
Consider the diagram
\begin{align}
T^{\ast}X_{L}\xleftarrow{d\iota} X_{L}\times_{X_G}T^{\ast}X_G\xrightarrow{\tau} T^{\ast}X_{L}. 
\end{align}
By Proposition \ref{prop:KS} we have
\begin{align*}
CC(R\iota_{\ast}\mathcal{F})&=\iota_{\ast}CC(\mathcal{F})\\
&=\tau_{\ast}\circ d\iota^{\ast}CC(\mathcal{F}),
\end{align*}
where $\tau_{\ast}$ and $d\iota^{\ast}$ are defined as in (\ref{eq:Gysin}). 
Writing $CC(\mathcal{F})$ as the limit of a family of cycles transverse to
$d\iota$, by Proposition \ref{prop:familylimit} we
 conclude that $\iota_{\ast}CC(\mathcal{F})$ is 
the inverse image of $CC(\mathcal{F})$ under 
$$d\iota:T^{\ast}{X_G}|_{\iota(X_{L})}\rightarrow T^{\ast}X_{L}.$$ 
Since the inverse image of the conormal bundle of each $K_L$-orbit 
$S_{L}$ in $X_{L}$ is given by
$$d\iota^{-1}(T_{S_{L}}^{\ast}X_{L})=T_{\iota(S_{L})}^{\ast}X_G$$
we obtain
\begin{align*}
CC(R\iota_{\ast}\mathcal{F})&=\iota_{\ast}CC(\mathcal{F})\\
&=\sum_{K_L-\mathrm{orbits }~S_L~\mathrm{in}~X_L} \chi_{S_{L}}^{mic}(\mathcal{F})[\overline{T_{\iota(S_{L})}^{\ast}X_G}].
\end{align*}
Consequently
\begin{align*}
CC(\widetilde{R\iota_{\ast}\mathcal{F}})&=K\times_{K_L} \sum_{K_L-\mathrm{orbits }~S_L~\mathrm{in}~X_L} \chi_{S_{L}}^{mic}(\mathcal{F})[\overline{T_{\iota(S_{L})}^{\ast}X_G}]\\
&=\sum_{K_L-\mathrm{orbits }~S~\mathrm{in}~X_{G}} \chi_{S}^{mic}({R\iota_{\ast}\mathcal{F}})[\overline{T_{K\times_{K_L} \iota_{\ast}(S)}^{\ast}(K\times_{K_L}X_G})]\\
&=\sum_{K_L-\mathrm{orbits }~S_L~\mathrm{in}~X_L} \chi_{S_{L}}^{mic}(\mathcal{F})[\overline{T_{K\times_{K_L} \iota(S_{L})}^{\ast}(K\times_{K_L}X_G})].
\end{align*}
\end{proof}
The remaining step in the characterization of $CC(I_{L}^{G}\mathcal{F})$
is to describe the effect of taking $a_{\ast}$
on $CC(\widetilde{R\iota_{\ast}\mathcal{F}}).$ 
This is done in the proof of the following theorem. 
\begin{theo}\label{prop:ccLG}
In the setting of Definition \ref{deftn:geominduction},
 let $\mathcal{F}\in D_{c}^{b}(X_{L},K_{L})$. Then
$$CC\left(I_{L}^{G}\mathcal{F}\right)=\sum_{K_L-\mathrm{orbits }~S_L~\mathrm{in}~X_L} \chi_{S_{L}}^{mic}(\mathcal{F})[\overline{T_{K\cdot\iota(S_L)}^{\ast}X_G}]+
\sum_{\substack{K-orbits~S\text{~in~}\\
(K\cdot \iota(X_{L}))^c~\cap~ \partial(K\cdot \iota(X_{L})) }} \chi_{S}^{mic}(I_{L}^{G}\mathcal{F})[\overline{T_{S}^{\ast}X_G}].$$ 
\end{theo}
\begin{proof}
Suppose $\mathcal{F}\in D_{c}^{b}(X_{L},K_{L})$ and write
$$CC(\mathcal{F})=\sum_{K_{L}-\mathrm{orbits }~S_L~\mathrm{in}~X_{L}} \chi_{S_L}^{mic}(\mathcal{F})[\overline{T_{S_L}^{\ast}X_{L}}]$$
for the corresponding characteristic cycle.
 From Equation (\ref{eq:reduction}) and Corollary \ref{cor:reductionstep2}
we can write
\begin{align*}
CC\left(I_{L}^{G}\mathcal{F}\right)&=CC(Ra_{\ast}(\widetilde{R\iota_{\ast}\mathcal{F}}))\\
&=a_{\ast}(CC(\widetilde{R\iota_{\ast}\mathcal{F}}))\\
&=a_{\ast}\left( \sum_{K_L-\mathrm{orbits }~S_L~\mathrm{in}~X_L}  \chi_{S_{L}}^{mic}(\mathcal{F})[\overline{T_{K\times_{K_L} \iota(S_{L})}^{\ast}(K\times_{K_L}X_G})]\right)
\end{align*}
We recall how $a_{\ast}$ is defined. As explained after Proposition \ref{prop:openembeding},
we embed ${K\times_{K_L} X_G}$ as an open subset of a compact algebraic manifold 
$\overline{K\times_{K_L} X_G}$ and factorize $a$ into a product of three maps:
the closed embedding 
\begin{align*}
i : K\times_{K_L} X_G &\rightarrow (K\times_{K_L} X_G)\times X_G,\\
(k,x)&\mapsto ((k,x),a(k,x))=((k,x),kx)
\end{align*} 
the open inclusion
$$j : (K\times_{K_L} X_G)\times X_G \rightarrow \overline{K\times_{K_L} X_G}\times X_G,$$
and the projection 
$$\overline{p} : \overline{K\times_{K_L} X_G}\times X_G\rightarrow X_G.$$
For later use we also denote the restriction of $\overline{p}$ to ${K\times_{K_L} X_G}\times X_G$ as
$$p : {K\times_{K_L} X_G}\times X_G\rightarrow X_G.$$ 
The map $a_{\ast}$ is defined by 
$$a_{\ast}:=\overline{p}_{\ast}\circ j_{\ast}\circ i_{\ast}$$
with $\overline{p}_{\ast}$ and $i_{\ast}$ defined by (\ref{eq:Gysin})
and $j_{\ast}$ as in (\ref{eq:openpush}). 
We have
\begin{align*}
CC\left(I_{L}^{G}\mathcal{F}\right)&=\overline{p}_{\ast}(j_{\ast}(i_{\ast}(CC(\widetilde{R\iota_{\ast}\mathcal{F}}))))\\
&=\overline{p}_{\ast}\left(j_{\ast}\left(i_{\ast}\left( \sum_{K_L-\mathrm{orbits }~S_L~\mathrm{in}~X_L}  \chi_{S_{L}}^{mic}(\mathcal{F})[\overline{T_{K\times_{K_L} \iota(S_{L})}^{\ast}(K\times_{K_L}X_G})]\right)\right)\right).
\end{align*}
We start by describing 
$i_{\ast}CC(\widetilde{R\iota_{\ast}\mathcal{F}})$.
Notice that 
the same argument used to compute $\iota_{\ast}CC(\mathcal{F})$ in Corollary \ref{cor:reductionstep2} allows us to conclude that  
$i_{\ast}CC(\widetilde{R\iota_{\ast}\mathcal{F}})$ is the inverse image of $CC(\widetilde{R\iota_{\ast}\mathcal{F}})$ under
$$di:T^{\ast}((K\times_{K_L}X_G)\times X_G)|_{i(K\times_{K_L}X_G)}\rightarrow T^{\ast} (K\times_{K_L}X_G).$$
Since the map $a$ occurs in the definition of $i$, to compute the inverse image 
of $di$ we need first to compute the inverse image under 
$$da: (K\times_{K_L}X_G)\times_{X_G}T^{\ast}X_G\rightarrow T^{\ast}(K\times_{K_L}X_G)$$
of the conormal bundle of each $K$-orbit in $K\times_{K_L}X_G$. To do that
we begin by noticing that, since by Proposition \ref{prop:reduction2}$(3)$ for each $(k,x)\in K\times_{K_L}X_G$  we have
$$T_{(k,x)}(K\times_{K_L}X_G)\cong (\mathfrak{k}\times T_{x}X_G)/\{(X,\mathcal{A}_{x}(X)):X\in \mathfrak{k}\cap \mathfrak{l}\},$$
with $\mathcal{A}_{x}:\mathfrak{k}\rightarrow T_{x}X_G$ the differential of the action map at $x$
(see Equation 19.1(a) \cite{ABV}),
the map $da_{(k,x)}:T_{(k,x)}(K\times_{K_L}X_G)\rightarrow T_{kx}X_G$
can be represented as
$$da_{(k,x)}=\text{Ad}(k)\circ \mathcal{A}_{x}+\text{Ad}(k).$$
Then from Proposition \ref{prop:reduction2}$(4)$, it is an easy 
exercise 
to verify that 
for each $K$-orbit 
$K\times_{K_L} \iota(S_{L})$ in $K\times_{K_L}X_G$, we have
$$d{a}^{-1}({T_{(k,x),K\times_{K_L} \iota(S_{L})}^{\ast}(K\times_{K_L}X_G)})={T_{kx,K\cdot\iota(S_L)}^{\ast}X_G}.$$
Now, for each $((k,x),x')\in (K\times_{K_L}\times X_G)\times X_G$ we have
$$T^{\ast}_{((k,x),x')}((K\times_{K_L}X_G)\times X_G)\cong 
T^{\ast}_{(k,x)}(K\times_{K_L}X_G)\times T^{\ast}_{x'}X_G,$$
and the image of each element $(\lambda,\lambda')\in T^{\ast}_{(k,x)}(K\times_{K_L}X_G)\times T^{\ast}_{kx}X_G$ under the map $di$ is  
$$(\lambda,\lambda')\mapsto \lambda+\lambda'\circ da_{(k,x)}.$$
Consequently, each element of $T^{\ast}_{(k,x)}(K\times_{K_L}X_G)\times T^{\ast}_{kx}X_G$
in the preimage under $di$ of the annihilator of $T_{(k,x)} K\times_{K_L} \iota(S_{L})$ in 
$T^{\ast}_{(k,x)} (K\times_{K_L}X_G)$
must be a linear combination of elements of the form 
\begin{itemize}
\item $(\lambda,0)\in T^{\ast}_{(k,x)}(K\times_{K_L}X_G)\times T^{\ast}_{kx}X_G$,\text{ with } $\lambda|_{T_{(k,x)} K\times_{K_L} \iota(S_{L})}=0,$
\item $(0,\lambda')\in T^{\ast}_{(k,x)}(K\times_{K_L}X_G)\times T^{\ast}_{kx}X_G$,\text{ with } $\lambda'|_{T_{kx} K\cdot\iota(S_{L})}=0,$
\item $(\lambda,\lambda')\in T^{\ast}_{(k,x)}(K\times_{K_L}X_G)\times T^{\ast}_{kx}X_G$,\text{ with } $(\lambda+\lambda'\circ da_{(k,x)})|_{T_{(k,x)} K\times_{K_L} \iota(S_{L})}=0,$
\end{itemize}
and one may verify that this space is ${T_{((k,x),kx),i(K\times_{K_L}\iota(S_{L}))}^{\ast}((K\times_{K_L}X_G)\times X_G)}$. 
Therefore, 
for the conormal bundle of each $K$-orbit 
$K\times_{K_L} \iota(S_{L})$ in $K\times_{K_L}X_G$ we obtain
\begin{align}\label{eq:before}
di^{-1}({T_{K\times_{K_L} S_{L}}^{\ast}(K\times_{K_L}X_G)})
={T_{i(K\times_{K_L}\iota(S_{L}))}^{\ast}((K\times_{K_L}X_G)\times X_G)}
\end{align}
and so
\begin{align*}
i_{\ast}CC(\widetilde{R\iota_{\ast}\mathcal{F}})
&=i_{\ast}\left( \sum_{K_L-\mathrm{orbits }~S_L~\mathrm{in}~X_L}  \chi_{S_{L}}^{mic}(\mathcal{F})[\overline{T_{K\times_{K_L} \iota(S_{L})}^{\ast}(K\times_{K_L}X_G})]\right)\\
&=\sum_{K_L-\mathrm{orbits }~S_L~\mathrm{in}~X_L} \chi_{S_{L}}^{mic}(\mathcal{F})[\overline{T_{i(K\times_{K_L}\iota(S_{L}))}^{\ast}((K\times_{K_L}X_G)\times X_G)}].
\end{align*}
Next we compute $j_{\ast}(i_{\ast}CC(\widetilde{R\iota_{\ast}\mathcal{F}}))$.
In order to do this we fix, as in the paragraph previous to Equation (\ref{eq:open1}), a 
function
$$f:\overline{K\times_{K_L}X_G}\rightarrow \mathbb{R},$$ 
which takes strictly positive values 
on $K\times_{K_L}X_G$ and vanishes on the boundary of ${K\times_{K_L}X_G}$ in $\overline{K\times_{K_L}X_G}$. 
Define the family of cycles 
$\{i_{\ast}CC(\widetilde{R\iota_{\ast}\mathcal{F}})+sd\log f\}_{s>0}$ as in (\ref{eq:open1}). 
From Proposition \ref{prop:openembeding} we have
\begin{align}\label{eq:intermediate}
j_{\ast}(i_{\ast}CC(\widetilde{R\iota_{\ast}\mathcal{F}}))&=\lim_{s\rightarrow 0}
i_{\ast}CC(\widetilde{R\iota_{\ast}\mathcal{F}})+sd\log f\\
&=\lim_{s\rightarrow 0} \sum_{K_L-\mathrm{orbits }~S_L~\mathrm{in}~X_L} \chi_{S_{L}}^{mic}(\mathcal{F})[\overline{T_{i(K\times_{K_L}\iota(S_{L}))}^{\ast}((K\times_{K_L}X_G)\times X_G)}]+sd\log f.\nonumber
\end{align}
It only remains to compute the image under $\overline{p}_{\ast}$ of $j_{\ast}(i_{\ast}CC(\widetilde{R\iota_{\ast}\mathcal{F}}))$.
Consider the diagrams
\begin{align*}
T^{\ast}(\overline{K\times_{K_L}X_G}\times X_G)\xleftarrow{d\overline{p}}(\overline{K\times_{K_L}X_G}\times X_G)\times_{X_G} T^{\ast}X_G\xrightarrow{\bar{\tau}}T^{\ast}X_G\\
T^{\ast}({K\times_{K_L}X_G}\times X_G)\xleftarrow{d{p}}({K\times_{K_L}X_G}\times X_G)\times_{X_G} T^{\ast}X_G\xrightarrow{{\tau}}T^{\ast}X_G
\end{align*}
Then by (\ref{eq:familylimit}) and (\ref{eq:intermediate})
\begin{align*}
\overline{p}_{\ast}(j_{\ast}(i_{\ast}CC(\widetilde{R\iota_{\ast}\mathcal{F}})))=\overline{\tau}_{\ast} d\overline{p}^{\ast}\left(\lim_{s\rightarrow 0}
i_{\ast}CC(\widetilde{R\iota_{\ast}\mathcal{F}})+sd\log f\right).
\end{align*}
By Lemma 6.4  \cite{SV}, the function $f$ can be chosen  
in such a way that for every sufficiently small $s>0$, 
$T^{\ast}_{K}(K\times_{K_L}X_G\times X_G)+sd\log f$
is transverse to 
$(\overline{K\times_{K_L}X_G}\times X_G)\times_{X_G} T^{\ast}X_G$.
The transversality condition implies that
the geometric inverse image 
$d\overline{p}^{-1}(i_{\ast}CC(\widetilde{R\iota_{\ast}\mathcal{F}})+sd\log f)$
is well-defined as a cycle and
$$d\overline{p}^{\ast}(i_{\ast}CC(\widetilde{R\iota_{\ast}\mathcal{F}})+sd\log f)=
d\overline{p}^{-1}(i_{\ast}CC(\widetilde{R\iota_{\ast}\mathcal{F}})+sd\log f).$$
Moreover, by Proposition \ref{prop:familylimit} 
\begin{align}\label{eq:dercor}
\overline{\tau}_{\ast} d\overline{p}^{\ast}\left(\lim_{s\rightarrow 0}
i_{\ast}CC(\widetilde{R\iota_{\ast}\mathcal{F}})+sd\log f\right)
=\lim_{s\rightarrow 0}\overline{\tau}_{\ast}d\overline{p}^{-1}(i_{\ast}CC(\widetilde{R\iota_{\ast}\mathcal{F}})+sd\log f).
\end{align}
Next, from Equation (\ref{eq:open1}) for each $s>0$, $i_{\ast}CC(\widetilde{R\iota_{\ast}\mathcal{F}})+s d\log f$ 
is a cycle of  $K\times_{K_L}X_G\times X_G$. Consequently, the family of cycles $\{i_{\ast}CC(\widetilde{R\iota_{\ast}\mathcal{F}})+s d\log f\}_{s>0}$ lies entirely in $T^{\ast}(K\times_{K_L}X_G\times X_G)$.
This permits us to use the map $\tau_{\ast}$ and $dp$  instead of $\overline{\tau}_{\ast}$ and $d\overline{p}$
 on the right hand side of (\ref{eq:dercor}) and write
\begin{align}\label{eq:eqcorfinal}
\overline{p}_{\ast}(j_{\ast}(i_{\ast}CC(\widetilde{R\iota_{\ast}\mathcal{F}})))&=
\lim_{s\rightarrow 0}\tau_{\ast}d{p}^{-1}(i_{\ast}CC(\widetilde{R\iota_{\ast}\mathcal{F}})+s d\log f)\\
&=\lim_{s\rightarrow 0}\tau_{\ast}d{p}^{-1}\left(\sum_{K_L-\mathrm{orbits }~S_L~\mathrm{in}~X_L} \chi_{S_{L}}^{mic}(\mathcal{F})[\overline{T_{i(K\times_{K_L}\iota(S_{L}))}^{\ast}((K\times_{K_L}X_G)\times X_G)}]+s d\log f\right).\nonumber
\end{align}
To compute the right hand side of (\ref{eq:eqcorfinal}) we begin by noticing that
$$(K\times_{K_L}X_G)\times X_G\times_{X_G} T^{\ast}X_G=(K\times_{K_L}X_G)\times T^{\ast}X_G=T^{\ast}_{K\times_{K_L}X_G}(K\times_{K_L}X_G)\times T^{\ast}X_G.$$
Hence $dp$ can be written as
$$dp: T^{\ast}_{K\times_{K_L}X_G}(K\times_{K_L}X_G)\times T^{\ast}X_G \rightarrow T^{\ast}(K\times_{K_L}X_G\times X_G),$$ 
with the image of each point $(0,\lambda)\in T^{\ast}_{(k,x),K\times_{K_L}X_G}(K\times_{K_L}X_G)\times T_{x'}^{\ast}X_G$ given by
\begin{align}
(0,\lambda) \mapsto \lambda\circ dp_{((k,x),x')}=(0,\lambda). 
\end{align}
Therefore $dp$ defines an embedding, and we conclude that $dp^{-1}(i_{\ast}CC(\widetilde{R\iota_{\ast}\mathcal{F}})+
s d\log f)$ is simply the intersection 
between $T_{K\times_{K_L}X_G}^{\ast}(K\times_{K_L}X_G)\times T^{\ast}X_G$ and 
$i_{\ast}CC(\widetilde{R\iota_{\ast}\mathcal{F}})+s d\log f$.
Now, by Lemma \ref{lem:contention} the
space $K\cdot \mathrm{Ch}(R{\iota}_{\ast}\mathcal{F})$ is contained 
in the characteristic variety of $I_{L}^{G}(\mathcal{F})$. Moreover, for each $s>0$, $dp^{-1}(i_{\ast}CC(\widetilde{R\iota_{\ast}\mathcal{F}})+
s d\log f)$ defines a non-zero cycle. Hence 
from the description of the elements in 
 $T_{((k,x),kx),i(K\times_{K_L}\iota(S_{L}))}^{\ast}((K\times_{K_L}X_G)\times X_G)$
given before Equation (\ref{eq:before}), 
we obtain that for each point $(k,x)\in K\times_{K_L}\iota(S_{L})$ there exists 
$\lambda(f)_{kx}\in T_{kx}^{\ast}X_G$ such that
$$\frac{df_{(k,x)}}{f(k,x)}=(\lambda(f)_{kx}\circ da_{(k,x)}).$$
Each point at the intersection of $T_{(k,x,kx),K\times_{K_L}X_G}(K\times_{K_L}X_G)\times T^{\ast}X_G$ and 
$i_{\ast}CC(\widetilde{R\iota_{\ast}\mathcal{F}})+s d\log f$ is then of the form
$$(-s\lambda(f)_{kx}\circ da_{(k,x)},\lambda+s\lambda(f)_{kx})+\left(s\frac{df_{(k,x)}}{f(k,x)},0\right)=(0,\lambda+s\lambda(f)_{kx}),\text{ where }\lambda\in T_{kx,K\cdot\iota(S_L)}^{\ast}X_G.$$
Consequently
$$\tau_{\ast}dp^{-1}\left({T_{((k,x),kx),i(K\times_{K_L}\iota(S_{L}))}^{\ast}(K\times_{K_L} X_G)\times X_G}+s \frac{df_{(k,x)}}{f(k,x)}\right)=
{T_{kx,K\cdot\iota(S_L)}^{\ast}X_G}+s\lambda(f)_{kx}.$$
Thus, if for each cycle $C\in \mathscr{L}(X_G,K)$
with support $|C|\subset {K\cdot \iota(X_L)}$ (See Equation \ref{eq:supportCycle}) and for each $s>0$, we define 
$C + s\lambda(f)$ as the cycle of $X_G$ equal to the image of $C$ under the automorphism 
\begin{align*}
(x,\xi)\mapsto (x,\xi+s\lambda (f)_{x}),\quad x\in {K\cdot \iota(X_L)},~\xi \in T_{x}^{\ast}X_G,
\end{align*} 
we obtain 
\begin{align*}
\tau_{\ast}d{p}^{-1}(i_{\ast}CC(\widetilde{R\iota_{\ast}\mathcal{F}})+s d\log f)
&=\tau_{\ast}dp^{-1}\left(\sum_{K_L-\mathrm{orbits }~S_L~\mathrm{in}~X_L} \chi_{S_{L}}^{mic}(\mathcal{F})[\overline{T_{i(K\times_{K_L}\iota(S_{L}))}^{\ast}((K\times_{K_L}X_G)\times X_G)}]\right.\\
&\qquad\qquad\qquad \qquad\qquad\qquad\qquad\qquad\qquad \qquad\qquad\qquad\qquad\qquad\qquad+sd\log f\Bigg)\\
&=\sum_{K_L-\mathrm{orbits }~S_L~\mathrm{in}~X_L} \chi_{S_{L}}^{mic}(\mathcal{F})[\overline{T_{K\cdot\iota(S_L)}^{\ast}X_G}]+s\lambda(f).
\end{align*}
Therefore
\begin{align*}
CC\left(I_{L}^{G}\mathcal{F}\right)&=a_{\ast}(CC(\widetilde{R\iota_{\ast}\mathcal{F}}))\\
&=\overline{p}_{\ast}(j_{\ast}(i_{\ast}(CC(\widetilde{R\iota_{\ast}\mathcal{F}}))))\\
&=\lim_{s\rightarrow 0}\sum_{K_L-\mathrm{orbits }~S_L~\mathrm{in}~X_L} \chi_{S_{L}}^{mic}(\mathcal{F})[\overline{T_{K\cdot\iota(S_L)}^{\ast}X_G}]+s\lambda(f).
\end{align*}
Next, since 
for each $x\in a(K\times_{K_L}\iota(X_{L}))=K\cdot \iota(X_{L}) $ and every element $\lambda\in T_{x,K\cdot \iota(S_L)}^{\ast}X_G$,
we have $\lim_{s\rightarrow 0} \lambda+ s\lambda(f)|_{x}=\lambda,$
the restriction of $CC(I_{L}^{G}\mathcal{F})$ to 
$a(K\times_{K_L}\iota(X_{L}))$
must coincide with the cycle 
$$\sum_{K_L-\mathrm{orbits }~S_L~\mathrm{in}~X_L} \chi_{S_{L}}^{mic}(\mathcal{F})[\overline{T_{K\cdot\iota(S_L)}^{\ast}X_G}].$$ Consequently, if we  
write
$$CC_{\partial} =CC\left(I_{L}^{G}\mathcal{F}\right)-\sum_{K_L-\mathrm{orbits }~S_L~\mathrm{in}~X_L} \chi_{S_{L}}^{mic}(\mathcal{F})[\overline{T_{K\cdot\iota(S_L)}^{\ast}X_G}],$$
then the support $|CC_{\partial}|$ of the cycle $CC_{\partial}$ will be
contained in the inverse image of the intersection
$[(K\cdot \iota(X_{L}))^c~\cap~ \partial(K\cdot \iota(X_{L}))]$ 
in $T^{\ast}X_G$. Here $\partial(K\cdot \iota(X_{L}))$ denotes the boudary of $K\cdot \iota(X_{L})$ in the flag variety
$X_{G}$. By a similar argument to the one of (4.6c)\cite{SV}, we can moreover conclude 
that $|CC_{\partial}|$ is the union of the conormal bundles of a family of 
$K$-orbits in $[(K\cdot \iota(X_{L}))^c~\cap~ \partial(K\cdot \iota(X_{L}))]$.
This leads us to the desired equality
$$CC\left(I_{L}^{G}\mathcal{F}\right)=\sum_{K_L-\mathrm{orbits }~S_L~\mathrm{in}~X_L} \chi_{S_{L}}^{mic}(\mathcal{F})[\overline{T_{K\cdot\iota(S_L)}^{\ast}X_G}]+
\sum_{\substack{K-orbits~S\text{~in~}\\
(K\cdot \iota(X_{L}))^c~\cap~ \partial(K\cdot \iota(X_{L})) }}  \chi_{S}^{mic}(I_{L}^{G}\mathcal{F})[\overline{T_{S}^{\ast}X_G}].$$ 
\end{proof}
To end with this section notice that from (\ref{eq:diagramcycle1}) 
and Proposition \ref{prop:reduction2}(6),
we can by defining
\begin{align}
\left(I_{L}^{G}\right)_{\ast}:\mathscr{L}(X_{L},K_L)&\rightarrow  \mathscr{L}(X_{G},K)\\
\sum_{K_{L}-\text{orbits }S\text{ in }X_L}m_S[\overline{T_{S}^{\ast}X_L}]& \mapsto a_{\ast}\left(
\sum_{K_{L}-\text{orbits }S\text{ in }X_L}m_S[\overline{T_{K\times_{K_L}\iota(S)}^{\ast}(K\times_{K_L} X_G)}]\right)\nonumber
\end{align}
extend (\ref{eq:cdiagramme1}) to obtain the commutative diagram 
\begin{align*}
  \xymatrix{
   K\mathcal{M}(\mathfrak{g},K,I_{X_{G}}) \ar[r] & K(X_{G},K)\ar[r]^{CC} & \mathscr{L}(X_G,K)\\
    K\mathcal{M}(\mathfrak{l},K_{L},I_{X_{L}}) \ar[u]^{\mathscr{R}_{(\mathfrak{l},K_{L})}^{(\mathfrak{g},K)}}\ar[r] & K(X_{L},K_{L})\ar[u]^{I_{L}^{G}}\ar[r]^{CC} & \mathscr{L}(X_{L},K_L).\ar[u]^{\left(I_{L}^{G}\right)_{\ast}} 
     }
\end{align*}
\section{The Langlands Correspondence}
In this section we present the objects of \cite{ABV} required 
for the definition of micro-packets.
The section is a quick review 
of Chapters 2, 4, 5, 6 and  10 \cite{ABV}.
We begin by describing the context in which we will do representation theory. 
We recall the notion of extended group and of (projective) representation of a strong real form. Next, we 
introduce the dual objects that are going to parameterize the set of representations and replace the Langlands parameters in the Langlands classification, namely, the set of geometric parameters. Once
the geometric parameters have been introduced, we reformulate the local Langlands Correspondence 
in this new setting, and explain how it can be extended to include representations of a special type of covering group.\\

Following the philosophy of \cite{ABV}, in this 
article we are not going to fix a real form and study the corresponding set of representations. 
Instead, we fix an inner class of real forms and 
consider at the same time the set of representations of each real form in the inner class.
Extended groups have been introduced in \cite{ABV}, 
as a manner
to study and describe 
in an organized and uniform way, the representation theory corresponding 
to an inner class of real forms.\\

Let ${G}$ be a connected reductive complex algebraic group. Write 
$\Psi_{0}(G)=\left(X^{\ast},\Delta,X_{\ast},\Delta^{\vee}\right)$ for the based root datum of $G$ and set Aut$(\Psi_{0}(G))$ for its group of automorphism. 
 An \textbf{extended group containing} $G$ (see Definition 2.13 \cite{ABV}) is a real Lie group $G^{\Gamma}$
containing $G$ as a subgroup of index two. That is, there is a short exact sequence
\begin{align*}
1\rightarrow {G}\rightarrow {}G^{\Gamma}\rightarrow \Gamma\rightarrow 1,
\end{align*}
where $\Gamma=\mathrm{Gal}(\mathbb{C}/\mathbb{R})$,
such that every element of $G^{\Gamma}-G$ acts on $G$ as an antiholomorphic automorphism.
A \textbf{strong real form} of $G^{\Gamma}$ is an element  
$\delta\in G^{\Gamma}$ such that $\delta^{2}$ is central and has finite order. 
To each strong real form we associate a real form $\sigma(\delta)$ of $G$ defined by
conjugation by $\delta$.
The group of real points $G(\mathbb{R},\delta)$ of $\delta$ is defined to be the group of real points of $\sigma(\delta)$, namely $G(\mathbb{R},\delta)=G(\mathbb{R},\sigma_{\delta})$. Notice that what we call here an extended group is called a weak extended group in \cite{ABV}. Moreover, we have:
\begin{itemize}
\item From Proposition 2.14 \cite{ABV},
the set of real forms of $G$ associated to strong real forms of the extended group $G^{\Gamma}$ constitutes exactly one inner class of real forms of $G$. 
\item The extended group 
$G^{\Gamma}$ for $G$ 
may be characterized in terms of two invariants (see Corollary 2.16 \cite{ABV}): 
$$G^{\Gamma}\longrightarrow (a,\overline{z}).$$
The first invariant is an automorphism $a$ of the canonical based root datum $\Psi_{0}(G)$, which
is induced from conjugation by any element $\delta\in G^{\Gamma}-G$. 
To define the second invariant write $\theta_Z$ for the antiholomorphic involution of $Z(G)$ 
defined by the conjugation action of any element $\delta\in G^{\Gamma}-G$. 
Fix a quasisplit real form in the inner class defined by $a$, and chose $\delta_{q}$ 
so that $\sigma(\delta_{q})=\sigma_{q}$. Then the second invariant 
is the coset $\overline{z}\in Z(G)^{\sigma_Z}/(1+\sigma_Z)Z(G)$ of $\delta_{q}^{2}\in Z(G)^{\sigma_Z}$. The pair of invariants $(a,\overline{z})$ determines the weak extended group $G^{\Gamma}$ up to isomorphism. Furthermore, for any couple $(a,\overline{z})$ defined as before, 
there is a weak extended group $G^{\Gamma}$ with invariants $(a,\overline{z})$.
\end{itemize}

Let us recall now the notion of $E$-group, of (projective) representation of a strong real from of $G$ and introduce the dual objects that are going to parameterize them.

In order to extend 
the Langlands Correspondence to include projective representations of
some special type of covering group of $G$,
the role played by the $L$-groups in the descriptions of the representations 
must be replaced by the more general notion of $E$-group. 
The $E$-groups 
are introduced for the first time by Adams and Vogan in \cite{AV}.
$E$-groups will also be necessary 
to describe in Section 4.3 the Adams-Johnson packets.

Let ${}^{\vee}\Psi_{0}(G)=\left(X_{\ast},\Delta^{\vee},X^{\ast},\Delta\right)$ be the dual based root datum to $\Psi_{0}(G)$. We have
\begin{align}\label{eq:dualautomorphism1}
\text{Aut}(\Psi_{0}({G}))\cong\text{Aut}\left({}^{\vee}\Psi_{0}({G})\right).
\end{align} 
A \textbf{dual group} for $G$ is a complex connected reductive algebraic group ${}^{\vee}{G}$ whose
based root datum  is dual to the based root datum of $G$ (i.e.
$\Psi_{0}\left({}^{\vee}{G}\right)\cong{}^{\vee}\Psi_{0}({G})$).
A \textbf{weak $E$-group} for $G$ (see Definition 4.3 \cite{ABV}) is an algebraic group ${}^{\vee}G^{\Gamma}$ containing the dual group ${}^{\vee}{G}$ for $G$ 
as a subgroup of index two. That is, there is a short exact sequence
\begin{align*}
1\rightarrow {}^{\vee}{G}\rightarrow {}^{\vee}G^{\Gamma}\rightarrow \Gamma\rightarrow 1,
\end{align*}
where $\Gamma=\mathrm{Gal}(\mathbb{C}/\mathbb{R})$. Finally, an \textbf{$E$-group} for $G$ (see Definition 4.6, Definition 4.12 and Definition 4.14 \cite{ABV}) is a pair $\left({}^{\vee}G^{\Gamma},\mathcal{S}\right)$, where ${}^{\vee}G^{\Gamma}$ is a weak $E$-group for $G$ and $\mathcal{S}$ a conjugacy class of pairs $({}^{\vee}\delta,{}^{d}B)$ with ${}^{\vee}\delta$ an element of finite order in ${}^{\vee}G^{\Gamma}-{}^{\vee}{G}$ and ${}^{d}B$ a Borel subgroup of ${}^{\vee}G$, such that for any $\left({}^{\vee}\delta,{}^{d}B\right)\in \mathcal{S}$ conjugation by 
${}^{\vee}{\delta}$ is a \textbf{distinguished} involutive automorphism  of ${}^{\vee}G$ 
(see the definition after Proposition 2.11 \cite{ABV})
preserving ${}^{d}B$.

 Just as for extended groups, there is a classification of (weak) E-groups in terms of first and second
invariants (see Proposition 4.4 and Proposition 4.7 \cite{ABV}).  
\begin{itemize}
\item To any weak $E$-group ${}^{\vee}G^{\Gamma}$ for $G$ we can attach two invariants: 
$${}^{\vee}G^{\Gamma}\longrightarrow (a,\overline{z}).$$
Similar to extended groups, the first invariant is an automorphism $a\in \text{Aut}(\Psi_{0}({}^{\vee}G))(\cong \text{Aut}(\Psi_{0})(G))$, which
is induced from conjugation by any element ${}^{\vee}\delta\in {}^{\vee}G^{\Gamma}-{}^{\vee}G$. 
To define the second invariant write $\theta_Z$ for the holomorphic involution of $Z({}^{\vee}G)$ 
defined by the conjugation action of any element ${}^{\vee}\delta\in {}^{\vee}G^{\Gamma}-{}^{\vee}G$. 
Let ${}^{\vee}\delta$ be any element in ${}^{\vee}G^{\Gamma}-{}^{\vee}G$ such that the conjugation action of ${}^{\vee}\delta$ on ${}^{\vee}G$ is a distinguished automorphism (such element necessarily exist). Then the second invariant 
is the coset $\overline{z}\in Z({}^{\vee}G)^{\theta_Z}/(1+\theta_Z)Z({}^{\vee}G)$ of 
${}^{\vee}\delta^{2}\in Z({}^{\vee}G)^{\theta_Z}$. The pair of invariants $(a,\overline{z})$ determines the weak extended group $G^{\Gamma}$ up to isomorphism. Furthermore, for any couple $(a,\overline{z})$ defined as before, there is a weak $E$ group ${}^{\vee}G^{\Gamma}$ with invariants $(a,\overline{z})$. 
\item The invariants of the $E$-group are the automorphism $a$ attached to ${}^{\vee}G^{\Gamma}$. The
second invariant of an $E$-group is the canonical representative of the
second invariant $\overline{z}$ given by the square of any element $\delta\in \mathcal{S}$, i.e. 
$z={}^{\vee}\delta^{2}\in Z({}^{\vee}G)^{\theta_Z}$. 
Conversely, if ${}^{\vee}G^{\Gamma}$ is a weak $E$-group with invariants $(a,\overline{z})$ and $z$ is a representative for the class of $\overline{z}$, then there is an
$E$-group structure on ${}^{\vee}G^{\Gamma}$ with second invariant $z$.

\item An \textbf{$L$-group} for $G$ is an $E$-group whose second invariant is equal to $1$, that is to say $z={}^{\vee}\delta^{2}=1$.
\end{itemize}
Suppose now that $G^{\Gamma}$ is an extended group for $G$ with first invariant $a$. Then a (weak) $E$-group ${}^{\vee}G^{\Gamma}$ with first invariant $a$ will be called 
\textit{\textbf{a (weak)} $E$\textbf{-group for} $G^{\Gamma}$} or \textit{\textbf{a (weak)} $E$\textbf{-group for} $G$ \textbf{and the specified inner class of real forms}}.\\


To define the notion of projective representation for a strong real form of $G$, we need first to introduce the family of covering groups that is going to be of interest to us.

Let $G^{sc}$ be the simply-connected covering group of $G$ 
and write 
 $\pi_{1}(G)$ for the kernel of the covering map $G^{sc}\rightarrow G$. 
 Let $G^{can}$ be the \textbf{canonical covering} of $G$, namely, the projective limit of all
the \textbf{distinguished coverings} of $G$ (see Definition 10.1 \cite{ABV}). We notice that the kernel of each distinguished coverings may be seen as a quotient of $\pi_{1}(G)$. We have
$$1\rightarrow \pi_1(G)^{can}\rightarrow {G}^{can}\rightarrow G\rightarrow 1.$$
The group $\pi_1(G)^{can}$ is a pro-finite abelian group, the inverse limit of certain finite quotients of $\pi_{1}(G)$.
Now, if $\delta$ is a strong real form of $G^{\Gamma}$, let ${G}(\mathbb{R},\delta)^{can}$ be the preimage of ${G}(\mathbb{R},\delta)$ in
${G}^{can}$. Then there is a short exact sequence 
\begin{align}\label{eq:cancover}
1\rightarrow \pi(G)^{can}\rightarrow {G}(\mathbb{R},\delta)^{can}\rightarrow 
{G}(\mathbb{R},\delta)\rightarrow 1.
\end{align}
A \textbf{canonical projective representation of a strong real form} of $G^{\Gamma}$ (see Definition 2.13 and Definition 10.3 \cite{ABV})
is a pair $(\pi,\delta)$ in which: $\delta$ is a strong real form of $G^{\Gamma}$ and
$\pi$ is an admissible representation of $G^{can}(\mathbb{R},\delta)$.
Two such representations $(\pi,\delta)$, $(\pi',\delta')$ are said to be 
equivalent if there is an element $g\in G$ such that 
$g\delta g^{-1}=\delta'$ and $\pi\circ Ad(g^{-1})$ is (infinitesimally) equivalent to
$\pi'$.

Now, let $z\in Z({}^{\vee}G)$ be an element of finite order in $Z({}^{\vee}G)^{\theta_Z}$. By Lemma 10.2$(d)$ \cite{ABV}, 
$z$ determines a character $\chi_{z}:\pi_{1}(G)^{can}\rightarrow \mathbb{C}^{\times}$.  
 We say then that $(\pi,\delta)$ is of \textbf{type $z$} if the restriction of $\pi$ to 
$\pi_{1}(G)^{can}$ is a multiple of $\chi_z$. When $z=1$ the character $\chi_{z}$ is trivial  
and the projective representation of type $z$ are actual representations of $G(\mathbb{R}, \delta)$. 
Finally define
$$\Pi^{z}(G/\mathbb{R})$$
to be the set of infinitesimal equivalence classes of irreducible canonical projective representations
of type $z$ of strong real form of $G^{\Gamma}$. When $z=1$ we simply write $\Pi(G/\mathbb{R})$, this is the set of infinitesimal equivalence classes of irreducible representations of strong real form of $G^{\Gamma}$.\\

The local Langlands Correspondence (see \cite{Langlands} and \cite{Borel}), as originally conceived, is a bijection between ${}^{\vee}G$-orbits of $L$-parameters and $L$-packets.
The underlying novelty of \cite{ABV}, is the introduction of a topological space $X\left({}^{\vee}{G}^{\Gamma}\right)$ (see Definition 6.9 \cite{ABV}), called space of geometric parameters, which reparameterizes the set of $L$-parameters. In other terms, the authors define an space $X\left({}^{\vee}{G}^{\Gamma}\right)$ equipped with a ${}^{\vee}G$-action such that the ${}^{\vee}{G}$-orbits of $X\left({}^{\vee}{G}^{\Gamma}\right)$ are in bijection with ${}^{\vee}{G}$-orbits of $L$-parameters for ${}^{\vee}{G}^{\Gamma}$. Consequently, the local Langlands Correspondence can be stated as a bijection between ${}^{\vee}G$-orbits of $X\left({}^{\vee}{G}^{\Gamma}\right)$ and $L$-packets. Furthermore, Adams, Barbasch and Vogan refine the Langlands Correspondence to a bijection between (equivalence classes of) what they refer to as complete geometric parameters and (equivalence classes of) (projective) representations of strong real forms. Another important property of $X\left({}^{\vee}G^{\Gamma}\right)$, 
resides in its richer geometry when compared with the geometry of the space of Langlands parameters. 
Orbits are not necessary closed in $X\left({}^{\vee}{G}^{\Gamma}\right)$. This property is crucial in the definition of micro-packets.

Let us introduce the space of geometric parameters and give a quick review on some of their properties. We end the section with the formulation of the Local Langlands Correspondence as stated in \cite{ABV}.\\

Fix an extended group $G^{\Gamma}$ for $G$, and an $E$-group $\left({}^{\vee}{G}^{\Gamma},\mathcal{S}\right)$ for the corresponding inner class of real forms. Let $\lambda\in {}^{\vee}{\mathfrak{g}}$ be a semisimple element and set $\mathcal{O}=
{}^{\vee}G\cdot \lambda$. 
For every semisimple $\lambda\in{}^{\vee}\mathfrak{g}$, let 
$\mathfrak{n}(\lambda)$ be the positive integral eigenspaces of ad$(\lambda)$ (see Equation (6.1)(d)\cite{ABV}). We define  the \textbf{canonical flat through} $\lambda$ as the affine subspace
\begin{align*}
\Lambda=\lambda+\mathfrak{n}(\lambda)
\end{align*}
and set
\begin{align*}
\mathcal{F}(\mathcal{O})&=\text{set of canonical flats in }\mathcal{O}.
\end{align*}
This set encodes information about the restriction of $L$-parameters to the connected component of the Weil group of $\mathbb{R}$ (cf. Proposition 5.6 \cite{ABV}).
Write 
\begin{align}\label{eq:groupslambda}
{}^{\vee}{G}(\lambda)&=\text{centralizer in }{}^{\vee}{G}\text{ of }\exp(2\pi i\lambda),\\
L(\lambda)&=\text{centralizer in }{}^{\vee}{G}\text{ of }\lambda,\nonumber\\
N(\lambda)&=\text{connected unipotent subgroup with Lie algebra }\mathfrak{n}(\lambda),\nonumber\\
P(\lambda)&=L(\lambda)N(\lambda).\nonumber
\end{align}
Then $P(\lambda)$ is a parabolic subgroup of ${}^{\vee}{G}(\lambda)$ with Levi decomposition 
$P(\lambda)=L(\lambda)N(\lambda)$. Moreover, from Proposition 6.5 \cite{ABV} for every $\lambda'\in\Lambda=\lambda + \mathfrak{n}(\lambda)$ we have
${}^{\vee}{G}(\lambda')={}^{\vee}{G}(\lambda)$ and
${P}(\lambda')={P}(\lambda).$
Therefore, the definition of the groups ${}^{\vee}{G}(\lambda)$ and $P(\lambda)$ only depend on the canonical flat, 
and we can respectively define ${}^{\vee}{G}(\Lambda)={}^{\vee}{G}(\lambda)$ and $P(\Lambda)=P(\lambda)$.

Fix $\Lambda\in\mathcal{F}(\mathcal{O})$ and consider the sets
\begin{align*}
\mathcal{C}(\mathcal{O})=\{ge(\Lambda)g^{-1}:g\in {}^{\vee}{G}\}\quad\text{ and }\quad
\mathcal{I}(\mathcal{O})=\{y\in {}^{\vee}{G}^{\Gamma}-{}^{\vee}{G}:y^{2}\in\mathcal{C}(\mathcal{O})\}.
\end{align*}
This last set encodes information about the restriction of $L$-parameters to the non-connected component of the Weil group of $\mathbb{R}$.
For every ${}^{\vee}{G}$-orbit $\mathcal{O}$ of semi-simple elements in $\mathfrak{g}$,
the set of \textbf{geometric parameters (for $\mathcal{O}$)} is defined as 
\begin{align}\label{eq:Vofgeometricparameter}
X\left(\mathcal{O},{}^{\vee}{G}^{\Gamma}\right)=\mathcal{F}(\mathcal{O})\times_{\mathcal{C}(\mathcal{O})}\mathcal{I}(\mathcal{O}).
\end{align}
Since $X\left(\mathcal{O},{}^{\vee}{G}^{\Gamma}\right)$ is a fibre product, it carries a natural structure of a complex algebraic variety (see Proposition 6.16 \cite{ABV}). Now,
the set of all \textbf{geometric parameters} is the disjoint union (see definition 6.9 \cite{ABV}) 
\begin{align}\label{eq:VofgeometricparameterUnion}
X\left({}^{\vee}G^{\Gamma}\right)=\bigsqcup_{\mathcal{O}} X\left(\mathcal{O},{}^{\vee}G^{\Gamma}\right).
\end{align}
To apply the result on characteristic cycles as expressed in the previous section, we need to relate the variety of geometric parameters to a (generalized) flag variety, this is done in Proposition 6.16 \cite{ABV}. Let us sketch the result.
From Proposition 6.13 \cite{ABV} the set $\mathcal{I}(\mathcal{O})$
decomposes into a finite number of ${}^{\vee}{G}$-orbits. List the orbits as $\mathcal{I}_{1}(\mathcal{O}),\cdots,
\mathcal{I}_{r}(\mathcal{O})$ and for each $1\leq i\leq r$ choose a point
$$y_{i}\in \mathcal{I}_{i}(\mathcal{O})\quad\text{with}\quad y_{i}^{2}=e(\Lambda).$$
Then conjugation by $y_{i}$ defines an involutive automorphism of ${}^{\vee}{G}(\Lambda)$
with fixed point set denoted by $K_{y_i}$. 
 Therefore, $X\left(\mathcal{O},{}^{\vee}{G}^{\Gamma}\right)$ is the disjoint union of $r$-closed subvarieties
\begin{align}\label{eq:varieties}
X_{y_i}\left(\mathcal{O},{}^{\vee}{G}^{\Gamma}\right)=
\mathcal{F}(\mathcal{O})\times_{\mathcal{C}(\mathcal{O})}\mathcal{I}_{i}(\mathcal{O})
\end{align}
 and from Proposition 6.16 \cite{ABV} for each one of this varieties we have
\begin{align}\label{eq:varietiesisomor}
X_{y_i}\left(\mathcal{O},{}^{\vee}{G}^{\Gamma}\right)\cong {}^{\vee}{G}\times_{K_{y_i}}{}^{\vee}{G}(\Lambda)/P(\Lambda).
\end{align} 
In particular, the orbits of ${}^{\vee}{G}$ on $X_{y_i}\left(\mathcal{O},{}^{\vee}{G}^{\Gamma}\right)$ are in one-to-one correspondence 
with the orbits of $K_{y_i}$ on the partial flag variety 
${}^{\vee}{G}(\Lambda)/P(\Lambda)$ 
and from Proposition 7.14 \cite{ABV} 
this correspondence preserves their closure relations and the nature of the
singularities of closures. Furthermore, for each $i$ we have the equivalence of categories 
$\mathcal{P}(X_{y_{i}}\left(\mathcal{O},{}^{\vee}G^{\Gamma}\right),{}^{\vee}G)\cong \mathcal{P}({}^{\vee}G(\Lambda)/P(\lambda),K_{y_{i}})$, $\mathcal{D}(X_{y_{i}}\left(\mathcal{O},{}^{\vee}G^{\Gamma}\right),{}^{\vee}G)\cong \mathcal{D}({}^{\vee}G(\Lambda)/P(\lambda),K_{y_{i}})$.
From Proposition \ref{prop:reduction2}, the techniques described in section 2 apply then to the variety of geometric parameters. 
This is going to be important for the next section.\\ 

Now, let $\Phi\left({}^{\vee}G^{\Gamma}\right)$ be the set of ${}^{\vee}G$-orbits of \textbf{$L$-parameters} for ${}^{\vee}G^{\Gamma}$ (see \cite{Langlands} and Section 8 \cite{Borel}). As mentioned before, the ${}^{\vee}{G}$-orbits in $X\left({}^{\vee}G^{\Gamma}\right)$ are in bijection with 
$\Phi\left({}^{\vee}G^{\Gamma}\right)$ (see Proposition 6.17 \cite{ABV}). From this we can reformulate the local Langlands Correspondence as a bijection
\begin{align}\label{eq:cllc}
\left\{{}^{\vee}{G}-\text{orbits in }X\left({}^{\vee}G^{\Gamma}\right)\right\}\longleftrightarrow \Phi\left({}^{\vee}G^{\Gamma}\right)\xleftrightarrow{LLC} \{L-\text{packets of }\Pi^{z}(G/\mathbb{R})\}.
\end{align}
To refine the Local Langlands correspondence into a bijection with the set $\Pi^{z}(G/\mathbb{R})$, the authors of \cite{ABV} supplement each ${}^{\vee}{G}$-orbit in $X\left({}^{\vee}G^{\Gamma}\right)$ with the representation of a finite group, defining thereby, what they refer to as the set of \textbf{complete geometric parameters}.
More precisely, let  $x\in X\left({}^{\vee}{G}^{\Gamma}\right)$,  set 
${}^{\vee}{G}_{x}$ to be the isotropy group of $x$ and write 
${}^{\vee}{G}_{x}^{alg}$ for its preimage 
in the universal algebraic cover (see Definition 1.16 \cite{ABV}): 
$$1\rightarrow \pi_{1}\left({}^{\vee}G^{alg}\right)\rightarrow {}^{\vee}G^{alg}\rightarrow {}^{\vee}G\rightarrow 1.$$
The finite group in question is the \textbf{equivariant fundamental group} at $x$ (see Definition 7.1 and Definition 7.6 \cite{ABV})
$$A_{x}^{alg}={}^{\vee}{{G}}_{x}^{alg}/({}^{\vee}{{G}}_{x}^{alg})_{0}.$$
A \textbf{complete geometric parameter} for  ${}^{\vee}{G}^{\Gamma}$ 
is then a pair 
\begin{align}\label{eq:completegeometricparameter}
(x,\tau)
\end{align}
with 
$x\in X\left({}^{\vee}{G}^{\Gamma}\right)$ 
and $\tau$ an irreducible representation of $A_{\varphi}^{alg}$. We make ${}^{\vee}G^{alg}$ act on the set of complete Langlands parameter by conjugation and define $\Xi\left({}^{\vee}G^{\Gamma}\right)$ to be the set of conjugacy classes of complete geometric parameters for ${}^{\vee}G^{\Gamma}$ (see Definition 7.6 \cite{ABV}).
When we want to specify the choice of of an $E$-group $\left({}^{\vee}G^{\Gamma},\mathcal{S}\right)$ for $G$ with second invariant $z$ we write $\Xi^{z}\left({G}/\mathbb{R}\right)=\Xi\left({}^{\vee}{G}^{\Gamma}\right)$, or simply $\Xi\left({G}/\mathbb{R}\right)$ if ${}^{\vee}{G}^{\Gamma}$ is an $L$-group.
Similarly, we denote $\Phi^{z}\left({G}/\mathbb{R}\right)=\Phi\left({}^{\vee}{G}^{\Gamma}\right)$ 
when working with the set of ${}^{\vee}G$-conjugacy class of $L$-parameters (see Definition 5.11 \cite{ABV}).  
Finally, for each ${}^{\vee}{G}$-orbit $\mathcal{O}$ of semi-simple
elements in $\mathfrak{g}$, we define
$\Xi\left(\mathcal{O},{}^{\vee}G^{\Gamma}\right)$ as the subset of
$\Xi\left({}^{\vee}{G}^{\Gamma}\right)$ consisting of 
${}^{\vee}G$-orbits ${}^{\vee}G\cdot (x,\tau)$ with $x\in X\left(\mathcal{O},{}^{\vee}{G}^{\Gamma}\right)$ and denote $\Xi^{z}\left(\mathcal{O},{G}/\mathbb{R}\right)=\Xi\left(\mathcal{O},{}^{\vee}{G}^{\Gamma}\right)$ in the case that we want to make the choice of the $E$-group with second invariant $z$ explicit. 


We can now state the Local Langlands Correspondence (see Theorem 10.4 \cite{ABV}).
\begin{theo}[\textbf{Local Langlands Correspondence for $E$-groups}]\label{theo:4.1}
Suppose $G^{\Gamma}$
is an extended group for $G$, and $({}^{\vee}{G}^{\Gamma},\mathcal{S})$ is a $E$-group for the corresponding inner class of real forms. Write $z$ for the second invariant of
the $E$-group. Then there is a natural bijection between the set $\Pi^{z}({G}/\mathbb{R})$ 
of equivalence classes of canonical irreducible projective representations of strong real forms of ${G}$ of type $z$
and the set $\Xi^{z}({G}/\mathbb{R})$ of complete geometric parameters 
for ${}^{\vee}{G}^{\Gamma}$.   
\end{theo}
We notice that, when $\left({}^{\vee}{G}^{\Gamma},\mathcal{S}\right)$ is an $L$-group for $G$
the projective representations in the Correspondence
are actual representations of strong real forms of ${G}$.
Moreover, in this parameterization, the set of representations of a fixed real form ${G}(\mathbb{R},\delta)$ corresponding to complete Langlands parameter supported on a single orbit is precisely 
the $L$-packet for ${G}(\mathbb{R},\delta)$ attached to that orbit.
\section{Micro-packets}
In this section we give a quick review on micro-packets. For a more complete exposition on the subject see Chapters 7, 19 and 22 of \cite{ABV}.\\

For all of this section, let $G^{\Gamma}$ be an extended group for $G$, 
and let $\left({}^{\vee}{G}^{\Gamma},\mathcal{S}\right)$ be an $E$-group 
 for the corresponding inner class of real forms. Suppose $(x,\tau)$ is a complete geometric parameter for ${}^{\vee}{G}^{\Gamma}$. 
Write 
$V_{\tau}$ for the space of $\tau$ and 
$S_{x}={}^{\vee}{{G}}\cdot x$ for the corresponding ${}^{\vee}{G}$-orbit on $X\left({}^{\vee}{G}^{\Gamma}\right)$. 
From  Lemma 7.3 \cite{ABV}, by regarding $\tau$ as a representation of ${}^{\vee}{{G}}_{x}^{alg}$ 
trivial on $({}^{\vee}{{G}}_{x}^{alg})_{0}$, the induced bundle
$$\mathcal{V}_{x,\tau}:={}^{\vee}{{G}}^{alg}\times_{{}^{\vee}{{G}}_{x}^{alg}}V_{\tau}\rightarrow S_{x},$$
carries a ${}^{\vee}{G}^{alg}$-invariant flat connection. Therefore, 
$\mathcal{V}_{x,\tau}$ defines an irreducible ${}^{\vee}{G}^{alg}$-equivariant local system on $S_x$. 
Moreover, by Lemma 7.3(e) \cite{ABV}, 
the map
\begin{align}\label{e{appendices}q:bijectionLangGeo}
(x,\tau)\mapsto \xi_{x,\tau}:=(S_{x},\mathcal{V}_{x,\tau})
\end{align}
induces a bijection between $\Xi\left({}^{\vee}G^{\Gamma}\right)$ the set of equivalence classes of complete geometric parameters on $X\left({}^{\vee}G^{\Gamma}\right)$,
and the set of couples $(S,\mathcal{V})$ where $S$ is an orbit of ${}^{\vee}{G}$ on $X\left({}^{\vee}G^{\Gamma}\right)$ 
and $\mathcal{V}$ an irreducible ${}^{\vee}{G}^{alg}$-equivariant local system on $S$. 
Following this bijection, the set of couples $(S,\mathcal{V})$ will also be denoted 
$\Xi\left({}^{\vee}G^{\Gamma}\right)$ and
for each ${}^{\vee}{G}$-orbit $\mathcal{O}$ of semi-simple elements in $\mathfrak{g}$, we identify $\Xi\left(\mathcal{O},{}^{\vee}{G}^{\Gamma}\right)$ with the set of couples $(S,\mathcal{V})$ with $S\subset X\left(\mathcal{O},{}^{\vee}{G}^{\Gamma}\right)$.\\

Let $\mathcal{O}$ be a ${}^{\vee}{G}$-orbit of semisimple elements in ${}^{\vee}\mathfrak{g}$. 
As in Section 2, write
$$\mathcal{P}\left(X\left(\mathcal{O},{}^{\vee}{G}^{\Gamma}\right),{}^{\vee}{{G}}^{alg}\right)$$
for the category of ${}^{\vee}{{G}}^{alg}$-equivariant perverse sheaves 
on $X\left(\mathcal{O},{}^{\vee}{G}^{\Gamma}\right)$. We 
define 
$\mathcal{P}\left(X\left({}^{\vee}{G}^{\Gamma}\right), {}^{\vee}{{G}}^{alg}\right)$ 
 to be the direct sum over semisimple orbits $\mathcal{O}\subset {}^{\vee}{\mathfrak{g}}$ 
of the categories $\mathcal{P}\left(X\left(\mathcal{O},{}^{\vee}{G}^{\Gamma}\right),{}^{\vee}{{G}}^{alg}\right)$. 
The last necessary 
step before the introduction of
micro-packets, is to explain 
how the irreducible objects in $\mathcal{P}\left(X\left({}^{\vee}{G}^{\Gamma}\right),{}^{\vee}{{G}}^{alg}\right)$ 
are parameterized by the set $\Xi\left({}^{\vee}{G}^{\Gamma}\right)$ of equivalence classes of complete geometric parameters. Fix $\xi=(S,\mathcal{V})\in \Xi\left(\mathcal{O},{}^{\vee}{G}^{\Gamma}\right)$. 
Write
$$j:S\rightarrow \overline{S},$$
for the inclusion of $S$ in its closure and
$$i:\overline{S}\rightarrow X\left(\mathcal{O},{}^{\vee}{G}^{\Gamma}\right),$$ 
for the inclusion of the closure of $S$ in $X(\mathcal{O},{}^{\vee}{G}^{\Gamma})$.
Let $d=d(S)$ be the dimension of $S$. If we regard the local system $\mathcal{V}$ as
a constructible sheaf on $S$, the complex $\mathcal{V}[-d]$, consisting of the single sheaf $\mathcal{V}$ 
in degree $-d$ defines an ${}^{\vee}{{G}}^{alg}$-equivariant perverse sheaf on $S$. 
Applying to it the intermediate extension functor 
$j_{!\ast}$, followed by the direct image $i_{\ast}$, 
we get an irreducible perverse sheaf on  $X(\mathcal{O},{}^{\vee}{G}^{\Gamma})$
\begin{align}\label{eq:irredPerv}
P(\xi)=i_{\ast}j_{!\ast}\mathcal{V}[-d],
\end{align}
the perverse extension of $\xi$.
 That every irreducible ${}^{\vee}{G}^{alg}$-equivariant perverse sheaf on $X({}^{L}G)$ is of this form, follows from Theorem 4.3.1 \cite{BBD}. 
The set
$\left\{P(\xi):\xi\in \Xi\left({}^{\vee}{G}^{\Gamma}\right)\right\}$ 
is therefore a base of the Grothendieck group $KX\left({}^{\vee}{G}^{\Gamma}\right)$. We can finally give the
definition of a micro-packet.
\begin{deftn}[Definition 19.13 \cite{ABV}]\label{deftn:micropacket}
Let $S$ be a ${}^{\vee}{G}$-orbit in $X\left({}^{\vee}{G}^{\Gamma}\right)$. To every complete geometric parameter
$\xi\in\Xi\left({}^{\vee}{G}^{\Gamma}\right)$ we have attached in (\ref{eq:irredPerv}) a perverse sheaf $P(\xi)$. 
From Proposition 19.12 \cite{ABV}, the conormal bundle $T^{\ast}_{S}\left(X({}^{\vee}{G}^{\Gamma})\right)$ has a non-negative integral multiplicity
$\chi_{S}^{mic}(\xi)$ in the characteristic cycle $CC(P(\xi))$ of $P(\xi)$. 
We define the \textbf{micro-packet of geometric parameters} attached to $S$, to be the set of complete geometric parameters for which this multiplicity is non-zero
\begin{align*}
\Xi\left({}^{\vee}{G}^{\Gamma},{}^{\vee}{G}\right)^{mic}_{S}=\left\{\xi\in\Xi\left({}^{\vee}{G}^{\Gamma}\right):\chi_{S}^{mic}(\xi)\neq 0\right\}.
\end{align*}
\end{deftn}

\begin{deftn}[Definition 19.15 \cite{ABV}]\label{deftn:micropacket}
Suppose 
$\left({}^{\vee}{G}^{\Gamma},\mathcal{S}\right)$ is an $E$-group for $G$ 
 with second invariant $z$.
Let $\varphi\in\Phi^{z}({G}/\mathbb{R})$ be an equivalence class of Langlands parameters for ${}^{\vee}{G}^{\Gamma}$ and write $S_{\varphi}$ for the corresponding orbit of ${}^{\vee}{{G}}$ in $X\left({}^{\vee}{G}^{\Gamma}\right)$ (see Equation (\ref{eq:cllc})).
Then we define the micro-packet of geometric parameters attached to $\varphi$ as
\begin{align*}
\Xi^{z}({G}/\mathbb{R})_{\varphi}=\Xi\left(X\left({}^{\vee}{G}^{\Gamma}\right),{}^{\vee}{G}\right)^{mic}_{S_{\varphi}}.
\end{align*}
For any complete parameter $\xi\in \Xi^{z}\left({G}/\mathbb{R}\right)$, let $\pi(\xi)$ 
be the representation in $\Pi^{z}({G}/\mathbb{R})$ 
associated to $\xi$ by Theorem \ref{theo:4.1}. 
Then the \textbf{micro-packet} of $\varphi$ is defined as  
\begin{align*}
\Pi^{z}({G}/\mathbb{R})^{mic}_{\varphi}=\{\pi(\xi'):\xi'\in \Xi^{z}({G}/\mathbb{R})^{mic}_{\varphi}\}.
\end{align*}
Finally, let $\delta$ be a strong real form of $G^{\Gamma}$, then we define the restriction of 
$\Pi({G}/\mathbb{R})^{mic}_{\varphi}$ to $\delta$ as
\begin{align*}
\Pi^{z}({G}(\mathbb{R},\delta))^{mic}_{\varphi}=\{\pi\in \Pi^{z}({G}/\mathbb{R})^{mic}_{\varphi}: \pi \text{ is a representation of }G(\mathbb{R},\delta) \}.
\end{align*} 
\end{deftn}
We notice that the Langlands packet attached to a Langlands parameter $\varphi$ is always contained in the corresponding micro-packet
\begin{align}\label{eq:contentionoflpackets}
\Pi^{z}({G}/\mathbb{R})_{\varphi}\subset\Pi^{z}({G}/\mathbb{R})^{mic}_{\varphi}.
\end{align}
This is a consequence of $(ii)$ of the following lemma. Point $(i)$ of the lemma  
will show to be quite useful later in this section. For a proof, see Lemma 19.14 \cite{ABV}.
\begin{lem}\label{lem:lemorbit}
Let $\eta=(S,\mathcal{V})$ and $\eta'=\left(S',\mathcal{V}'\right)$ be two geometric parameters for 
${}^{\vee}{G}^{\Gamma}$. 
\begin{enumerate}[i.]
\item If $\eta'\in \Xi\left(X\left({}^{\vee}{G}^{\Gamma}\right),{}^{\vee}{G}\right)_{S}^{mic}$, then $S\subset \overline{S}'$.
\item If $S'=S$, then $\eta'\in \Xi\left(X\left({}^{\vee}{G}^{\Gamma}\right),{}^{\vee}{G}\right)_{S}^{mic}$. 
\end{enumerate}
\end{lem}
We are going to be mostly interested in the case of micro-packets attached to
Langlands parameters coming from Arthur parameters. 
These types of micro-packets, are going to be called Adams-Barbasch-Vogan packets or simply ABV-packets. 
 We continue by recalling the definition of an Arthur parameter.

An \textbf{Arthur parameter} is a homomorphism
\begin{align*}
\psi:W_{\mathbb{R}}\times\textbf{SL}(2,\mathbb{C})\longrightarrow {}^{\vee}{G}^{\Gamma},
\end{align*}
where $W_{\mathbb{R}}$ denotes the Weil group of $\mathbb{R}$ (i.e. $W_{\mathbb{R}}=\mathbb{C}^{\times}\cup j\mathbb{C}^{\times}$ with $j^{2}=-1$ and $jzj^{-1}=\bar{z}$), satisfying
\begin{itemize}
\item The restriction of $\psi$ to $W_{\mathbb{R}}$ is a tempered Langlands parameter (i.e.
the closure of $\psi(W_{\mathbb{R}})$ in the analytic topology is compact).
\item The restriction of $\psi$ to $\mathbf{SL}(2,\mathbb{C})$ is holomorphic.
\end{itemize}
Two such parameters are called equivalent if they are conjugate by the action of ${}^{\vee}{{G}}$. The set of equivalences classes is written $\Psi\left({}^{\vee}{G}^{\Gamma}\right)$
or $\Psi^{z}({G}/\mathbb{R})$, when we want to specify that 
${}^{\vee}{G}^{\Gamma}$ is an $E$-group for $G$ with second invariant $z$.
To every Arthur parameter $\psi$, we can associate a Langlands parameter $\varphi_{\psi}$, by
the following formula (see Section 4 \cite{Arthur89})
\begin{align*}
\varphi_{\psi}&:W_{\mathbb{R}}\longrightarrow {}^{\vee}G^{\Gamma},\\
 \varphi_{\psi}(w)&=\psi\left(w,\left(\begin{array}{cc}
|w|^{1/2}& 0\\
             0& |w|^{-1/2}
\end{array}\right)\right).
\end{align*}
Now, to $\varphi_{\psi}$ correspond an orbit $S_{\varphi_{\psi}}$ of ${}^{\vee}{G}$ on $X\left({}^{\vee}{G}^{\Gamma}\right)$. We define
\begin{align*}
S_{\psi}=S_{\varphi_{\psi}}.
\end{align*}
\begin{deftn}\label{deftn:abvpackets}
Suppose 
$\left({}^{\vee}{G}^{\Gamma},\mathcal{S}\right)$ is an 
$E$-group for $G$ 
with second invariant $z$.
Let $\psi\in\Psi^{z}({G}/\mathbb{R})$ be an Arthur parameter. We define the Adams-Barbasch-Vogan packet $\Pi^{z}({G}/\mathbb{R})_{\psi}^{\mathrm{ABV}}$ 
of $\psi$, as the micro-packet of the Langlands parameter $\varphi_{\psi}$ attached to $\psi$:
\begin{align*}
\Pi^{z}({G}/\mathbb{R})_{\psi}^{\mathrm{ABV}}=\Pi^{z}({G}/\mathbb{R})_{\varphi_{\psi}}^{mic}.
\end{align*}
\end{deftn}
In other words, $\Pi^{z}({G}/\mathbb{R})_{\psi}^{\mathrm{ABV}}$ is the
set of all irreducible representations with the 
property that the corresponding irreducible perverse sheaf contains the conormal bundle $T_{S_{\psi}}^{\ast}\left(X\left({}^{\vee}{G}^{\Gamma}\right)\right)$ 
in its characteristic cycle.

Micro-packets attached to Arthur parameters satisfy the following important properties.

\begin{theo}\label{theo:abvproperties}
Let $\psi$ be an Arthur parameter for ${}^{\vee}{G}^{\Gamma}$. 
\begin{enumerate}[i.]
\item $\Pi^{z}(G/\mathbb{R})_{\psi}^{mic}$ contains the $L$-packet $\Pi^{z}(G/\mathbb{R})_{\varphi_{\psi}}$.
\item $\Pi^{z}(G/\mathbb{R})_{\psi}^{mic}$ is the support of a stable formal virtual character:
$$\eta_{\psi}^{mic}=\sum_{\xi\in\Xi({G}/\mathbb{R})_{\varphi_{\psi}}}e(\xi)(-1)^{d(S_{\xi})-d(S_{\psi})}\chi_{S_{\xi}}^{mic}(P(\xi))\pi(\xi).$$
where for each orbit $S$, $d(S)$ is the dimension of the orbit and $e(\xi)$ the Kottwitz sign attached to the real form of which $\pi(\xi)$ is a representation. 
\item $\Pi^{z}(G/\mathbb{R})_{\psi}^{mic}$ satisfies the ordinary endoscopic identities predicted by the theory of endoscopy. 
\end{enumerate}
\end{theo}
 As explained after (\ref{eq:contentionoflpackets}), point $(i)$ is a consequence of Lemma \ref{lem:lemorbit}$(ii)$. For a proof of the second statement, see Corollary 19.16\cite{ABV} and Theorem 22.7\cite{ABV}. Finally, for a proof and a more precise statement of the last point, see Theorem 26.25 \cite{ABV}.
\subsection{Tempered representations}
In this section we study ABV-packets attached to tempered Langlands parameters. Recall that 
$\varphi\in P({}^{\vee}{G}^{\Gamma})$ is said to be \textbf{tempered} if the closure of $\varphi(W_{\mathbb{R}})$ in the analytic topology is compact.
\begin{prop}
Let $\varphi$ be a tempered Langlands parameter for ${}^{\vee}{G}^{\Gamma}$. Then  under (\ref{eq:cllc}) the corresponding orbit 
$S_{\varphi}$ of ${}^{\vee}{G}$ in $X\left({}^{\vee}{G}^{\Gamma}\right)$, is open and dense. 
\end{prop}
\begin{proof}
The assertion follows from Proposition 22.9$(b)$ {\cite{ABV}},
applied to an Arthur parameter with trivial $\textbf{SL}(2,\mathbb{C})$ part.
Indeed, suppose $\psi$ is an Arthur parameter with restriction to $W_{\mathbb{R}}$ equal to 
$\varphi$ and trivial $\textbf{SL}(2,\mathbb{C})$ part. Let 
$(y,\lambda)$ be the couple corresponding to $\varphi$ under Proposition 5.6 \cite{ABV}. Define $P(\lambda)$ as in (\ref{eq:groupslambda}), 
and write $K(y)$ for the centralizer of $y$ in ${}^{\vee}G^{\Gamma}$.
Finally, set $$E_{\psi}=d\psi|_{\textbf{SL}(2,\mathbb{C})}
\left(\begin{array}{cc}
0& 1\\
0& 0
\end{array}\right)=0.$$ 
Then by Proposition 22.9$(b)$ {\cite{ABV}} for each $x\in S_{\varphi}$ the orbit
$(P(\lambda)\cap K(y))\cdot E_{\psi}$ is dense in 
$T_{S_{\varphi},x}^{\ast}\left(X\left({}^{\vee}G^{\Gamma}\right)\right)$.  
But $E_{\psi}=0$, hence $T_{S_{\varphi},x}^{\ast}\left(X\left({}^{\vee}G^{\Gamma}\right)\right)=0$, that is, the annihilator of $T_x S_{\varphi}$ in $T_x^{\ast}(X({}^{\vee}G^{\Gamma}))$ is 
equal to zero.     
Therefore, $T_{x}\left(X\left({}^{\vee}G^{\Gamma}\right)\right)=T_x S_{\varphi}$
and the result follows. 
\end{proof}

\begin{cor}
Suppose $({}^{\vee}{G}^{\Gamma},\mathcal{S})$
is an $E$-group for $G$ 
with second invariant $z$.
Let $\varphi$ be a tempered Langlands parameter for ${}^{\vee}G^{\Gamma}$, then
$$\Pi^{z}({G}/\mathbb{R})_{\varphi}^{\mathrm{ABV}}=\Pi^{z}({G}/\mathbb{R})_{\varphi}.$$
where at right, $\Pi^{z}({G}/\mathbb{R})_{\varphi}$ denotes the Langlands packet of $\varphi$.
\end{cor}
\begin{proof}
By Theorem \ref{theo:abvproperties}($i$), the $L$-packet $\Pi^{z}({G}/\mathbb{R})_{\varphi}$ is contained in $\Pi^{z}({G}/\mathbb{R})_{\varphi}^{{mic}}$. We only need to show the opposite inclusion.
Let $\pi\in\Pi^{z}({G}/\mathbb{R})_{\varphi}^{\text{ABV}}$ and write $S$ 
for the orbit of ${}^{\vee}{G}$ in $X({}^{\vee}{G}^{\Gamma})$ corresponding to the Langlands parameter of $\pi$ under Proposition 6.17 \cite{ABV} (see also Equation (\ref{eq:cllc})). From Lemma \ref{lem:lemorbit} and the definition of $\Pi^{z}({G}/\mathbb{R})_{\varphi}^{\text{ABV}}$, the orbit $S$ contains $S_{\varphi}$ in its closure. As 
$S_{\varphi}$ is open, $S_{\varphi}\cap S\neq \emptyset$ and we have $S_{\varphi}=S$.
By Lemma \ref{lem:lemorbit}($ii$), $\pi\in\Pi^{z}(G/\mathbb{R})_{\varphi}$ and we can conclude the desired inclusion.
\end{proof}
\subsection{Essentially principal unipotent Arthur parameters}
In this section we give a full description of the ABV-packets attached to essentially unipotent Arthur parameters, whose retriction to 
$\mathbf{SL}(2,\mathbb{C})$ is a principal morphism. 
Our goal is to 
prove that
the  
ABV-packets corresponding to such parameters, 
consists of characters, one for each real form of $G$ in our fixed inner class.   
We obtain this through a 
slight generalization of Theorem 27.18 \cite{ABV} (see Theorem \ref{theo:unipotentparameter} below) which only treats the principal unipotent case. 
This result  
is a key step in the proof 
 that the packets defined in \cite{Adams-Johnson} are ABV-packets.\\

Let $\psi$ be an Arthur parameter of ${}^{\vee}G^{\Gamma}$. We say that\\ 

$\bullet$  $\psi$ is \textbf{unipotent}, if its restriction to the identity component $\mathbb{C}^{\times}$ of 
$W_{\mathbb{R}}$ is trivial.\\

More generally, we say that:\\

$\bullet$ $\psi$ is \textbf{essentially
unipotent}, if the image of its restriction to the identity component $\mathbb{C}^{\times}$ of 
$W_{\mathbb{R}}$ is contained in the center $Z({}^{\vee}{G})$ of ${}^{\vee}{G}$.\\

Next, fix a morphism
\begin{align*}
\psi_{1}:\mathbf{SL}(2,\mathbb{C})\longrightarrow {}^{\vee}{G}.
\end{align*}
Then we say that $\psi_1$ is \textbf{principal}, if $\psi(\mathbf{SL}(2,\mathbb{C}))$ contains a principal unipotent element. Finally, the (essentially)
unipotent Arthur parameter $\psi$ is called \textbf{(essentially) principal unipotent Arthur parameter} if 
$\psi|_{\mathbf{SL}(2,\mathbb{C})}$ is principal.\\

The next result due to Adams, Barbasch and Vogan (see Theorem 27.18 \cite{ABV} and the remark at the end of page 310 of \cite{ABV})
gives a description of the set $\Pi(G/\mathbb{R})_{\psi}^{\text{ABV}}$ for $\psi$ a principal unipotent Arthur parameter. 
\begin{theo}\label{theo:unipotentparameter}
Suppose $\left({}^{\vee}{G}^{\Gamma},\mathcal{S}\right)$
is an $E$-group for $G$ 
with second invariant $z$.
Fix a principal morphism
\begin{align*}
\psi_{1}:\mathbf{SL}(2,\mathbb{C})\longrightarrow {}^{\vee}{G}.
\end{align*}
\begin{enumerate}[a)]
\item The centralizer $S_{0}$ of $\psi_{1}$ in ${}^{\vee}{G}$ is $Z({}^{\vee}{G})$.
\item Suppose $z=1$ (i.e. $\left({}^{\vee}{G}^{\Gamma},\mathcal{S}\right)$ is an $L$-group). Then the set of equivalence classes of unipotent Arthur parameters attached to $\psi_{1}$ may be identified with
$$H^{1}(\Gamma,Z({}^{\vee}{G}))=\{z\in Z({}^{\vee}{G}):z\theta_{Z}(z)=1\}/\{w\theta_{Z}^{-1}(w):w\in Z({}^{\vee}{G})\}.$$
More generally, if $z\in (1+\theta_{Z})Z({}^{\vee}G)$
then the set of equivalence classes of unipotent Arthur parameters attached to $\psi_{1}$ 
is a principal homogeneous space for
$H^{1}(\Gamma,Z({}^{\vee}{G})).$

\item The unipotent representations of type $z$ (of some real form $G(\mathbb{R},\delta))$ attached
to $\psi_{1}$ are precisely the projective representations of type z trivial on the identity component $G(\mathbb{R},\delta))_0$.
\item Suppose $z\in (1+\theta_{Z})Z({}^{\vee}G)$.
Let $\delta$ be any strong real form of $G^{\Gamma}$ and write  
$$\mathrm{Hom}^{z}_{\mathrm{cont}}({G}^{can}(\mathbb{R},\delta)/{G}^{can}(\mathbb{R},\delta)_{0},\mathbb{C}^{\times}),$$ for the set of characters of type $z$ of
$G^{can}(\mathbb{R},\delta)$ trivial on $G^{can}(\mathbb{R},\delta)_0$.
Then there is a natural surjection
\begin{align}\label{eq:mapunipotent}
H^{1}(\Gamma,Z({}^{\vee}{G}))\rightarrow \mathrm{Hom}^{z}_{\mathrm{cont}}({G}^{can}(\mathbb{R},\delta)/{G}^{can}(\mathbb{R},\delta)_{0},\mathbb{C}^{\times}).
\end{align}
\item 
Suppose $\psi$ is a unipotent Arthur parameter attached to $\psi_1$ and $\delta$
is a strong real form of $G^{\Gamma}$. Write $\pi(\psi,\delta)$ for the character of 
$G^{can}(\mathbb{R},\delta)$ trivial on $G^{can}(\mathbb{R},\delta)_0$ attached to $\psi$ by composing the bijection of $(b)$ with the surjection of $(c)$. Let $P(\psi,\delta)$, be the perverse 
sheaf on $X\left({}^{\vee}G^{\Gamma}\right)$ corresponding to $\pi(\psi,\delta)$
under the Local Langlands Correspondence (see Theorem \ref{theo:4.1}). Then for any perverse sheaf $P$ on $X\left({}^{\vee}G^{\Gamma}\right)$, we have   
$$\chi_{S_{\psi}}^{mic}(P)=\left\{\begin{array}{cl}
1& \text{if }P\cong P(\psi,\delta),\quad\text{ for some strong real form }\delta\text{ of }G\\
0& otherwise.
\end{array}\right.
$$
Consequently,
\begin{align}\label{eq:repunipotenttordu}
\Pi^{z}(G/\mathbb{R})_{\psi}^{\mathrm{ABV}}=\{\pi(\psi,\delta)\}_{\delta \text{ strong real form of } G^{\Gamma}}.
\end{align}
\end{enumerate}
\end{theo}

We turn now to the study of ABV-packets attached to 
essentially principal unipotent Arthur parameters. 
We begin 
by describing the behaviour of micro-packets under twisting. We record this in 
the next proposition. 
In order to enunciate the result we 
recall that for each strong real form $\delta$ of 
$G^{\Gamma}$, there is a natural morphism
\begin{align}\label{eq:cocylemap}
H^{1}(W_{\mathbb{R}},Z({}^{\vee}{G})) &\rightarrow \mathrm{Hom}_{\mathrm{cont}}({G}(\mathbb{R},\delta),\mathbb{C}^{\times})\\
\mathbf{a}&\mapsto \chi(\mathbf{a},\delta),\nonumber
\end{align}
which is surjective and maps cocyles with compact image to unitary characters of $G(\mathbb{R},\delta)$ (see for example Section 2 \cite{Langlands}). 
\begin{prop}\label{prop:torsion}
Suppose $\left({}^{\vee}{G}^{\Gamma},\mathcal{S}\right)$ is an $E$-group for $G$ with second invariant $z$.
Suppose $\mathbf{a}\in H^{1}(W_{\mathbb{R}},Z({}^{\vee}{G}))$ and let the cocycle $\mathrm{a}$ be a representative of $\mathbf{a}$. For each strong real form $\delta$ of 
$G^{\Gamma}$ let $\chi(\mathbf{a},\delta)$ be the character of $G(\mathbb{R},\delta)$ attached to $\mathbf{a}$ as in (\ref{eq:cocylemap}).
\begin{enumerate}[i.]
\item If $\varphi$ is a Langlands parameter for ${}^{\vee}G^{\Gamma}$, then the morphism
$$\varphi_{\mathbf{a}}:W_{\mathbb{R}}\rightarrow{}^{\vee}G^{\Gamma},\quad w\mapsto (\mathrm{a}(w),1)\varphi(w)$$
defines a Langlands parameter of ${}^{\vee}G^{\Gamma}$ whose equivalence class depends exclusively on $\mathbf{a}$ and the equivalence class of $\varphi$.  Furthermore
\begin{align*}
\Pi^{z}(G/\mathbb{R})^{mic}_{\varphi_{\mathbf{a}}}=\coprod_{\substack{\delta\text{ strong real form}\\
                  \text{of }G^\Gamma}}\Pi^{z}(G(\mathbb{R},\delta))^{mic}_{\varphi_{\mathbf{a}}}
\end{align*} 
where
\begin{align*}
\Pi^{z}(G(\mathbb{R},\delta))^{mic}_{\varphi_{\mathbf{a}}}=\{\pi\otimes\chi(\mathbf{a},\delta):\pi\in \Pi^{z}(G(\mathbb{R},\delta))^{mic}_{\varphi}\}.
\end{align*} 
\item 
If $\psi$ is an Arthur parameter for ${}^{\vee}G^{\Gamma}$, then the morphism
\begin{align}\label{eq:arthurcocycle}
\psi_{\mathbf{a}}:W_{\mathbb{R}}\times\mathbf{SL}(2,\mathbb{C})\rightarrow{}^{\vee}G^{\Gamma},\quad (w, g)\mapsto (a(w), 1)\psi(w,g)
\end{align}
defines an Arthur parameter of ${}^{\vee}G^{\Gamma}$ whose equivalence class depends exclusively on $\mathbf{a}$ and the equivalence class of $\psi$. Furthermore
$$\Pi^{z}(G/\mathbb{R})_{\psi_{\mathbf{a}}}^{\mathrm{ABV}}=\coprod_{\substack{\delta\text{ strong real form}\\
                  \text{of }G^\Gamma}}\Pi^{z}(G(\mathbb{R},\delta))^{mic}_{\psi_{\mathbf{a}}}$$ 
where
\begin{align*}
\Pi^{z}(G(\mathbb{R},\delta))^{mic}_{\psi_{\mathbf{a}}}=\{\pi\otimes\chi(\mathbf{a},\delta):\pi\in \Pi^{z}(G(\mathbb{R},\delta))^{mic}_{\psi}\}.
\end{align*} 
\end{enumerate}
\end{prop}
\begin{proof}
To check that $\varphi_{\mathbf{a}}$ defines a Langlands parameter is straightforward. 
Let $\lambda_{\varphi}$ be defined as in Equation (5.7)(a) \cite{ABV}. By the definition of $\varphi_{\mathbf{a}}$
there exist $\lambda_{\mathbf{a}}\in Z({}^{\vee}\mathfrak{g})$ such that $\lambda_{\varphi_{\mathbf{a}}}=
\lambda_{\varphi}+\lambda_{\mathbf{a}}$.
Write 
$$\mathcal{O}_{\varphi}={}^{\vee}G\cdot \lambda_{\varphi}\quad 
\text{ and }\quad  \mathcal{O}_{\varphi_{\mathbf{a}}}={}^{\vee}G\cdot \lambda_{\varphi_{\mathbf{a}}}.$$  
Then
\begin{align*}
X\left(\mathcal{O}_{\varphi},{}^{\vee}{G}^{\Gamma}\right)&\rightarrow X\left(\mathcal{O}_{\varphi_{\mathbf{a}}},{}^{\vee}{G}^{\Gamma}\right)\\
(y,\Lambda)&\mapsto (\text{exp}(\pi i \lambda_{\mathbf{a}})y,\Lambda+\lambda_{\mathbf{a}}),
\end{align*}
defines and isomorphism of varieties, that induces a bijection of orbits 
\begin{align*}
{}^{\vee}G-\text{orbits on }X(\mathcal{O}_{\varphi},{}^{\vee}{G}^{\Gamma}) &\rightarrow
{}^{\vee}G-\text{orbits on }X(\mathcal{O}_{\varphi_{\mathbf{a}}},{}^{\vee}{G}^{\Gamma})\\
S&\mapsto S',
\end{align*}
and of geometric parameters
\begin{align*}
\Xi^{z}(\mathcal{O}_{\varphi},G/\mathbb{R})&\rightarrow \Xi^{z}(\mathcal{O}_{\varphi_{\mathbf{a}}},G/\mathbb{R})\\
\xi&\mapsto\xi'.
\end{align*} 
Furthermore, from the description of irreducible perverse sheaves given in Equation (\ref{eq:irredPerv}),
we have an isomorphism $KX\left(\mathcal{O}_{\varphi},{}^{\vee}{G}^{\Gamma}\right)\cong KX\left(\mathcal{O}_{\varphi_{\mathbf{a}}},{}^{\vee}{G}^{\Gamma}\right)$ that restrict  
for each $\xi\in \Xi^{z}(\mathcal{O}_{\varphi},G/\mathbb{R})$
to an isomorphism $P(\xi)\cong P(\xi')$. Therefore, for each 
${}^{\vee}G$-orbit $S$ on $X\left(\mathcal{O}_{\varphi},{}^{\vee}G^{\Gamma}\right)$ we obtain
\begin{align}\label{eq:tensorcycles}
\chi_{S}^{mic}(P(\xi))=\chi_{S'}^{mic}(P(\xi')).
\end{align}
Now, from the properties of $L$-packets described in Section 3 \cite{Langlands}, we deduce 
that for each strong real form $\delta\in G^{\Gamma}$, the set of irreducible representations of
$G(\mathbb{R},\delta)$ corresponding to complete geometric parameters for
$X\left(\mathcal{O}_{\varphi_{\mathbf{a}}},{}^{\vee}G^{\Gamma}\right)$, under the Local Langlands Correspondence (see Theorem \ref{theo:4.1}),
is equal to the set of irreducible representations of $G(\mathbb{R},\delta)$ 
attached to complete geometric parameters for 
$X\left(\mathcal{O}_{\varphi},{}^{\vee}G^{\Gamma}\right)$
tensored by $\chi({\mathbf{a}},\delta)$. Point $(i)$ follows then from the definition of micro-packets and equality (\ref{eq:tensorcycles}).
Point $(ii)$ is a direct consequence of point $(i)$ applied to $\varphi_{\psi_{\mathbf{a}}}$ and the definition of ABV-packets.
\end{proof}
The following corollary generalizes Theorem \ref{theo:unipotentparameter}
to the case of 
$E$-groups admitting principal unipotent Arthur parameters.
\begin{cor}\label{cor:essentialunipotent}
Suppose $\left({}^{\vee}{G}^{\Gamma},\mathcal{S}\right)$ is an $E$-group for $G$. Write $z$ for the second invariant of the $E$-group, 
and suppose $z\in (1+\theta_{Z})Z({}^{\vee}G)$.
Let $\psi$ be an essentially unipotent Arthur parameter
for ${}^{\vee}{G}^{\Gamma}$ such that $\psi|_{\mathbf{SL}(2,\mathbb{C})}$ is principal. 
Then 
there exist a cocycle $\mathbf{a}\in H^{1}(W_{\mathbb{R}},Z({}^{\vee}{G}))$ and a
principal unipotent Arthur parameter $\psi_{u}$ for ${}^{\vee}G^{\Gamma}$,
such that for all
$(w,g)\in W_{\mathbb{R}}\times \mathbf{SL}(2,\mathbb{C})$ we have
$$\psi(w,g)=(\mathrm{a}(w),g)\psi_u(w,g),$$
where $\mathrm{a}$ is a representative of $\mathbf{a}$.
Finally, for each strong real form $\delta$ of $G^{\Gamma}$, 
let $\chi(\mathbf{a},\delta)$ be the character
of $G(\delta,\mathbb{R})$ corresponding to 
$\mathbf{a}$ under the map in (\ref{eq:cocylemap}), and let
$\pi(\psi_u,\delta)$ be the character 
of $G(\delta,\mathbb{R})$ attached to $\psi_u$
in point Theorem \ref{theo:unipotentparameter}$(d)$ .
Define
$$\pi(\psi,\delta)=\chi(\mathbf{a},\delta)\pi(\psi_u,\delta).$$ 
Then
\begin{align*}
\Pi^{z}(G/\mathbb{R})_{\psi}^{\mathrm{ABV}}=\{\pi(\psi,\delta)\}_{\delta \text{ strong real form of } G^{\Gamma}}.
\end{align*}
\end{cor}
\begin{proof}
Let $\psi$ be an essentially unipotent Arthur parameter
of ${}^{\vee}G^{\Gamma}$ such that $\psi|_{\mathbf{SL}(2,\mathbb{R})}=\psi_1$, is a principal morphism.
Let $\psi_{u}$ be any unipotent Arthur parameter extending $\psi_1$.
It is straightforward to check that
$$(\psi|_{W_{\mathbb{R}}})(\psi_{u}|_{W_{\mathbb{R}}})^{-1},$$
corresponds to a cocycle $a\in Z^{1}(W_{\mathbb{R}},Z({}^{\vee}{G}))$.
The first part of the corollary follows.
For each strong real form $\delta$ of $G^{\Gamma}$ let $\chi(\mathbf{a},\delta)$
and $\pi(\psi_u,\delta)$, be as in the statement of the corollary.  
From Theorem \ref{theo:unipotentparameter}$(e)$  we have 
$$\Pi^{z}(G/\mathbb{R})_{\psi_u}^{\mathrm{ABV}}=\{\pi(\psi_u,\delta)\}_{\delta \text{ strong real form of } G^{\Gamma}}.$$
Hence from Proposition \ref{prop:torsion}($ii)$ we conclude 
\begin{align*}
\Pi(G/\mathbb{R})_{\psi}^{\mathrm{ABV}}
&=\{\pi(\psi_u,\delta)\chi(\mathbf{a},\delta)\}_{\delta \text{ strong real form of } G^{\Gamma}}\\
&=\{\chi(\psi,\delta)\}_{\delta \text{ strong real form of } G^{\Gamma}}
\end{align*}
\end{proof}

The next result fully generalizes Theorem \ref{theo:unipotentparameter} to $E$-groups
admitting essentially principal unipotent Arthur parameters. 
The techniques employed in the proof are the same as the one used to show Theorem \ref{theo:unipotentparameter} in \cite{ABV}(see pages 306-310 \cite{ABV}). They are based on results
coming from Chapters 20-21 \cite{ABV}, which are still valid
in our framework, and from Chapter 27, 
only  proved in \cite{ABV} in the case of representations with infinitesimal character $\mathcal{O}$ arising from homomorphisms $\mathbf{SL}(2,\mathbb{C})\rightarrow {}^{\vee}G$.
Since the infinitesimal character of the representations that we consider here  
differs from $\mathcal{O}$ by a central element, the results on Chapter 27, easily generalize to our setting.
\begin{theo}\label{theo:essentiallyunipotentparameter}
Suppose $\left({}^{\vee}{G}^{\Gamma},\mathcal{S}\right)$ is an $E$-group for $G$. Write $z$ for the second invariant of the $E$-group, and suppose there is $\lambda$, a central element in ${}^{\vee}{\mathfrak{g}}$ with $[\lambda,\theta_{Z}(\lambda)]=0$, such that
$$z\exp(\pi i(\lambda-\theta_{Z}(\lambda)))\in (1+\theta_Z)Z({}^{\vee}G).$$
Fix a principal morphism
\begin{align*}
\psi_{1}:\mathbf{SL}(2,\mathbb{C})\longrightarrow {}^{\vee}{G},
\end{align*}
and write $\mathcal{O}={}^{\vee}G\cdot (\lambda+\lambda_1)$, where 
$\lambda_1=d\psi_{1}
\left(\begin{array}{cc}1/2& 0\\ 
0& 1/2
\end{array}\right)$.
\begin{enumerate}[a)]
\item 
The set of equivalence classes of essentially unipotent Arthur parameters attached to $\psi_{1}$,
with corresponding orbit contained in $X\left(\mathcal{O},{}^{\vee}G^{\Gamma}\right)$,
may be identified with 
$$H^{1}(\Gamma,Z({}^{\vee}{G}))=\{z\in Z({}^{\vee}{G}):z\theta_{Z}(z)=1\}/\{w\theta_{Z}^{-1}(w):w\in Z({}^{\vee}{G})\}.$$
\item The projective representations (of some real form $G^{can}(\mathbb{R},\delta)$) having infinitesimal character $\mathcal{O}$ 
attached to $\psi_{1}$ are precisely the projective characters of type $z$ with infinitesimal character $\mathcal{O}$.
\item 
Let $\delta$ be any strong real form of $G^{\Gamma}$ and write  
$\mathrm{Hom}^{z}_{\mathrm{cont}}({G}^{can}(\mathbb{R},\delta),\mathbb{C}^{\times})$, 
for the set of characters of type $z$ of
$G^{can}(\mathbb{R},\delta)$. 
Then there is a natural map
\begin{align}\label{eq:mapessentially}
H^{1}(\Gamma,Z({}^{\vee}{G}))\rightarrow \mathrm{Hom}^{z}_{\mathrm{cont}}({G}^{can}(\mathbb{R},\delta),\mathbb{C}^{\times}),
\end{align}
whose image is the set of projective characters of $G^{can}(\mathbb{R},\delta)$ of type $z$, 
having infinitesimal character $\mathcal{O}$.
\item 
Suppose $\psi$ is an essentially unipotent Arthur parameter attached to $\psi_1$ 
with corresponding orbit contained in $X\left(\mathcal{O},{}^{\vee}G^{\Gamma}\right)$. Let $\delta$
be a strong real form of $G^{\Gamma}$. Write $\pi(\psi,\delta)$ for the projective character of 
type $z$ of $G^{can}(\mathbb{R},\delta)$ 
attached to $\psi$ by composing the bijection of $a)$ with the map of $c)$. Let $P(\psi,\delta)$, be the perverse 
sheaf on $X\left({}^{\vee}G^{\Gamma}\right)$ corresponding to $\pi(\psi,\delta)$
under the Local Langlands Correspondence (see Theorem \ref{theo:4.1}). Then for all perverse sheaf $P$ on $X\left({}^{\vee}G^{\Gamma}\right)$, we have   
$$\chi_{S_{\psi}}^{mic}(P)=\left\{\begin{array}{cl}
1& \text{if }P\cong P(\psi,\delta),\quad\text{ for some strong real form }\delta\text{ of }G\\
0& otherwise.
\end{array}\right.
$$
Consequently,
\begin{align}\label{eq:repunipotenttordu}
\Pi^{z}(G/\mathbb{R})_{\psi}^{\mathrm{ABV}}=\{\pi(\psi,\delta)\}_{\delta \text{ strong real form of } G^{\Gamma}}.
\end{align}
\end{enumerate}
\end{theo}
\begin{proof}
For $(a)$ we start by fixing an essentially unipotent Arthur parameter 
$\psi_0$ attached to $\psi_1$. Notice that for each essentially unipotent Arthur parameter 
$\psi$ attached to $\psi_1$, the equivalence class of the product
$$(\psi|_{W_{\mathbb{R}}})(\psi_{0}|_{W_{\mathbb{R}}})^{-1},$$
defines a cocycle $\mathbf{a}_{\psi}\in H^{1}(\Gamma,Z({}^{\vee}{G}))$. Therefore, each
essentially unipotent Arthur parameter $\psi$ attached to $\psi_1$ can be written as 
$$\psi(w,g)=(\mathrm{a}_{\psi}(w),g)\psi_0(w,g),$$
where $\mathrm{a}_{\psi}$ is a representative of $\mathbf{a}$. Point $(a)$ follows. 

For $(b)$ we notice that the infinitesimal character corresponding to $\mathcal{O}$ (see Lemma 15.4 \cite{ABV})
is the infinitesimal character of a one dimensional representation, 
so the corresponding maximal ideal in $U(\mathfrak{g})$ (see Theorem 21.8 \cite{ABV}) is the annihilator of a one dimensional representation. Point $(b)$, like Theorem \ref{theo:unipotentparameter}(c),  
follows from Corollary 27.13 \cite{ABV}. We notice that  
Corollary 27.13 \cite{ABV} has been demonstrated in \cite{ABV}, only in the case when the infinitesimal character $\mathcal{O}$ arises from a homomorphism of $\mathbf{SL}(2,\mathbb{C})$ into ${}^{\vee}G$, but is easily generalized to our setting. Indeed, Corollary 27.13 \cite{ABV}
is a consequence of Theorem 27.10 and Theorem 27.12 \cite{ABV}, 
whose proof are still valid in the case of our infinitesimal character $\mathcal{O}$,
and of Theorem 21.6 and Theorem 21.8 \cite{ABV}, whose proof are valid for any infinitesimal
character.

For (c), we start as in (a) by fixing an essentially unipotent Arthur parameter 
$\psi_0$ attached to $\psi_1$. Write $\varphi_{\psi_0}$ for 
the corresponding Langlands parameter. Let $\delta$ be a strong real form 
of $G^{\Gamma}$, whose associated real form is quasi-split. Then the Langlands packets
$\Pi^{z}(G(\mathbb{R},\delta))_{\varphi_{\psi_0}}$ contains exactly one representation,
a canonical projective character of type $z$ that we denote $\chi(\psi,\delta_0)$.
Next, for each strong real form $\delta$ of $G^{\Gamma}$ define a canonical projective character 
$\chi(\psi,\delta)$ of type $z$, as follows; Suppose $T$ is a Cartan subgroup of $G$
with $\text{Ad}(\delta)T=T$, and such that $T(\mathbb{R},\delta)$
is a maximally split Cartan subgroup of $G(\mathbb{R},\delta)$. 
Let $g\in G$ be such that $\delta=\text{Ad}(g)\delta$,
then $gTg^{-1}(\mathbb{R},\delta)$ defines a Cartan subgroup of $G(\mathbb{R},\delta)$. 
For any $t\in T(\delta',\mathbb{R})$ set 
$$\pi(\psi,\delta)(t)=\pi(\psi,\delta_0)(\text{Ad}(g)t).$$
Then the conditions of Lemma 2.5.2 \cite{Adams-Johnson} hold, and $\chi(\psi,\delta)$ extends uniquely to a one-dimensional representation, also denoted $\chi(\psi,\delta)$
of $G(\delta',\mathbb{R})$.
Now, for each essentially unipotent Arthur parameter $\psi$ attached to $\psi_1$, 
let 
$\mathbf{a}_{\psi}\in H^{1}(\Gamma,Z({}^{\vee}{G}))$ be defined as in $(a)$.
From \ref{theo:unipotentparameter}(d), we know how to attach to each couple 
$(\mathbf{a}_{\psi},\delta)$ a character 
$\chi(\mathbf{a}_{\psi},\delta)\in \mathrm{Hom}_{\mathrm{cont}}({G}^{can}(\mathbb{R},\delta)/{G}^{can}(\mathbb{R},\delta)_{0},\mathbb{C}^{\times})$. 
We can now define the natural map of $(c)$;  
for each strong real form $\delta$ of $G^{\Gamma}$ we define (\ref{eq:mapessentially}) by sending 
$\mathbf{a}_{\psi}\in H^{1}(\Gamma,Z({}^{\vee}{G}))$ to
$$\mathbf{a}_{\psi}\mapsto \chi(\mathbf{a}_{\psi},\delta)\chi({\psi_0},\delta).$$
The surjectivity of (\ref{eq:mapessentially}) in the set of projective characters of $G^{can}(\mathbb{R},\delta)$ of type $z$, 
having infinitesimal character $\mathcal{O}$, is a consequence of the surjectivity of (\ref{eq:mapunipotent}). 
 
Finally, the proof of $(d)$ can be done along the same lines as that of Theorem \ref{theo:unipotentparameter}(e) (see page 309 \cite{ABV}). For each strong real form 
of $G$ we just need to replace the characters appearing in \ref{theo:unipotentparameter}(d) 
with the ones in the image of the map on (c), 
and use point (b) instead of \ref{theo:unipotentparameter}(c).
\end{proof}
\subsection{The Adams-Johnson construction}
In this section we study ABV-packets attached to a particular family of Arthur parameters,
namely those related to representations with cohomology. In \cite{Adams-Johnson}, Adams and Johnson
attached to any Arthur parameter in this family a packet consisting of 
representations cohomologically induced from unitary characters. Moreover, they proved that each packet defined in this way is the support of a stable distribution (see Theorem (2.13) \cite{Adams-Johnson}), and that these distributions satisfy the ordinary endoscopic identities predicted by the theory of endoscopy (see Theorem (2.21) \cite{Adams-Johnson}). 
The objective of this section is to show that the packets defined by Adams-Johnson are ABV-packets, that is, for the family of Arthur parameters studied in \cite{Adams-Johnson},
the packet associated in (\ref{deftn:abvpackets}) to any  
parameter in this family, coincides with the packet defined by Adams and Johnson.
This will follow as a corollary of a  
more general result (see Theorem \ref{theo:micropacket} below),
where under some hypothesis on the Langlads parameter $\varphi$ of $G$, we are able to reduce the description of the micro-packet corresponding to $\varphi$, 
to the description of the micro-packet of a Levi subgroup of $G$.

We begin the section by describing the family of Arthur parameters studied in \cite{Adams-Johnson} 
 and the construction of Adams-Johnson packets. 
 The description of the parameters and of the Adams-Johnson packets that we give here, is inspired by Section 3.4.2.2 \cite{Otaibi} and by Section 5 \cite{Arthur89}. 
 The main difference with these two references is that in this article, we use the Galois form of the $L$-group instead of the Weil form. By doing this we are able to describe the Adams-Johnson construction 
in the language of extended groups of \cite{ABV}. The only complication in using the Galois form is that to define the packets of Adams-Johnson it will be necessary to work with 
an $E$-group of some Levi subgroup of $G$, and consequently to use the canonical cover (see Definition 10.1 \cite{ABV} and Section 3) of this Levi subgroup, and to work with some projective characters of strong real forms of this cover.\\

Suppose 
$\left({}^{\vee}{G}^{\Gamma},\mathcal{S}\right)$ is an $L$-group for $G$. 
We recall that
$\mathcal{S}$ is a ${}^{\vee}{G}$-conjugacy class of pairs 
$\left({}^{\vee}\delta,{}^{d}{B}\right)$, with 
${}^{d}{B}$ a Borel subgroup ${}^{d}{B}$ of ${}^{\vee}G$
and 
${}^{\vee}\delta$ 
an element of order two in ${}^{\vee}G^{\Gamma}-{}^{\vee}{G}$ such that conjugation by 
${}^{\vee}\delta$ is a distinguished involutive automorphism 
$\sigma_{{}^{\vee}\delta}$ of ${}^{\vee}G$
preserving ${}^{d}B$.
Since $\sigma_{{}^{\vee}\delta}$ is distinguished, it preserves a splitting $\left({}^{\vee}G,{}^{d}{B},{}^{d}{T},\{X_{\alpha}\}\right)$ of $^{\vee}G$ that we fix from now on.
 We notice that the $L$-group ${}^{\vee}G^{\Gamma}$ can be more explicitly 
described as the disjoint union of 
${}^{\vee}G$ and the coset ${}^{\vee}G{}^{\vee}\delta$, with multiplication on ${}^{\vee}G^{\Gamma}$ defined by the rules:
$$(g_1{}^{\vee}\delta)(g_2{}^{\vee}\delta)=g_1\sigma_{{}^{\vee}\delta}(g_2) ,\qquad 
(g_1{}^{\vee}\delta)(g_2)=g_1\sigma_{{}^{\vee}\delta}(g_2){}^{\vee}\delta$$
and the obvious rules for the other two kinds of products. This explicit description of ${}^{\vee}G^{\Gamma}$, amounts to the 
 usual description of an $L$-group
as the semi-direct product of ${}^{\vee}G$ with the Galois group.

Now, let $\psi$ be an Arthur parameter for ${}^{\vee}G^{\Gamma}$ and recall the definition $W_{\mathbb{R}}=\mathbb{C}^{\times}\cup j\mathbb{C}^{\times}$. 
The restriction of $\psi$ to $\mathbb{C}^{\times}$ takes the following form (cf. Proposition 5.6 \cite{ABV}): there exist 
$\lambda$, $\lambda'\in X_{\ast}({}^{d}{T})\otimes\mathbb{C}$ with 
$\lambda-\lambda'\in X_{\ast}({}^{d}{T})$
such that
\begin{align}\label{eq:AJparameterequation0}
\psi(z)=z^{\lambda}\bar{z}^{\lambda'}.
\end{align}
We may suppose $\lambda$ is dominant for ${}^{d}{B}$.
Let ${}^{d}{L}$ be the centralizer 
of $\psi(\mathbb{C}^{\times})$ in ${}^{\vee}{{G}}$ and write ${}^{d}\mathfrak{l}$ for its lie algebra. We have
$$Z\left({}^{d}{L}\right)_{0}\subset {}^{d}{T}\subset{}^{d}{L}$$
and
$$\lambda,\lambda'\in X_{\ast}\left(Z\left({}^{d}{L}\right)_{0}\right)\otimes\mathbb{C}\quad \text{with}\quad
\lambda-\lambda'\in X_{\ast}\left(Z\left({}^{d}{L}\right)_{0}\right).
$$
Set $\psi(j)=n{}^{\vee}\delta$ where $n\in{}^{\vee}G$. Then
$$\text{Ad}(n{}^{\vee}\delta)(\lambda)=\lambda'\quad \text{and} \quad (n{}^{\vee}\delta)^{2}=\psi(-1)=(-1)^{\lambda+\lambda'}.$$
Set ${}^{d}{B}_{L}={}^{d}{L}\cap {}^{d}{B}$ and write ${}^{\vee}{\rho}_{{}^{d}{L}}$ for the 
half-sum of the positive coroots defined by the system of positive roots 
$R\left({}^{d}{B}_{{}^{d}{L}},{}^{d}{T}\right)$. 
The family of Arthur parameters studied by Adams and Johnson are those which satisfy the following properties:
\begin{enumerate}
\item[AJ1.] The identity component of $Z\left({}^{d}{L}\right)^{\psi(j)}$ is contained in $Z\left({}^{\vee}{G}\right)$. 
\item[AJ2.] $\psi(\mathbf{SL}(2,\mathbb{C}))$ contains a principal unipotent element of ${}^{d}{L}$.
\item[AJ3.] $\left<\lambda+{}^{\vee}{\rho}_{{}^{d}{L}},\alpha\right>\neq 0$ for all root 
$\alpha\in R\left({}^{\vee}{G},{}^{d}{T}\right)$.
\end{enumerate}
Let $S=\psi(\mathbf{SL}(2,\mathbb{C}))$. 
It is immediate from (AJ2), that $S$ is $\mathbf{SL}_2$-principal in ${}^{\vee}{G}$. Write $\mathfrak{s}$ 
for the lie algebra of $S$, and let $(h,e,f)$ be a $\mathfrak{sl}_{2}$-triplet generating $\mathfrak{s}$. Up to conjugation we can, and will, suppose that
$$h={}^{\vee}{\rho}_{{}^{d}{L}}\in {}^{d}{T}.$$
Since $\psi(j)=n{}^{\vee}\delta$ commutes with $S$, conjugation by $n{}^{\vee}\delta$ fixes $h$, and because $h$ is regular in ${}^{d}{L}$, $n{}^{\vee}\delta$ normalizes ${}^{d}{T}$. It is also
obvious that $n{}^{\vee}\delta$ normalizes ${}^{d}{L}$ and $Z\left({}^{d}{L}\right)$. Now, ${}^{\vee}\delta$ normalizes ${}^{d}{T}$, 
so $n$ normalizes ${}^{d}{T}$.
From now on we assume that ${}^{\vee}{G}$ is semi-simple. We make this assumption just to simplify the exposition that follows, the conclusions in
(\ref{eq:AJLparameter}) and (\ref{eq:AJLparameter2})
remain true in the reductive case.
Point (AJ1) is therefore equivalent to\\

$\mathrm{AJ1}^{\prime}$. 
$Z({}^{d}{L})_{0}^{\psi(j)}$ is trivial.\\
~\\
As $\mathbb{R}^{\times}$ is the center of $W_{\mathbb{R}}$, 
the group $\psi(\mathbb{R}^{\times})$ 
commutes with $\psi(\mathbb{R}^{\times}\times \mathbf{SL}(2,\mathbb{C}))$. Hence
$\psi(\mathbb{R}^{\times})\subset Z\left({}^{d}{L}\right)^{\psi(j)}$ and 
it follows that $\psi(\mathbb{R}_{+}^{\times})\subset Z\left({}^{d}{L}\right)_{0}^{\psi(j)}$
is trivial. Consequently,
$\lambda+\lambda'=0$
and we can write
$$\psi(z)=z^{\lambda}\bar{z}^{-\lambda}=\left(\frac{z}{|z|}\right)^{2\lambda}.$$
Hence $\psi(j)^{2}=(-1)^{2\lambda}$. In particular Ad$(\psi(j))$
is a linear automorphism of order two of ${}^{\vee}{\mathfrak{g}}$. It is semi-simple
with eigenvalues equal to $\pm 1$. Condition (AJ1$^{\prime}$) is then equivalent to
$$\text{Ad}(\psi(j))|_{\mathfrak{z}({}^{d}\mathfrak{l})}=-\text{Id}|_{\mathfrak{z}({}^{d}\mathfrak{l})}.$$
The following arguments are taken from 3.4.2.2 \cite{Otaibi}. 

Let $\mathbf{n}:W\left({}^{\vee}{G},{}^{d}{T}\right)\rtimes \Gamma\rightarrow N\left({}^{\vee}{G},{}^{d}{T}\right)\rtimes \Gamma$ be the section defined in Section 2.1 \cite{Langlands-Shelstad}. 
Let $w_0$ be the longest element in $W({}^{\vee}{G},{}^{d}{T})$ and let $w_{{}^{d}{L}}$
be the longest element in $W\left({}^{d}{L},{}^{d}{T}\right)$. Write 
\begin{align}\label{eq:n1}
n_{1}{}^{\vee}\delta=\mathbf{n}(w_{0}w_{{}^{d}{L}}{}^{\vee}\delta). 
\end{align}
Let $\Delta\left({}^{d}{B}_{{}^{d}{L}},{}^{d}{T}\right)$, be the set of simple roots of the positive root system $R\left({}^{d}{B}_{{}^{d}{L}},{}^{d}{T}\right)$. Since $w_0$ (resp. $w_{{}^{d}{L}}$) sends positive roots in  $R\left({}^{d}{B},{}^{d}{T}\right)$ (resp. $R\left({}^{d}{B}_{{}^{d}{L}},{}^{d}{T}\right)$) 
to negative roots, and because $\sigma_{{}^{\vee}\delta}$ preserves the splitting $\left({}^{\vee}G,{}^{d}{B},{}^{d}{T},\{{X}_{\alpha}\}\right)$, we can conclude that $n_{1}{}^{\vee}\delta$ preserves $\Delta\left({}^{d}{B}_{{}^{d}{L}},{}^{d}{T}\right)$.
Moreover, $n_{1}{}^{\vee}\delta$ acts as -Id on $\mathfrak{z}({}^{d}\mathfrak{l})$ and  
by $t\mapsto t^{-1}$ in $Z\left({}^{d}{L}\right)$.  
In particular, Ad$(n_{1}{}^{\vee}\delta)\lambda=-\lambda=\lambda'$. Thus  
$$(n_{1}{}^{\vee}\delta)\psi(z)(n_{1}{}^{\vee}\delta)^{-1}=\psi(\overline{z}),$$
the element  $nn_{1}^{-1}$ commutes with $\psi(\mathbb{C}^{\times})$, 
and we have $nn_{1}^{-1}\in {}^{d}{L}$.  
Furthermore, from Proposition 9.3.5 \cite{Springer}, $w_0w_{{}^{d}{L}}{}^{\vee}\delta$ preserves the splitting  
$\left({}^{d}{L},{}^{d}{B}_{{}^{d}{L}},{}^{d}{T},\{{X}_{\alpha}\}_{\alpha\in \Delta\left({}^{d}{B}_{{}^{d}{L}},{}^{d}{T}\right)}\right)$. Hence $n_{1}{}^{\vee}\delta$ commutes with $S$, and we can say the same for $nn_{1}^{-1}$. Now, $S$ is principal in ${}^{\vee}{G}$ so $nn_{1}^{-1}\in Z\left({}^{d}{L}\right)$
and there exists $t\in Z\left({}^{d}{L}\right)$ such that 
\begin{align}\label{eq:AJparameterequation01}
tn_1=n.
\end{align}
We compute
\begin{align}\label{eq:ajparameteridentity1}
(-1)^{2\lambda}=\psi(j)^{2}=(n{}^{\vee}\delta)^{2}=&(t n_{1}{}^{\vee}\delta)^{2}=t(n_1{}^{\vee}\delta)t(n_1{}^{\vee}\delta)^{-1}(n_1{}^{\vee}\delta )^{2}\\
=&tt^{-1}(n_1{}^{\vee}\delta)^{2}=(n_1{}^{\vee}\delta)^{2}.\nonumber
\end{align}
Let ${}^{d}{Q}={}^{d}{L}{}^{d}{U}$ be the parabolic subgroup of ${}^{\vee}{G}$ containing 
${}^{d}{T}$ such that the roots of ${}^{d}{T}$ in ${}^{d}{Q}$ are the 
$\alpha\in R\left({}^{\vee}{G},{}^{d}T\right)$ satisfying $\left<\lambda,\alpha\right>\geq 0$.
In particular ${}^{d}{L}$ is a Levi factor of ${}^{d}{Q}$.
The section 
$\mathbf{n}:W\left({}^{\vee}{G},{}^{d}{T}\right)\rtimes \Gamma\rightarrow N\left({}^{\vee}{G},{}^{d}{T}\right)\rtimes \Gamma$
has the property that 
\begin{align}\label{eq:ajparameteridentity}
(n_{1}{}^{\vee}\delta)^{2}=\mathbf{n}(w_0w_{{}^{d}{L}})=\prod_{\alpha\in R\left({}^{d}{U},{}^{d}{T}\right)}\alpha^{\vee}(-1).
\end{align} 
Therefore,
$\prod_{\alpha\in R\left({}^{d}{U},{}^{d}T\right)}\alpha^{\vee}(-1)=(-1)^{2\lambda}$
and from Proposition 1.3.5 \cite{Shelstad81} we can conclude
$$\lambda\in X_{\ast}\left(Z\left({}^{d}{L}\right)\right)+\frac{1}{2}\sum_{\alpha\in R\left({}^{d}{U},{}^{d}T\right)}\alpha^{\vee}.$$
Let
\begin{align}\label{eq:AJLparameter}
\mu=\lambda+{}^{\vee}{\rho}_{{}^{d}{L}},\quad \mu'=\mu=-\lambda+{}^{\vee}{\rho}_{{}^{d}{L}}.
\end{align}
Then
\begin{align}\label{eq:higuest}
\mu,\mu'\in X_{\ast}({}^{d}{T})+\frac{1}{2}\sum_{\alpha\in R\left({}^{d}{B},{}^{d}T\right)}\alpha^{\vee}
\end{align}
and
$$\text{Ad}(\psi(j))\cdot\mu=\mu'.$$
The Langlands parameter corresponding to $\psi$ verifies for every $z\in\mathbb{C}^{\times}$
\begin{align}\label{eq:AJLparameter2}
\varphi_{\psi}(z)=z^{\mu}\bar{z}^{\mu'}.
\end{align}
We point out that from point (AJ3) above,
$\mu$ is regular, and from (\ref{eq:higuest}) $\mu$ is the infinitesimal character of a finite-dimensional representation of some real form of $G$ with highest weight (relative to the root system  
$R\left({}^{d}{B},{}^{d}T\right)$) equal to
$$\mu-\frac{1}{2}\sum_{\alpha\in R\left({}^{d}{B},{}^{d}{T}\right)}\alpha^{\vee}=\lambda-\frac{1}{2}\sum_{\alpha\in R\left({}^{d}{U},{}^{d}{T}\right)}\alpha^{\vee}.$$

This concludes the description of the Arthur parameters studied in \cite{Adams-Johnson}. To
define the cohomology packets we need first to connect the Levi subgroup ${}^{d}{L}$ 
of ${}^{\vee}G$ to an extended group for some Levi subgroup $L$ of $G$, and factor $\psi$
through an Arthur parameter $\psi_L$ of some $E$-group for $L$. The second invariant of this $E$-group 
will be equal to $(n_1{}^{\vee}\delta)^{2}$, and since this element is not necessarily one, it is not
always possible to factor $\psi$ through the Arthur parameter of 
an $L$-group for $L$. 
Because of the Local Langlands Correspondence, this will have as a consequence the necessity of using the canonical cover of $L$ (see Definition 10.1 \cite{ABV} and Section 3).

The final step in the construction of the packets is to use Theorem \ref{theo:essentiallyunipotentparameter}
to associate to $\psi_L$ a family of canonical projective characters of strong real forms of $L$ of type $(n_1{}^{\vee}\delta)^{2}$ and apply cohomological induction to them.\\

Let $\delta_{qs}$ be a strong real form of $G^{\Gamma}$,
whose associated real form $\sigma_{\delta_{qs}}$ is quasi-split.  
Write $\theta_{\delta_{qs}}$ for the corresponding Cartan involution
(see Equation (5d)-(5g) \cite{AVParameters}). 
Suppose $T$ is a Cartan subgroup of $G$ stable under $\sigma_{\delta_{qs}}$ and $\theta_{\delta_{qs}}$. 
We have a canonical isomorphism between the based root data
\begin{align}\label{eq:rootdatarelation}
(X^{\ast}({T}),\Delta^{\vee},X_{\ast}({T}),\Delta)\quad\text{and}
\quad \left(X^{\ast}({}^{d}{T}),{}^{d}\Delta,X_{\ast}({}^{d}{T}),{}^{d}\Delta^{\vee}\right).
\end{align}
Using this isomorphism, to the couple $\left({}^{d}{Q},{}^{d}{L}\right)$ we can associate a parabolic group $Q=LU$ 
of $G$ containing $T$ such that ${}^{\vee}L\cong {}^{d}{L}$.
Now, conjugation by the element $n_1{}^{\vee}\delta$ defined in (\ref{eq:n1})
preserves the splitting  $\left({}^{d}{L},{}^{d}{B}_{{}^{d}{L}},{}^{d}{T},\{{X}_{\alpha}\}_{\alpha\in \Delta\left({}^{d}{B}_{{}^{d}{L}},{}^{d}{T}\right)}\right)$, so the isomorphism
${}^{\vee}L\cong {}^{d}{L}$ transfer it to a distinguished involutive automorphism of ${}^{\vee}L$. Write $a_L$ for 
the automorphism of the based root datum of ${}^{\vee}L$ (or equivalently
of the based root datum of $L$) induced in this way by $n_1{}^{\vee}\delta$.  
We define (see Proposition 2.16(c) \cite{ABV})
\begin{align}\label{eq:extendedgroupL}
L^{\Gamma} \text{ to be the extended group of $L$ with invariants }(a_L,\overline{1}).  
\end{align}
Let $a_G$ be the automorphism of the based root datum 
of $G$ induced by ${}^{\vee}\delta$ through duality. Since $n_{1}{}^{\vee}\delta$ 
induces the same automorphism, we have $a_G|_{\Psi_{0}(L)}=a_L$. Consequently, 
each real form in the inner class corresponding to $a_L$, is the restriction 
to $L$ of a real form in the inner class corresponding to $a_G$,  
 that is $L^{\Gamma}\subset G^{\Gamma}$. 
Next, 
 write 
\begin{align}\label{eq:sinvariant7}
z=(n_1{}^{\vee}\delta)^{2}
\end{align}
and define (see Proposition 4.7(c) \cite{ABV}) 
\begin{align}\label{eq:e-groupL}
\left({}^{\vee}L^{\Gamma},\mathcal{S}_L\right) \text{ to be the unique } E\text{-group with invariants }(a_L,z).
\end{align}
More precisely,
let ${}^{\vee}{\sigma}_{L}$ be any distinguished automorphism of ${}^{\vee}L$ 
corresponding to $a_L$. Then the group 
${}^{\vee}L^{\Gamma}$ is defined as the union
of ${}^{\vee}{L}$ and the set ${}^{\vee}{L}^{\vee}\delta_{L}$ of formal symbols $l{}^{\vee}\delta_{L}$, with multiplication defined according to the rules: 
$$(l_1{}^{\vee}\delta_L)(l_2{}^{\vee}\delta_L)=l_1\sigma_{L}(l_2) z,\qquad (l_1{}^{\vee}\delta_L)(l_2)=l_1\sigma_{L}(l_2){}^{\vee}\delta_L$$
and the obvious rules for the other two kinds of product. 

We explain now how to extend the isomorphism ${}^{\vee}{L}\cong{}^{d}{{L}}$ 
to an embedding
\begin{align}\label{eq:embedding1}
\iota_{{L},{G}}:{}^{\vee}L^{\Gamma}\longrightarrow {}^{\vee}{G}^{\Gamma}.
\end{align}
Since ${}^{\vee}L^{\Gamma}$ 
is the disjoint union of ${}^{\vee}L$ and the coset ${}^{\vee}L {}^{\vee}\delta_{L}$
we just need to define 
$\iota_{{L},{G}}$ on ${}^{\vee}L{}^{\vee}\delta_{L}$. 
We do this by sending 
each element
$l{}^{\vee}\delta_L\in {}^{\vee}L {}^{\vee}\delta_{L}$ to: 
\begin{align}\label{eq:embedding2}
\iota_{{L},{G}}(l{}^{\vee}\delta_L)=ln_1{}^{\vee}\delta.
\end{align} 
And because 
${}^{\vee}\delta_L^{2}=z=(n_1{}^{\vee}\delta)^{2}$, 
is straightforward to verify that 
the embedding is well-defined. 

Let us go back now to our Arthur parameter $\psi$  satisfying points
(AJ1), (AJ2) and (AJ3) above.
From the properties satisfied by $\psi$ and the definition of ${}^{\vee}{L}^{\Gamma}$
there exists (up to conjugation) a unique essentially unipotent Arthur parameter
\begin{align*}
\psi_{L}:W_{\mathbb{R}}\times \mathbf{SL}(2,\mathbb{C})\longrightarrow{}^{\vee}{L}^{\Gamma}, 
\end{align*}
with restriction to $\mathbf{SL}(2,\mathbb{C})$ equal to the principal morphism, 
such that up to conjugation by ${}^{\vee}{{G}}$ we have
\begin{align}\label{eq:arthurparameterL0}
\psi:W_{\mathbb{R}}\times \mathbf{SL}(2,\mathbb{C})\xrightarrow{\psi_{{L}}}{}^{\vee}{L}^{\Gamma}  \xrightarrow{\iota_{L,G}}{}^{\vee}{G}^{\Gamma}. 
\end{align}
Indeed, with notation as in (\ref{eq:AJparameterequation0}) and (\ref{eq:AJparameterequation01}), 
we just need to define $\psi_L$ by
\begin{align}\label{eq:arthurparameterL}
\psi_{L}(z)=\psi(z)=z^{\lambda}z^{\lambda'}\quad\text{ and }\quad \psi_L(j)=t{}^{\vee}\delta_L.  
\end{align}
Since $\psi_{{L}}$ is an essentially principal unipotent Arthur parameter, by
Theorem \ref{theo:essentiallyunipotentparameter}(d) there exists
for each strong real form $\delta$ of $L^{\Gamma}$ 
a projective character $\chi(\psi_L,\delta)$ of type $z$
of $L^{can}(\delta,\mathbb{R})$. 
We notice that since is $\psi_{L}|_{\mathbb{C}^{\times}}$ is bounded, the character 
$\chi(\psi_L,\delta)$ is unitary. 

Suppose now that $\delta\in L^{\Gamma}$ is a strong real form of $L^{\Gamma}$ 
so that $\delta^{2}\in Z({}G)$.
Then $\delta$ can be seen as a strong real form of $G^{\Gamma}$.
Write 
\begin{align}\label{eq:realformorbit}
\mathcal{O}_{\delta}=\{g\delta g^{-1}:g\in N_{G}(L)\}
\end{align}
and make $L$ act on $\mathcal{O}_{\delta}$ by conjugation. Denote by $\mathcal{S}_{\delta}$
the set of $L$-conjugacy classes of $\mathcal{O}_{\delta}$. Each $L$-orbit in $\mathcal{O}_{\delta}$
defines an $L$-conjugacy class of strong real forms in $L^{\Gamma}$, belonging to the same $G$-conjugacy class of strong real forms in $G^{\Gamma}$. 
For each $s\in \mathcal{S}_\delta$ fix a strong real form  
$\delta_s\in s$, 
and as above write $\chi(\psi_L,\delta_s)$ 
for the unitary character of type $z$ of ${L}^{can}(\mathbb{R},\delta_s)$
corresponding to $\delta_s$ and $\psi_L$. Now, let $\theta_{\delta_{s}}$ be the Cartan involution corresponding to $\sigma_{{\delta}_{s}}$, then the parabolic subgroup $Q$ of $G$ defined after Equation (\ref{eq:rootdatarelation}) is  $\theta_{\delta_s}$-stable. Write
\begin{itemize}
\item $K_{\delta_s}^{can}$ be the preimage in $G^{can}$ of $K_{\delta_s}=G^{\theta_{\delta_{s}}}$.
\item $\mathfrak{q}=\text{Lie}({Q})$, $\mathfrak{u}=\text{Lie}({U})$ and
$\mathfrak{k}_{\delta_s}=\text{Lie}(K_{\delta_s})$.
\item $S=\dim(\mathfrak{u}\cap\mathfrak{k}_{\delta_s})$. 
\end{itemize}
and consider the cohomologically induced representation 
\begin{align}\label{eq:canonicalinduction}
\left({}^{u}\mathscr{R}_{\mathfrak{q},K_{\delta_s}^{can}\cap L^{can}}^{\mathfrak{g},K_{\delta_s}^{can}}\right)^{S}(\chi(\psi_L,\delta_s)). 
\end{align}
We would like to begin by pointing out that, 
since $\chi(\psi_L,\delta_s)$
is a canonical projective representation of type 
$z$ (see Equation (\ref{eq:ajparameteridentity}) and Equation (\ref{eq:sinvariant7})),  
we are not tensoring 
$\chi(\psi_L,\delta_s)$
with 
$(\bigwedge^{\text{top}}\mathfrak{u})$
when applying the fuctor $\left({}^{u}\mathscr{R}_{\mathfrak{q},K_{\delta_s}^{can}\cap L^{can}}^{\mathfrak{g},K_{\delta_s}^{can}}\right)^{S}$
(cf. Equation (\ref{eq:cohomologicalinduction})). As a consequence the infinitisemal character of $\chi(\psi_L,\delta_s)$ is equal to the infinitesimal character of 
$\left({}^{u}\mathscr{R}_{\mathfrak{q},K_{\delta_s}^{can}\cap L^{can}}^{\mathfrak{g},K_{\delta_s}^{can}}\right)^{S}(\chi(\psi_L,\delta_s))$.
Second we notice that the $L$-parameter associated to $\left({}^{u}\mathscr{R}_{\mathfrak{q},K_{\delta_s}^{can}\cap L^{can}}^{\mathfrak{g},K_{\delta_s}^{can}}\right)^{S}(\chi(\psi_L,\delta_s))$
is the $L$-parameter of an $L$-group
obtained by composing the $L$-parameter corresponding to 
$\chi(\psi_L,\delta_s)$ with the embedding $\iota_{{L},{G}}:{}^{\vee}L^{\Gamma}\rightarrow {}^{\vee}{G}^{\Gamma}$. Therefore, using the Local Langlands Correspondence (see Theorem \ref{theo:4.1}) we conclude that
$\left({}^{u}\mathscr{R}_{\mathfrak{q},K_{\delta_s}^{can}\cap L^{can}}^{\mathfrak{g},K_{\delta_s}^{can}}\right)^{S}(\chi(\psi_L,\delta_s))$ is of type 1 (see Definition (10.3) \cite{ABV} and the paragraph after Equation (\ref{eq:cancover})), that is to say
trivial on $\pi_1(G)^{can}$, and consequently can be seen as a representation of $G(\delta_{s},\mathbb{R})$. 

We can now give the definition of the packets defined by Adams and Johnson in \cite{Adams-Johnson}.

\begin{deftn}\label{deftn:ajpackets}
Let $\psi$ be an Arthur parameter for ${}^{\vee}G^{\Gamma}$ satisfying points
$\mathrm{(AJ1)}$, $\mathrm{(AJ2)}$ and $\mathrm{(AJ3)}$, above.
The Adams-Johnson packet corresponding to $\psi$ is defined as the set
\begin{align*}
\Pi(G/\mathbb{R})_{\psi}^{\mathrm{AJ}}:=\left\{\left({}^{u}\mathscr{R}_{\mathfrak{q},K_{\delta_s}^{can}\cap L^{can}}^{\mathfrak{g},K_{\delta_s}^{can}}\right)^{S}(\chi(\psi_L,\delta_s)) :\delta\in L^{\Gamma}-L
\mathrm{~with~} \delta^{2}\in Z(G)\mathrm{~and~}s\in \mathcal{S}_{\delta}\right\}. 
\end{align*}
\end{deftn}
The next result enumerates some important properties satisfied
by Adams-Johnson packets. 
\begin{theo}\label{ajpproperties}
Let $\psi$ be an Arthur parameter for ${}^{\vee}G^{\Gamma}$ satisfying points
$\mathrm{(AJ1)}$, $\mathrm{(AJ2)}$ and $\mathrm{(AJ3)}$, above.
\begin{enumerate}[i.]
\item $\Pi(G/\mathbb{R})_{\psi}^{\mathrm{AJ}}$ contains the $L$-packet $\Pi(G/\mathbb{R})_{\varphi_{\psi}}$.
\item $\Pi(G/\mathbb{R})_{\psi}^{\mathrm{AJ}}$ is the support of a stable formal virtual character (For a precise description of the stable character see Theorem 2.13 \cite{Adams-Johnson}).
\item $\Pi(G/\mathbb{R})_{\psi}^{\mathrm{AJ}}$ satisfies the ordinary endoscopic identities predicted by the theory of endoscopy (For a more precise statement of this point see Theorem 2.21 \cite{Adams-Johnson}).
\end{enumerate}
\end{theo}
We turn now to the proof that Adams-Johnson packets are ABV-packets. We begin by 
noticing that from 	Theorem \ref{theo:essentiallyunipotentparameter}
we have 
$$\Pi(L/\mathbb{R})_{\psi_L}^{\text{ABV}}:=\left\{\chi(\psi_L,\delta):\delta \text{ a strong real form of }L^{\Gamma} \right\}.$$
Thus, we can give a reformulation of Definition \ref{deftn:ajpackets} 
in terms of the the ABV-packet attached to $\psi_{L}$ as follows 
\begin{align}\label{eq:reformulationAJpackets}
\Pi(G/\mathbb{R})_{\psi}^{\mathrm{AJ}}:=\left\{\left({}^{u}\mathscr{R}_{\mathfrak{q},K_{\delta_s}^{can}\cap L^{can}}^{\mathfrak{g},K_{\delta_s}^{can}}\right)^{S}(\pi) :\pi\in\Pi^{z}(L(\mathbb{R},\delta_s))_{\psi_L}^{\mathrm{ABV}},~ \delta\in L^{\Gamma}-L
\mathrm{~with~} \delta^{2}\in Z(G)\mathrm{~and~}s\in \mathcal{S}_{\delta}\right\} .
\end{align}
The next two results are going to be needed in what follows. The first is Proposition 1.11  \cite{ABV}.
\begin{prop}\label{prop:quasisplitrepresentation}
Suppose $\sigma$ is a quasisplit real form of $G$. Let $\varphi,~\varphi'\in \Phi({}^{\vee}{G}^{\Gamma})$ be two Langlands parameters, with $S$ and $S'$ as the corresponding ${}^{\vee}{{G}}$-orbits. Then the following conditions are equivalent:
\begin{enumerate}[i.]
\item $S$ is contained in the closure of $S'$. 
\item There are irreducible representations $\pi\in \Pi(G(\sigma,\mathbb{R}))_{\varphi}$ and $\pi'\in\Pi(G(\sigma,\mathbb{R}))_{\varphi'}$ with
the property that $\pi'$ is a composition factor of the standard module of which $\pi$ is the unique Langlands quotient.
\end{enumerate} 
\end{prop}
We remark that $(ii)$ implies $(i)$ even for $G(\sigma,\mathbb{R})$ not quasi-split. 
\begin{lem}\label{lem:orbit}
Let $L^{\Gamma}$ be the extended group defined in (\ref{eq:extendedgroupL}), and write $\left({}^{\vee}{L}^{\Gamma},\mathcal{S}_L\right)$ for the $E$-group defined in (\ref{eq:e-groupL}).
In the setting of (\ref{eq:VofgeometricparameterUnion}), let us write
$$\iota_{L,G}:X\left({}^{\vee}{L}^{\Gamma}\right)\rightarrow X\left({}^{\vee}{G}^{\Gamma}\right)$$ 
for the closed immersion induced from the inclusion $\iota_{L,G}:{}^{\vee}{L}^{\Gamma}\hookrightarrow {}^{\vee}{G}^{\Gamma}$ of 
Equation (\ref{eq:embedding1}) (see Corollary 6.21 \cite{ABV}).
Let
 $\varphi$ be a Langlands parameter 
for ${}^{\vee}G^{\Gamma}$ with restriction to 
$\mathbb{C}^{\times}$ given by
$$\varphi(z)=z^{\lambda_{\varphi}}\overline{z}^{\lambda_{\varphi}'},
$$
where $\lambda_{\varphi}$, $\lambda_{\varphi}'\in X_{\ast}({}^{d}{T})\otimes\mathbb{C}$ with 
$\lambda_{\varphi}-\lambda_{\varphi}'\in X_{\ast}({}^{d}{T})$ (cf. Proposition 5.6 \cite{ABV}).
Suppose that $\varphi$ satisfies the following properties
\begin{enumerate}[i.]
\item $\varphi$ factors through ${}^{\vee}{L}^{\Gamma}$, that is, there exists a Langlands parameter 
$\varphi_{L}$
of ${}^{\vee}{L}^{\Gamma}$ such that 
$$\varphi=\iota_{L,G}\circ\varphi_{L}.$$
\item $\left<\lambda_{\varphi},\alpha\right>\neq 0$ for all roots 
$\alpha\in R\left({}^{\vee}{G},{}^{d}{T}\right)$. 
\end{enumerate}
Write $S_{\varphi}$ for the ${}^{\vee}{G}$-orbit 
in $X\left({}^{\vee}{G}^{\Gamma}\right)$ corresponding 
to $\varphi$ under 
(\ref{eq:cllc}) and
$S_{\varphi_{L}}$ for the ${}^{\vee}{L}$-orbit 
in $X\left({}^{\vee}{L}^{\Gamma}\right)$ corresponding under (\ref{eq:cllc}) to $\varphi_{L}$.
Suppose $S$ is a ${}^{\vee}{G}$-orbit 
in $X\left({}^{\vee}{G}^{\Gamma}\right)$ containing $S_{\varphi}$ in its closure.
Then there exists an orbit $S_{L}$ of 
${}^{\vee}{L}$
in $X\left({}^{\vee}{L}^{\Gamma}\right)$ with 
$S_{\varphi_{L}}\subset \overline{S}_{{L}}$  
such that
$$S={}^{\vee}{G}\cdot \iota_{L,G}(S_{L}).$$
\end{lem}
\begin{proof}
Let $\delta\in L^{\Gamma}$ be a strong real form of $L^{\Gamma}$ with $\delta^{2}\in Z(G)$, such that its associated real form $\sigma_{\delta}$ defines a quasi-split real form for $G$. As in the 
paragraph following (\ref{eq:realformorbit}), 
for each $s\in \mathcal{S}_{\delta}$ fix a strong real form $\delta_s\in s$, and write 
$\theta_{\delta_{s}}$ for the Cartan involution corresponding to $\sigma_{{\delta}_{s}}$.
Let $Q$ be the parabolic subgroup of $G$ defined after Equation (\ref{eq:rootdatarelation}).
Then $Q$ is a $\theta_{\delta}$-stable 
parabolic subgroup with Levi decomposition $Q=LU$.   
After possibly conjugating $\varphi$, we may assume
from Lemma \ref{lem:orbit} $(ii.)$ that
$$\left<\lambda_{\varphi},\alpha\right>>0\text{ for all roots }\alpha\in R\left({}^{d}{U},{}^{d}T\right).$$ 
Then by Proposition 4.13 \cite{VoganUnitarizability}
and Theorem 8.2 \cite{Knapp-Vogan}, the cohomological induction functor (notation as in (\ref{eq:canonicalinduction}))
$$\left({}^{u}\mathscr{R}_{\mathfrak{q},K_{\delta_s}^{can}\cap L^{can}}^{\mathfrak{g},K_{\delta_s}^{can}}\right)^{S}(\cdot),$$ 
when restricted to the category of canonical projective representation of type $z$ (see Equation \ref{eq:sinvariant7}) with infinitesimal character 
$\lambda_{\varphi}$, is exact and carries irreducible representations to irreducible representations.

Suppose $S$ is a ${}^{\vee}{G}$-orbit 
in $X\left({}^{\vee}G^{\Gamma}\right)$ containing $S_{\varphi}$ in its closure,
and write $\varphi'$ for the Langlands parameter corresponding to $S$ 
under Proposition 6.17 \cite{ABV} (see also Equation (\ref{eq:cllc})).
From Proposition \ref{prop:quasisplitrepresentation} 
there are irreducible representations $\pi\in \Pi_{\varphi}({G}(\mathbb{R},\delta))$ 
and $\pi'\in \Pi_{\varphi'}({G}(\mathbb{R},\delta))$, with the
property that $\pi'$ is a composition factor of the standard module $M(\pi)$ of which $\pi$ is the unique quotient. 


Now, recall the bijection between complete geometric parameters and final limit characters stated in Theorem 12.9 \cite{ABV} and Proposition 13.12 \cite{ABV}.
Since $\varphi$ factors through ${}^{\vee}{L}^{\Gamma}$ and $\left<\lambda_{\varphi},\alpha\right>>0\text{ for all roots }\alpha\in R\left({}^{d}{U},{}^{d}T\right)$, this bijection together with Lemma 8.1.2 \cite{VoganGreenBook} and Theorem 8.2.4 \cite{VoganGreenBook},
implies the existence of an element $s\in\mathcal{O}_{\delta}$,
a strong real form $\delta_s\in s$, 
and a standard module $M_L$ 
of ${L}^{can}(\mathbb{R},\delta_s)$, such that 
\begin{align}\label{eq:greenbooktheo}
M(\pi)=\left({}^{u}\mathscr{R}_{\mathfrak{q},K_{\delta_s}^{can}\cap L^{can}}^{\mathfrak{g},K_{\delta_s}^{can}}\right)^{S}(M_L).
\end{align}
Let $\pi_{L}$ be the unique irreducible quotient of $M_L$. Using the notation introduced earlier, we denote $M_L=M(\pi_{L})$. Since $\left({}^{u}\mathscr{R}_{\mathfrak{q},K_{\delta_s}^{can}\cap L^{can}}^{\mathfrak{g},K_{\delta_s}^{can}}\right)^{S}$
when restricted to representations with infinitesimal character
$\lambda_{\varphi}$ is exact and carries irreducible representations
to irreducible representations, we deduce from (\ref{eq:greenbooktheo}) that
$$\pi=\left({}^{u}\mathscr{R}_{\mathfrak{q},K_{\delta_s}^{can}\cap L^{can}}^{\mathfrak{g},K_{\delta_s}^{can}}\right)^{S}(\pi_L),$$ 
and $\varphi_{L}$ is the Langlands parameter corresponding to $\pi_{L}$.
%

Furthermore, the exacteness and the preservation by $\left({}^{u}\mathscr{R}_{\mathfrak{q},K_{\delta_s}^{can}\cap L^{can}}^{\mathfrak{g},K_{\delta_s}^{can}}\right)^{S}$ of the irreducibility property,  
implies that every composition factor of $M(\pi)$ 
is obtained after applying cohomological induction
to a composition factor of $M(\pi_{L})$.
Therefore, 
there exists $\pi_{L}'$ an irreducible representation of $L(\mathbb{R},\delta_s)$, 
such that
\begin{align}\label{eq:resolutionL2} 
\pi'=\left({}^{u}\mathscr{R}_{\mathfrak{q},K_{\delta_s}^{can}\cap L^{can}}^{\mathfrak{g},K_{\delta_s}^{can}}\right)^{S}(\pi_{L}').
\end{align}
Let $\varphi_{L}'$ be the Langlands parameter corresponding to $\pi_{L}'$
and write $S_{{L}}$ for the orbit of ${}^{\vee}{L}$ in $X\left({}^{\vee}{L}^{\Gamma}\right)$ associated to $\varphi_L'$
under Proposition 6.17 \cite{ABV}.
Since $\pi_L'$ is a composition factor of 
$M(\pi_{L})$
the remark after Proposition \ref{prop:quasisplitrepresentation}$(ii)$, implies that 
the orbit $S_{\varphi_{L}}$ is contained in the closure of ${S_{{L}}}$. 
Furthermore, from Equation (\ref{eq:resolutionL2}) and 
the relation 
between the data parameterizing $\pi_L'$ and $\pi$ in Lemma 8.1.2 \cite{VoganGreenBook}, 
we can verify that the Langlands parameter $\varphi'$ factors through 
$\varphi_{L}'$. 
Thus by Corollary 6.21 \cite{ABV}
the orbits $S_L$ and $S$ correspond under the map $\iota_{L,G}:X\left({}^{\vee}{L}^{\Gamma}\right)\rightarrow X\left({}^{\vee}{G}^{\Gamma}\right)$, that is to say
$$S={}^{\vee}{G}\cdot \iota_{L,G}(S_{L}).$$
\end{proof}
That Adams-Johnson packets are ABV-packets,
follows as a corollary of the next theorem.
\begin{theo}\label{theo:micropacket}
Let $\varphi$ be a Langlands parameter for ${}^{\vee}G^{\Gamma}$ satisfying Lemma \ref{lem:orbit} $(i.)$-$(ii.)$.
Then
\begin{align*}
\Pi(G/\mathbb{R})_{\varphi}^{{mic}}=\left\{\left({}^{u}\mathscr{R}_{\mathfrak{q},K_{\delta_s}^{can}\cap L^{can}}^{\mathfrak{g},K_{\delta_s}^{can}}\right)^{S}(\pi) :\pi\in\Pi^{z}(L(\mathbb{R},\delta_s))_{\varphi_L}^{{mic}},~ \delta\in L^{\Gamma}-L
\mathrm{~with~} \delta^{2}\in Z(G)\mathrm{~and~}s\in \mathcal{S}_{\delta}\right\} .
\end{align*}
\end{theo}
\begin{proof}
Let $\varphi_{{L}}$ be the unique (up to conjugation) Langlands parameter of ${}^{\vee}{L}^{\Gamma}$ satisfying $\varphi=\iota_{L,G}\circ\varphi_{L}$. From
 the definition of the embedding $\iota_{L,G}:{}^{\vee}{L}^{\Gamma}\hookrightarrow {}^{\vee}{G}^{\Gamma}$ (see Equation (\ref{eq:embedding1})), we deduce 
$$\varphi_{L}(z)=\varphi(z)=z^{\lambda_{\varphi}}z^{\lambda'_{\varphi}},\quad z\in\mathbb{C}^{\times}.$$
Write
$$\mathcal{O}={}^{\vee}{G}\cdot\lambda_{\varphi}\quad\text{ and }\quad \mathcal{O}_{L}={}^{\vee}{L}\cdot\lambda_{\varphi}.$$
Since $\varphi$
satisfies Lemma \ref{lem:orbit}$(ii.)$, 
$\mathcal{O}_{\varphi}$ and $\mathcal{O}_{\varphi_L}$   
are orbits of regular elements in ${}^{\vee}\mathfrak{g}$ and ${}^{\vee}\mathfrak{l}$ respectively.
With notation as in 
Equations (\ref{eq:Vofgeometricparameter})-
(\ref{eq:VofgeometricparameterUnion}) and 
Equation (\ref{eq:completegeometricparameter}),
let $\xi=\left(\left(y,\Lambda\right),\tau\right)$ be a complete geometric parameter for ${}^{\vee}G^{\Gamma}$ 
such that its corresponding orbit ${S_{\xi}}$ of ${}^{\vee}{G}$ in $X\left({}^{\vee}G^{\Gamma}\right)$,
contains $S_{\varphi}$ in its closure. Let $\pi(\xi)$ be the irreducible representation corresponding to $\xi$ under the Local Langlands Correspondence (see Theorem \ref{theo:4.1}) and write 
$\delta_{\xi}$ for the strong real form of $G^{\Gamma}$ of which $\pi(\xi)$ is a representation,
i.e. $\pi(\xi)\in \Pi(G(\mathbb{R},\delta_{\xi}))$.
Let 
$$K:= \text{ set of fixed points of the Cartan involution corresponding to }
{\delta_{\xi}}$$ 
(see Equation (5d)-(5g) \cite{AVParameters}).
By Lemma \ref{lem:orbit} there exists an orbit 
$S_{L}$ of ${}^{\vee}{L}$ in $X\left({}^{\vee}L^{\Gamma}\right)$ satisfying 
\begin{align}\label{eq:orbitsrelation}
S_{\varphi_{L}}\subset \overline{S}_{{L}}\quad \text{and}\quad  S_{\xi}={}^{\vee}{G}\cdot \iota_{L,G} (S_{L}).
\end{align}
Therefore, the Langlands parameter associated to $\xi$
factors through ${}^{\vee}L^{\Gamma}$ and 
$\pi(\xi)$ is cohomologically induced from
an irreducible representation of some real form of $L$. More precisely, there exists a complete geometric parameter $\xi_{L}=((y_{L},\Lambda_{L}),\tau_{L})$ for ${}^{\vee}L^{\Gamma}$ with $S_{\xi_L}=S_{L}$, such that if we write $\pi_{L}(\xi_{L})$ for 
the irreducible canonical projective representation corresponding to $\xi_{L}$ under the Langlands correspondence,
then
\begin{align}\label{eq:LtoG1}
\pi(\xi)=\left({}^{u}\mathscr{R}_{\mathfrak{q},K^{can}\cap L^{can}}^{\mathfrak{g},K^{can}}\right)^{S}(\pi(\xi_{L})).
\end{align}
%
Now, $\Lambda \subset\mathcal{O}$, 
so $\Lambda$ is a canonical flat of regular
elements in ${}^{\vee}\mathfrak{g}$.
Similarly, since $\Lambda_{L} \subset \mathcal{O}_{\varphi_L}$, 
we deduce that $\Lambda_{L}$ 
is a canonical flat of regular 
elements in ${}^{\vee}\mathfrak{l}$.
Consequently, with notation as in 
(\ref{eq:groupslambda}),
we obtain that 
$P(\Lambda)$ (respectively $P(\Lambda_{L})$) is a Borel 
subgroup of ${}^{\vee}{G}(\Lambda)$ (respectively of ${}^{\vee}{L}(\Lambda_{{L}})$). 
Let $X_{y}\left(\mathcal{O},{}^{\vee}G^{\Gamma}\right)$ and $X_{y_{L}}\left(\mathcal{O},{}^{\vee}L^{\Gamma}\right)$ be the smooth subvarieties of $X\left({}^{\vee}G^{\Gamma}\right)$
and $X\left({}^{\vee}L^{\Gamma}\right)$ defined in (\ref{eq:varieties}).
It is clear from the definition of these varieties that   
$$S_{\xi}\subset X_{y}\left(\mathcal{O},{}^{\vee}G^{\Gamma}\right)\quad\text{and}\quad S_{\xi_L}\subset X_{y_{L}}\left(\mathcal{O}_{L},{}^{\vee}L^{\Gamma}\right).$$
Let 
\begin{align}\label{eq:settingcompacte}
{}^{\vee}K\quad=\quad\text{set of fixed points of 
the involutive automorphism defined by }y. 
\end{align}
Define  
${}^{\vee}K_{L}$ similarly.
We notice that, since  $\varphi$ factors through $\varphi_{L}$, we have 
$y=\iota_{L,G}(y_{L})$ (see Equation (\ref{eq:embedding1}) and (\ref{eq:embedding2})) and thus
${}^{\vee}K_{L}={}^{\vee}K\cap {}^{\vee}L$. 
As explained in (\ref{eq:varietiesisomor}) 
the varieties $X_{y}\left(\mathcal{O},{}^{\vee}G^{\Gamma}\right)$ and 
$X_{y_{L}}\left(\mathcal{O}_{L},{}^{\vee}L^{\Gamma}\right)$ satisfy
$$X_{y}\left(\mathcal{O},{}^{\vee}G^{\Gamma}\right)\cong {}^{\vee}{G}\times_{{}^{\vee}K}{}^{\vee}{G}(\Lambda)/P(\Lambda)\quad\text{and}\quad
X_{y_{L}}(\mathcal{O}_{L},{}^{\vee}L^{\Gamma})\cong{}^{\vee}{L}\times_{{}^{\vee}K_{L}}{}^{\vee}{L}(\Lambda_{L})/P(\Lambda_{L}).$$
Thus, denoting the flag varieties of 
${}^{\vee}{G}(\Lambda)$ and ${}^{\vee}{L}(\Lambda_{{L}})$ by $X_{{}^{\vee}{G}}(\Lambda)$ and $X_{{}^{\vee}{L}}(\Lambda_{L})$ respectively, we can write
\begin{align}\label{eq:settingvariety}
X_{y}\left(\mathcal{O},{}^{\vee}G^{\Gamma}\right)\cong{}^{\vee}{G}\times_{{}^{\vee}K}X_{{}^{\vee}{G}}
(\Lambda)
\quad\text{and}\quad
X_{y_{L}}\left(\mathcal{O}_{L},{}^{\vee}L^{\Gamma}\right)\cong{}^{\vee}{L}\times_{{}^{\vee}K_{L}}X_{{}^{\vee}{L}}(\Lambda_{L}).
\end{align}
Define $\Xi\left(X_{{}^{\vee}{G}}(\Lambda)\right)$ to be the set of couples $(S,\mathcal{V})$ with $S$ an orbit of 
${}^{\vee}K$ on $ X_{{}^{\vee}{G}}(\Lambda)$ and $\mathcal{V}$ an irreducible ${}^{\vee}K$-equivariant local system on $S$. Define $\Xi\left(X_{{}^{\vee}{L}}(\Lambda_{L})\right)$ similarly.  
Let 
$\Xi_{y}\left(\mathcal{O},{}^{\vee}G^{\Gamma}\right)$
be the subset of $\Xi\left(\mathcal{O},{}^{\vee}G^{\Gamma}\right)$
consisting of complete geometric parameters with first 
coordinate $y$. We make the same definition for the subset $\Xi_{y_{L}}\left(\mathcal{O}_{L},{}^{\vee}L^{\Gamma}\right)$ of $\Xi\left(\mathcal{O}_{L},{}^{\vee}L^{\Gamma}\right)$.
From Proposition 7.14 \cite{ABV} the map in
Proposition \ref{prop:reduction2}(2) is compatible with the parameterization of irreducibles 
by $\Xi_{y_{L}}\left(\mathcal{O}_{L},{}^{\vee}L^{\Gamma}\right)$ and $\Xi\left(X_{{}^{\vee}L}(\Lambda_{L})\right)$, 
respectively by $\Xi_{y}\left(\mathcal{O},{}^{\vee}G^{\Gamma})\right)$ and $\Xi\left(X_{{}^{\vee}G}(\Lambda)\right)$.  
In other words we have bijections 
\begin{align}\label{eq:bijectioncomparision}
\Xi_{y_{{L}}}\left(\mathcal{O}_{L},{}^{\vee}L^{\Gamma}\right)\longrightarrow \Xi(X_{{}^{\vee}L}(\Lambda_{L})) \quad\text{ and }\quad 
\Xi_{y}\left(\mathcal{O},{}^{\vee}G^{\Gamma}\right)\longrightarrow \Xi(X_{{}^{\vee}G}(\Lambda)).
\end{align}
Let $\xi'$ be the complete geometric parameter of $X_{{}^{\vee}{G}}(\Lambda)$
corresponding to $\xi$ 
 under (\ref{eq:bijectioncomparision})
and write 
$P(\xi')$ 
for the irreducible perverse sheaves on $X_{{}^{\vee}{G}}(\Lambda)$ 
defined from $\xi'$ 
as in (\ref{eq:irredPerv}). For all $\xi_{L}\in \Xi_{y_{L}}\left(\mathcal{O}_{L},{}^{\vee}L^{\Gamma}\right)$, we define $\xi_L'$ and $P(\xi_{L}')$ similarly.  
Following this notation, for each ${}^{\vee}{G}$-orbit $S$ in $X_{y}\left(\mathcal{O},{}^{\vee}G^{\Gamma}\right)$ let 
$S'$ be the ${}^{\vee}K$-orbit in $X_{{}^{\vee}{G}}(\Lambda)$ corresponding to $S$ under (\ref{eq:bijectioncomparision}). We make the same definition for orbits $S_L$ in $X_{y_L}\left(\mathcal{O}_L,{}^{\vee}L^{\Gamma}\right)$.

Let $S_{\varphi}$ be the ${}^{\vee}{G}$-orbit 
in $X\left({}^{\vee}{G}^{\Gamma}\right)$ corresponding to $\varphi$ under 
(\ref{eq:cllc}). We recall that the representation $\pi(\xi)$ belongs to the microlocal packet corresponding to $\varphi$ if and only if its microlocal multiplicity $\chi_{S_{\varphi}}^{\mathrm{mic}}(P(\xi))$ is non zero (Definition \ref{deftn:micropacket}). Now, from Proposition \ref{prop:reduction2}$(1)$ and \ref{prop:reduction2}(6)  we have 
$$S_{\varphi}\subset \overline{S}\quad \text{iff}\quad S_{\varphi}'\subset \overline{S'}\quad \text{ and } \quad
\chi_{S}^{\mathrm{mic}}(P(\xi))=\chi_{S'}^{\mathrm{mic}}(P(\xi')).$$
We have similar relations between ${}^{\vee}L$-orbits in 
$X\left(\mathcal{O}_{{L}},{}^{\vee}L^{\Gamma}\right)$
and ${}^{\vee}K_{{L}}$-orbits in $X_{{}^{\vee}{L}}(\Lambda_{L})$.

Hence to prove the theorem, we are going 
to describe in some detail 
the characteristic cycle of $CC(P(\xi'))$ so that we can give conditions for 
the microlocal multiplicity along $S'_{\varphi}$, 
to be non zero. 
To do this, the next lemma will be necessary. 
In order to simplify the statement of the lemma, we adopt the following notation. Let $H$ be equal to 
${}^{\vee}G$ or ${}^{\vee}L$. In the setting of (\ref{eq:settingcompacte})-(\ref{eq:bijectioncomparision}), for all parameter 
$\xi'\in \Xi\left(X_{{}^{\vee}{H}}(\Lambda)\right)$ we write
\begin{align}\label{eq:RHBBL1}
{}^{\vee}\pi(\xi')
\end{align}
for the irreducible 
$(\mathfrak{h},{}^{\vee}K)$-module corresponding to $P(\xi')$ under the Riemann-Hilbert correspondence and
Beilinson-Bernstein localization, and
\begin{align}\label{eq:RHBBL2}
{}^{\vee}M(\xi')
\end{align}
for the standard module corresponding to the constructible sheaf $(-1)^{d(\xi)}\mu(\xi')$ under the same functor. We recall that $d(\xi)$ is the dimension of the orbit $S_{\xi}$. See the proof of 
Proposition 16.13 \cite{ABV}, for an explanation of the appearence of $(-1)^{d(\xi)}$. 
\begin{lem}\label{lem:lastneededlemma}
 In the setting of (\ref{eq:settingcompacte})-(\ref{eq:bijectioncomparision}).
Let $\xi_{L}\in \Xi_{y_L}\left(\mathcal{O}_{\varphi_L},{}^{\vee}L^{\Gamma}\right)$ 
and write $\xi\in \Xi_{y}\left(\mathcal{O}_{\varphi},{}^{\vee}G^{\Gamma}\right)$
for the unique complete geometric parameter satisfying 
$$
\pi(\xi')=\left({}^{u}\mathscr{R}_{\mathfrak{q},K_{\xi}^{can}\cap L^{can}}^{\mathfrak{g},K_{\xi}^{can}}\right)^{S}(\pi(\xi_{L}')).
$$ 
Then in 
the Grothendieck group of $({}^{\vee}{\mathfrak{g}},{}^{\vee} K)$-modules, 
we have the following identity 
\begin{align}\label{eq:lastneededlemma}
\left({}^{u}\mathscr{R}_{{}^{\vee}\mathfrak{q},{}^{\vee}K\cap {}^{\vee}L}^{{}^{\vee}\mathfrak{g},{}^{\vee}K}\right)^{0}\left(
{}^{\vee}\pi\left(\xi_L'\right)\right)={}^{\vee}\pi(\xi')+\sum_{\substack{\gamma'\in\Xi\left(X_{{}^{\vee}{G}}(\Lambda)\right)\\S_{\gamma'}\nsubseteq {}^{\vee}K\cdot \iota(X_{{}^{\vee}{L}}(\Lambda_{L}))
}}m(\gamma'){}^{\vee}\pi(\gamma'),\quad (\text{for some }m(\gamma')\in\mathbb{N}),
\end{align}
where $\iota$ denotes the inclusion of flag varieties 
$\iota: X_{{}^{\vee}{L}}(\Lambda_{L})\rightarrow X_{{}^{\vee}{G}}(\Lambda)$.
\end{lem}
To not interrupt the development, we postpone the proof of the lemma to the end of the section.\\

Let $I_{{}^{\vee}{L}}^{{}^{\vee}{G}}$ be the geometric induction functor of Definition \ref{deftn:geominduction}. 
Using the Riemman-Hilbert and Beilinson-Bernstein 
correspondences, we deduce from the commutativity of
Diagram (\ref{eq:cdiagramme1}) and Equality (\ref{eq:lastneededlemma}) that 
$$I_{{}^{\vee}{L}}^{{}^{\vee}{G}}P(\xi'_{L})=P(\xi')+\sum_{\substack{\gamma'\in\Xi\left(X_{{}^{\vee}{G}}(\Lambda)\right)\\S_{\gamma'}\nsubseteq {}^{\vee}K\cdot \iota(X_{{}^{\vee}{L}}(\Lambda_{L}))
}}m(\gamma')P(\gamma'),$$
and since 
by Theorem 2.2.3 \cite{Hotta},
the characteristic cycle defines a $\mathbb{Z}$-linear map, we obtain
$$CC\left(I_{{}^{\vee}{L}}^{{}^{\vee}{G}}P(\xi'_{L})\right)=CC\left(P(\xi')\right)+\sum_{\substack{\gamma'\in\Xi\left(X_{{}^{\vee}{G}}(\Lambda)\right)\\
S_{\gamma'}\nsubseteq {}^{\vee}K\cdot \iota(X_{{}^{\vee}{L}}(\Lambda_{L}))
}}m(\gamma')CC\left(P(\gamma')\right).$$
Thus by Theorem \ref{prop:ccLG} we can write
\begin{align}\label{eq:cycleequation}
&CC(P(\xi'))=\sum_{\substack{{}^{\vee}K_{L}-\mathrm{orbits }~S_{{}^{\vee}{L}}~\mathrm{in}\\
X_{{}^{\vee}{L}}(\Lambda_L)}} \chi_{S_{{}^{\vee}{L}}}^{mic}(P(\xi'_{L}))[\overline{T_{{}^{\vee}K\cdot\iota(S_{{}^{\vee}{L}})}^{\ast}X_{{}^{\vee}{G}}(\Lambda)}]\qquad+\\
&
\left(\sum_{\substack{{}^{\vee}K-\text{orbits}~S'\text{~in}\\
\left({}^{\vee}K\cdot \iota(X_{{}^{\vee}{L}}(\Lambda_{L}))\right)^c
~\cap~ \partial\left({}^{\vee}K\cdot \iota(X_{{}^{\vee}{L}}(\Lambda_{L}))\right)}} \chi_{S'}^{mic}\left(I_{{}^{\vee}{L}}^{{}^{\vee}{G}}P(\xi'_{L})\right)[\overline{T_{S'}^{\ast}X_{{}^{\vee}{G}}(\Lambda)}]-\sum_{\substack{\gamma'\in\Xi\left(X_{{}^{\vee}{G}}(\Lambda)\right)\\
S_{\gamma'}\nsubseteq {}^{\vee}K\cdot \iota(X_{{}^{\vee}{L}}(\Lambda_{L}))
}}m(\gamma')CC\left(P(\gamma')\right)\right).\nonumber
\end{align}
Using Equality \ref{eq:cycleequation} we are going to study 
the value of the microlocal multiplicity $\chi_{S'_{\varphi}}^{\mathrm{mic}}(P(\xi'))$.
We begin by noticing that 
for all 
$
\gamma'\in\Xi\left(X_{{}^{\vee}{G}}(\Lambda)\right)$ with $S_{\gamma'}\nsubseteq {}^{\vee}K\cdot \iota(X_{{}^{\vee}{L}}(\Lambda_{L}))$, 
we have 
$$
S_{\gamma'}\neq {}^{\vee}K\cdot\iota(S_{{}^{\vee}{L}})
\quad\text{for all }{}^{\vee}K_{L}-\mathrm{orbits }~S_{{}^{\vee}{L}}.
$$
Hence by Lemma \ref{lem:orbit} we deduce
$$
{}^{\vee}K\cdot\iota(S_{{}^{\vee}{L}})\nsubseteq\overline{S_{\gamma'}},
\quad\text{for all }{}^{\vee}K_{L}-\mathrm{orbits }~S_{{}^{\vee}{L}},
$$
and from Lemma
\ref{lem:lemorbit}($i$), we obtain
\begin{align*}
\chi_{{}^{\vee}K\cdot\iota(S_{{}^{\vee}{L}})}^{mic}(P(\gamma'))
=0,\quad\text{for all }{}^{\vee}K_{L}-\mathrm{orbits }~S_{{}^{\vee}{L}}.
\end{align*}
In particular, since $S'_{\varphi}=K\cdot\iota(S'_{\varphi_{L}})$,
for all $\gamma'\in\Xi\left(X_{{}^{\vee}{G}}(\Lambda)\right)$ with $S_{\gamma'}\nsubseteq {}^{\vee}K\cdot \iota(X_{{}^{\vee}{L}}(\Lambda_{L}))$, we have 
\begin{align}\label{eq:thelastaoflastoflast}
\chi_{S'_{\varphi}}^{mic}(P(\gamma'))
=0.\end{align}
Moreover, $S'_{\varphi}$
 is not contained in $\left({}^{\vee}K\cdot \iota(X_{{}^{\vee}{L}}(\Lambda_{L}))\right)^c
~\cap~ \partial\left({}^{\vee}K\cdot \iota(X_{{}^{\vee}{L}}(\Lambda_{L}))\right)$,
so the conormal bundle
$\overline{T_{S'_{\varphi}}^{\ast}X_{{}^{\vee}G}(\Lambda)}$
 does not appears in the second sum of (\ref{eq:cycleequation}). 
 Thus from Equation (\ref{eq:thelastaoflastoflast}) we conclude
that 
$\overline{T^{\ast}_{S'_{\varphi}}X_{{}^{\vee}G}(\Lambda)}$
can only appears in the first sum of (\ref{eq:cycleequation}). This implies 
\begin{align*}
\chi_{S_{\varphi}'}^{\mathrm{mic}}(P(\xi'))\neq 0\quad \text{if and only if}\quad
\chi_{S_{\varphi_{L}}'}^{\mathrm{mic}}(P(\xi_{L}'))\neq 0.
\end{align*} 
and from Proposition \ref{prop:reduction2}(6) we can write 
\begin{align}\label{eq:cycleLcycleG}
\chi_{S_{\varphi}}^{\mathrm{mic}}(P(\xi))\neq 0\quad \text{if and only if}\quad
\chi_{S_{\varphi_{L}}}^{\mathrm{mic}}(P(\xi_{L}))\neq 0.
\end{align} 
Hence from Lemma
\ref{lem:lemorbit}($i$)
and
the definition of micro-packets, we obtain that
equation (\ref{eq:cycleLcycleG}) translates as: 
\begin{align}\label{eq:last}
\text{Let }\xi\in \Xi\left({}^{\vee}G^{\Gamma}\right)\text{, then }\pi(\xi)\in\Pi(G/\mathbb{R})_{\varphi}^{mic}
\text{ if and only if there exists }\xi_L \in \Xi\left({}^{\vee}L^{\Gamma}\right)\\
\text{ such that } \pi(\xi_{L})\in\Pi^{z}(L/\mathbb{R})_{\varphi_{L}}^{mic}
\text{ and }\pi(\xi)=\left(\mathscr{R}_{\mathfrak{q},K^{can}\cap L^{can}}^{\mathfrak{g},K^{can}}\right)&^{S}(\pi(\xi_{L})).\nonumber 
\end{align}
Therefore
\begin{align*}
\Pi(G/\mathbb{R})_{\varphi}^{mic}:=\left\{\left({}^{u}\mathscr{R}_{\mathfrak{q},K_{\delta_s}^{can}\cap L^{can}}^{\mathfrak{g},K_{\delta_s}^{can}}\right)^{S}(\pi) :\pi\in\Pi^{z}(L(\mathbb{R},\delta_s))_{\varphi_L}^{mic},~ \delta\in L^{\Gamma}-L
\mathrm{~with~} \delta^{2}\in Z(G)\mathrm{~and~}s\in \mathcal{S}_{\delta}\right\} .
\end{align*}
\end{proof}

\begin{cor}\label{theo:ABV-AJ}
Let $\psi$ be an Arthur parameter for ${}^{\vee}G^{\Gamma}$ satisfying points
$\mathrm{(AJ1)}$, $\mathrm{(AJ2)}$ and $\mathrm{(AJ3)}$, above.  Then
$$\Pi(G/\mathbb{R})_{\psi}^{\mathrm{ABV}}=\Pi(G/\mathbb{R})_{\psi}^{\mathrm{AJ}}.$$  
\end{cor}
\begin{proof}
Let $\varphi_{\psi}$ be the Langlands parameter attached to $\psi$. By $\mathrm{(AJ1)}$ and 
Equation (\ref{eq:arthurparameterL0}), $\varphi_{\psi}$ satisfies points $(i)$-$(ii.)$ Theorem (\ref{theo:micropacket}). Hence
\begin{align*}
\Pi(G/\mathbb{R})_{\psi}^{\mathrm{ABV}}&=
\Pi(G/\mathbb{R})_{\varphi_{\psi}}^{mic}\\
&=\left\{\left({}^{u}\mathscr{R}_{\mathfrak{q},K_{\delta_s}^{can}\cap L^{can}}^{\mathfrak{g},K_{\delta_s}^{can}}\right)^{S}(\pi) :\pi\in\Pi^{z}(L(\mathbb{R},\delta_s))_{\varphi_{\psi_{L}}}^{mic},~ \delta\in L^{\Gamma}-L
\mathrm{~with~} \delta^{2}\in Z(G)\mathrm{~and~}s\in \mathcal{S}_{\delta}\right\}\\
&=\left\{\left({}^{u}\mathscr{R}_{\mathfrak{q},K_{\delta_s}^{can}\cap L^{can}}^{\mathfrak{g},K_{\delta_s}^{can}}\right)^{S}(\pi) :\pi\in\Pi^{z}(L(\mathbb{R},\delta_s))_{\psi_L}^{\mathrm{ABV}},~ \delta\in L^{\Gamma}-L
\mathrm{~with~} \delta^{2}\in Z(G)\mathrm{~and~}s\in \mathcal{S}_{\delta}\right\},
\end{align*} 
and by Equation (\ref{eq:reformulationAJpackets}) we deduce
$$\Pi(G/\mathbb{R})_{\psi}^{\mathrm{ABV}}=\Pi(G/\mathbb{R})_{\psi}^{\mathrm{AJ}}.$$
\end{proof} 
We end this section with the proof of Lemma \ref{lem:lastneededlemma}.
\begin{proof}
 In the setting of (\ref{eq:settingcompacte})-(\ref{eq:bijectioncomparision}). Since the functor
$$\left({}^{u}\mathscr{R}_{\mathfrak{q},K_{\delta_s}^{can}\cap L^{can}}^{\mathfrak{g},K_{\delta_s}^{can}}\right)^{S}(\cdot)$$ 
when restricted to the category of canonical projective representation of 
infinitesimal character 
$\lambda_{\varphi}$, carries irreducible representations to irreducible representations,
there exists for all parameters $\mu_L\in\Xi_{y_L}\left(\mathcal{O}_L,{}^{\vee}L^{\Gamma}\right)$
a complete parameter $\mu_{L,G}\in \Xi_{y}\left(\mathcal{O},{}^{\vee}G^{\Gamma}\right)$
satisfying
\begin{align}\label{eq:1finallemma}
\pi(\mu_{L,G})=
\left({}^{u}\mathscr{R}_{\mathfrak{q},K\cap L}^{\mathfrak{g},K_{\delta_s}^{can}}\right)^{S}(
\pi(\mu_L)
).
\end{align}
To prove the lemma, we need to verify
the following list of assertions:
\begin{enumerate}
\item For all parameters $\mu_{L}\in\Xi_{y_L}\left(\mathcal{O}_L,{}^{\vee}L^{\Gamma}\right)$
we have
$$M(\mu_{L,G})=\left({}^{u}\mathscr{R}_{\mathfrak{q},K\cap L}^{\mathfrak{g},K_{\delta_s}^{can}}\right)^{S}(
M(\mu_L)).$$
\item For all parameters $\mu\in \Xi_{y}\left(\mathcal{O},{}^{\vee}G^{\Gamma}\right)$ with $
S_{\mu}={}^{\vee}K\cdot\iota(S_{{}^{\vee}{L}})$
for some ${}^{\vee}K_{\xi_L}$-$\mathrm{orbit }~S_{{}^{\vee}{L}},
$
there exists a parameter $\mu_{L}\in\Xi_{y_L}\left(\mathcal{O}_L,{}^{\vee}L^{\Gamma}\right)$
such that 
$$\pi(\mu)=\left({}^{u}\mathscr{R}_{\mathfrak{q},K\cap L}^{\mathfrak{g},K_{\delta_s}^{can}}\right)^{S}(
\pi(\mu_L)).$$ 
\end{enumerate}
To state the last three assertions, we recall that for all
$\xi\in \Xi_{y}\left(\mathcal{O},{}^{\vee}G^{\Gamma}\right)$
we write $\xi'$ for the complete geometric parameter in 
$\Xi\left(X_{{}^{\vee}{G}}(\Lambda)\right)$
corresponding to $\xi$ under (\ref{eq:bijectioncomparision}). We use the same notation for $\xi_{L}\in\Xi_{y}\left(\mathcal{O}_L,{}^{\vee}L^{\Gamma}\right)$. 
Then with notation as in equation (\ref{eq:RHBBL1}) and equation (\ref{eq:RHBBL2}) we have: 
\begin{enumerate}
 \setcounter{enumi}{2}
\item For all parameters $\mu_{L}'\in \Xi\left(X_{{}^{\vee}{L}}(\Lambda_{L})\right)$
we can write
$${}^{\vee}M(\mu_{L,G}')=\left({}^{u}\mathscr{R}_{{}^{\vee}\mathfrak{q},{}^{\vee}K\cap {}^{\vee}L}^{{}^{\vee}\mathfrak{g},{}^{\vee}K}\right)^{0}\left(
{}^{\vee}M(\mu_L')\right).$$ 
\item For all parameters $\mu_{L}'\in \Xi\left(X_{{}^{\vee}{L}}(\Lambda_{L})\right)$,
${}^{\vee}\pi(\mu_{L,G}')$ appears as a subquotient in 
$\left({}^{u}\mathscr{R}_{{}^{\vee}\mathfrak{q},{}^{\vee}K\cap {}^{\vee}L}^{{}^{\vee}\mathfrak{g},{}^{\vee}K}\right)^{0}(
{}^{\vee}\pi(\mu_L')
)$ with multiplicity one.
\item For all couples of parameters
$\gamma_L',~\mu_L'\in \Xi\left(X_{{}^{\vee}L}(\Lambda)\right)$, 
the multiplicity 
${}^{\vee}m(\mu_L',\gamma_L')$
of ${}^{\vee}\pi(\mu_L')$
in ${}^{\vee}M(\gamma_L')$,
is equal to the multiplicity 
${}^{\vee}m(\mu_{L,G}',\gamma_{L,G}')$ of ${}^{\vee}\pi(\mu_{L,G}')$
in ${}^{\vee}M(\gamma_{L,G}')$.
\end{enumerate}

To prove points (1)-(3) recall the
 bijection between complete geometric parameters and final limit characters
expressed in Theorem 12.9 \cite{ABV} and Proposition 13.12 \cite{ABV}. 
Since we 
are in the setting of (\ref{eq:settingcompacte})-(\ref{eq:bijectioncomparision})
(i.e regular and dominant infinitesimal character), through this bijection,
points (1.) and (3.) are a consequence   
of Corollary 8.1.3 \cite{VoganGreenBook} and 
Theorem 8.2.4 \cite{VoganGreenBook}. 
For Point (2.) we notice that, by hypothesis the Langlands parameter corresponding 
to $S_{\mu}$ 
factors through ${}^{\vee}L^{\Gamma}$. Then the argument used previous to equation (\ref{eq:greenbooktheo})
applies to the present context, and point (2.) follows.

To show Point (4.) we recall that 
${}^{\vee}\pi(\mu_{L}')$
is the unique irreducible quotient  
of ${}^{\vee}M(\mu_{L}')$. The same can be said for
${}^{\vee}\pi(\mu_{L,G}')$
and ${}^{\vee}M(\mu_{L,G}')$. Point (3.) and the exactness of the parabolic induction functor implies then that 
$\left({}^{u}\mathscr{R}_{{}^{\vee}\mathfrak{q},{}^{\vee}K\cap {}^{\vee}L}^{{}^{\vee}\mathfrak{g},{}^{\vee}K}\right)^{0}(
{}^{\vee}\pi(\mu_L')
)$ is a quotient of ${}^{\vee}M(\mu_{L,G}')$. Therefore, ${}^{\vee}\pi(\mu_{L,G}')$ is a quotient of $\left({}^{u}\mathscr{R}_{{}^{\vee}\mathfrak{q},{}^{\vee}K\cap {}^{\vee}L}^{{}^{\vee}\mathfrak{g},{}^{\vee}K}\right)^{0}(
{}^{\vee}\pi(\mu_L')
)$. 
Since ${}^{\vee}\pi(\mu_{L,G}')$ appears with multiplicity one in ${}^{\vee}M(\mu_{L,G}')$, Point (4.) follows.


We are going to reduce point (5.) to an identity
of \textbf{character matrices} in the Gronthendieck group of $(\mathfrak{g},K)$-modules.
By Theorem 13.13(c) \cite{ICIV} (see also Corollary 15.13 \cite{ABV} ), we have
$$
{}^{\vee}m(\mu_L',\gamma_L')=(-1)^{l(\mu_{L})-l(\gamma_{L})}c(\gamma_L,\mu_{L})\quad\text{and}\quad {}^{\vee}m(\mu_{L,G}',\gamma_{L,G}')=(-1)^{l(\mu_{L,G})-l(\gamma_{L,G})}c(\gamma_{L,G},\mu_{L,G}),
$$ 
where the length fonction $l(\cdot)$ is defined in (12.1) \cite{ICIV},
and $c$ is the \textbf{representation-theoretic
character matrix}, i.e. the matrix whose entries give, 
for $\mu_{L,G}\in\Xi\left(\mathcal{O},{}^{\vee}G^{\Gamma}\right)$,
the decomposition of $\pi(\mu_{L,G})$ into standard representations  
\begin{align}\label{eq:finallemmastandardecomposition}
\pi(\mu_{L,G})=\sum_{\gamma\in \Xi\left(\mathcal{O},{}^{\vee}G^{\Gamma}\right)}
c(\gamma,\mu_{L,G})M(\gamma).
\end{align}
Similarly, for $\mu_{L}\in \Xi\left(\mathcal{O}_L,{}^{\vee}L^{\Gamma}\right)$ we have 
\begin{align}\label{eq:finallemmastandardecomposition2}
\pi(\mu_L)=\sum_{\gamma_L\in \Xi\left(\mathcal{O}_L,{}^{\vee}L^{\Gamma}\right)}
c(\gamma_L,\mu_L)M(\xi_L).
\end{align}
Point (5.) is then equivalent to the equality 
$$
c(\gamma_{L,G},\mu_{L,G})=c(\gamma_L,\mu_L).
$$ 
Now, the decomposition in (\ref{eq:finallemmastandardecomposition2}) and Equality (\ref{eq:1finallemma}) imply that
\begin{align*}
\pi(\mu_{L,G})&=\left({}^{u}\mathscr{R}_{\mathfrak{q},K\cap L}^{\mathfrak{g},K^{can}}\right)^{S}(\pi(\mu_L))\\
&=\left({}^{u}
\mathscr{R}_{\mathfrak{q},K\cap L}^{\mathfrak{g},K^{can}}\right)^{S}\left(\sum_{\gamma_L\in \Xi\left(\mathcal{O}_L,{}^{\vee}L^{\Gamma}\right)}c(\gamma_L,\mu_L)(M(\gamma_L))\right)\\
&=\sum_{\gamma_L\in \Xi\left(\mathcal{O}_L,{}^{\vee}L^{\Gamma}\right)}c(\gamma_L,\mu_L)\left({}^{u}
\mathscr{R}_{\mathfrak{q},K\cap L}^{\mathfrak{g},K^{can}}\right)^{S}(M(\gamma_L)).
\end{align*}
Then by Point (1.) we deduce
\begin{align*}
\pi(\mu_{L,G})=\sum_{\gamma_L\in \Xi\left(\mathcal{O}_L,{}^{\vee}L^{\Gamma}\right)}
c(\gamma_L,\mu_L)M(\gamma_{L,G}) 
\end{align*} 
and by comparing with Equation (\ref{eq:finallemmastandardecomposition}), we conclude
$$
c(\gamma_{L,G},\mu_{L,G}))=c(\gamma_L,\mu_L).
$$
Now that points (1.)-(5.) have been verified,
we can end with the proof of Lemma \ref{lem:lastneededlemma}. 
Since we are only going to be working with parameters in the varieties
$X_{{}^{\vee}{L}}(\Lambda_{L})$ and $X_{{}^{\vee}{G}}(\Lambda)$, 
we omit the prime in the notation of the parameters.
Suppose $\xi_{L}\in \Xi\left(X_{{}^{\vee}L}(\Lambda)\right)$.
We begin by decomposing ${}^{\vee}M(\xi_{L,G})$ into two sums of irreducible representations
\begin{align*}
{}^{\vee}M(\xi_{L,G})&=
\sum_{\substack{\mu\in\Xi\left(X_{{}^{\vee}{G}}(\Lambda)\right)\\
S_{\mu}\subseteq {}^{\vee}K\cdot \iota(X_{{}^{\vee}{L}}(\Lambda_{L}))
}}
{}^{\vee}m(\mu,\xi_{L,G}){}^{\vee}\pi(\mu)+
\sum_{\substack{\gamma\in\Xi\left(X_{{}^{\vee}{G}}(\Lambda)\right)\\
S_{\gamma}\nsubseteq {}^{\vee}K\cdot \iota(X_{{}^{\vee}{L}}(\Lambda_{L}))
}}
{}^{\vee}m(\gamma,\xi_{L,G}){}^{\vee}\pi(\gamma)
\end{align*}
By Point (2.) the parameters in the first sum are in bijection with the set
$\Xi\left(X_{{}^{\vee}{L}}(\Lambda_L)\right)$, so  
\begin{align}\label{eq:equality1final}
{}^{\vee}M(\xi_{L,G})&=
\sum_{\mu_L\in\Xi\left(X_{{}^{\vee}{L}}(\Lambda_L)\right)}
{}^{\vee}m(\mu_{L,G},\xi_{L,G}){}^{\vee}\pi(\mu_{L,G})+
\sum_{\substack{\gamma\in\Xi\left(X_{{}^{\vee}{G}}(\Lambda)\right)\\
S_{\gamma}\nsubseteq {}^{\vee}K\cdot \iota(X_{{}^{\vee}{L}}(\Lambda_{L}))
}}
{}^{\vee}m(\gamma,\xi_{L,G}){}^{\vee}\pi(\gamma)
\end{align}
From Point (3.), ${}^{\vee}M(\xi_{L,G})$ can also be decomposed as
\begin{align}\label{eq:equality2final}
{}^{\vee}M(\xi_{L,G})&=\left({}^{u}\mathscr{R}_{{}^{\vee}\mathfrak{q},{}^{\vee}K\cap {}^{\vee}L}^{{}^{\vee}\mathfrak{g},{}^{\vee}K}\right)^{0}(
{}^{\vee}M(\xi_L))\nonumber\\
&=\left({}^{u}\mathscr{R}_{{}^{\vee}\mathfrak{q},{}^{\vee}K\cap {}^{\vee}L}^{{}^{\vee}\mathfrak{g},{}^{\vee}K}\right)^{0}\left(\sum_{\mu_L\in\Xi(X_{{}^{\vee}L}(\Lambda_L))}
{}^{\vee}m(\mu_L,\xi_L){}^{\vee}\pi(\mu_L)\right)\nonumber\\
&=\left({}^{u}\mathscr{R}_{{}^{\vee}\mathfrak{q},{}^{\vee}K\cap {}^{\vee}L}^{{}^{\vee}\mathfrak{g},{}^{\vee}K}\right)^{0}\left(\sum_{\mu_L\in\Xi(X_{{}^{\vee}L}(\Lambda_L))}
{}^{\vee}m(\mu_{L,G},\xi_{L,G}){}^{\vee}\pi(\mu_L)\right)\quad\qquad\qquad\quad(\text{By point }(5.))\nonumber\\
&=\sum_{\mu_L\in\Xi(X_{{}^{\vee}L}(\Lambda_L))}
{}^{\vee}m(\mu_{L,G},\xi_{L,G})\left({}^{u}\mathscr{R}_{{}^{\vee}\mathfrak{q},{}^{\vee}K\cap {}^{\vee}L}^{{}^{\vee}\mathfrak{g},{}^{\vee}K}\right)^{0}({}^{\vee}\pi(\mu_L)).
\end{align}
For all $\mu_{L}\in \Xi\left(X_{{}^{\vee}{L}}(\Lambda_L)\right)$, set $
\left({}^{u}\mathscr{R}_{{}^{\vee}\mathfrak{q},{}^{\vee}K\cap {}^{\vee}L}^{{}^{\vee}\mathfrak{g},{}^{\vee}K}\right)^{0}({}^{\vee}\pi(\mu_L)):=\mathscr{R}({}^{\vee}\pi(\mu_{L}))$.
Since by Point (4.), for any parameter $\mu_{L}\in \Xi\left(X_{{}^{\vee}{L}}(\Lambda_{L})\right)$ the module
${}^{\vee}\pi(\mu_{L,G})$ appears in 
$\mathscr{R}(
{}^{\vee}\pi(\mu_L)
)$ with multiplicity one, we can write 
\begin{align}\label{eq:equality3final}
\mathscr{R}({}^{\vee}\pi(\mu_L))=
{}^{\vee}\pi(\mu_{L,G})+\sum_{
\substack{\eta\in\Xi(X_{{}^{\vee}G}(\Lambda))\\
\eta\neq\mu_{L,G}\qquad
}}\text{mult}\left({}^{\vee}\pi(\eta),\mathscr{R}({}^{\vee}\pi(\mu_{L}))\right) {}^{\vee}\pi(\eta).
\end{align}
Hence by Equality (\ref{eq:equality2final}) 
\begin{align}\label{eq:equality4final}
{}^{\vee}M(\xi_{L,G})=&\sum_{\mu_L\in\Xi(X_{{}^{\vee}L}(\Lambda_L))}
{}^{\vee}m(\mu_{L,G},\xi_{L,G})\left(
{}^{\vee}\pi(\mu_{L,G})\quad+\sum_{
\substack{\eta\in\Xi(X_{{}^{\vee}G}(\Lambda))\\
\eta\neq\mu_{L,G}\qquad
}} \text{mult}\left({}^{\vee}\pi(\eta),\mathscr{R}({}^{\vee}\pi(\mu_{L}))\right) {}^{\vee}\pi(\eta)
\right)\nonumber\\
=&\sum_{\mu_L\in\Xi(X_{{}^{\vee}L}(\Lambda_L))}
{}^{\vee}m(\mu_{L,G},\xi_{L,G})
{}^{\vee}\pi(\mu_{L,G})\quad+\nonumber\\
&\quad\sum_{\substack{(\mu_L,\eta)\in\Xi(X_{{}^{\vee}L}(\Lambda_L))\times\Xi(X_{{}^{\vee}G}(\Lambda))
\\
\eta\neq\mu_{L,G}\qquad\qquad\qquad}}{}^{\vee}m(\mu_{L,G},\xi_{L,G})
\text{mult}\left({}^{\vee}\pi(\eta),\mathscr{R}({}^{\vee}\pi(\mu_{L}))\right) {}^{\vee}\pi(\eta).
\end{align}
Now, for $\xi_{L}$ we have ${}^{\vee}m(\xi_{L,G},\xi_{L,G})=1$. Then since the first sum in (\ref{eq:equality1final}) is equal to the first sum in the second equality of (\ref{eq:equality4final}), 
we deduce 
that the integers $\text{mult}\left({}^{\vee}\pi(\eta),\mathscr{R}({}^{\vee}\pi(\xi_{L}))\right) $ in (\ref{eq:equality3final}), satisfy
$$ \text{mult}\left({}^{\vee}\pi(\eta),\mathscr{R}({}^{\vee}\pi(\xi_{L}))\right) =0\quad\text{ for all }\eta\text{ with }\quad S_{\eta}\subset {}^{\vee}K\cdot \iota(X_{{}^{\vee}{L}}(\Lambda_{L})).$$
Therefore
$$
\mathscr{R}\left({}^{\vee}\pi(\xi_L)\right)=
{}^{\vee}\pi(\xi_{L,G})+\sum_{
\substack{\eta\in\Xi(X_{{}^{\vee}G}(\Lambda))\\
S_{\eta}\nsubseteq {}^{\vee}K\cdot \iota(X_{{}^{\vee}{L}}(\Lambda_{L}))
}}\text{mult}\left({}^{\vee}\pi(\eta),\mathscr{R}({}^{\vee}\pi(\xi_{L}))\right) {}^{\vee}\pi(\eta).
$$
and the lemma follows. 
\end{proof}
\bibliographystyle{alpha}
\bibliography{reference}
\noindent\rule{2cm}{0.4pt}\\
~\\
Nicol\'as Arancibia Robert, CY Cergy Paris Université - Carleton University $\bullet$ \textit{E-mail:} \textbf{nicolas.arancibia-robert@cyu.fr}
\end{document}